\documentclass[11pt,reqno]{amsart}
\usepackage{amsmath,amssymb,amsthm}
\usepackage[margin=2.5cm]{geometry}
\usepackage[all]{xy}
\usepackage{microtype}
\usepackage[hidelinks,hypertexnames=false]{hyperref}
\hypersetup{pdftitle={Gopakumar-Vafa invariants and Macdonald formula II},pdfauthor={Lutian Zhao}}

\numberwithin{equation}{section}

\theoremstyle{plain}
\newtheorem{thm}{Theorem}[section]
\newtheorem{prop}[thm]{Proposition}
\newtheorem{lem}[thm]{Lemma}
\newtheorem{cor}[thm]{Corollary}
\newtheorem{defi}[thm]{Definition}
\newtheorem{rmk}[thm]{Remark}

\newtheorem*{lem*}{Lemma}

\newcommand{\PP}{\mathbb{P}}
\newcommand{\QQ}{\mathbb{Q}}
\newcommand{\CC}{\mathbb{C}}
\newcommand{\ZZ}{\mathbb{Z}}
\newcommand{\AAA}{\mathbb{A}}
\newcommand{\LL}{\mathbb{L}}
\newcommand{\Gm}{\mathbb{G}_m}
\newcommand{\cO}{\mathcal{O}}
\newcommand{\cC}{\mathcal{C}}
\newcommand{\cA}{\mathcal{A}}
\newcommand{\cS}{\mathcal{S}}
\newcommand{\cG}{\mathcal{G}}
\newcommand{\cR}{\mathcal{R}}
\newcommand{\cM}{\mathcal{M}}
\newcommand{\cP}{\mathcal{P}}
\newcommand{\cZ}{\mathcal{Z}}
\newcommand{\cH}{\mathcal{H}}
\newcommand{\cK}{\mathcal{K}}
\newcommand{\cL}{\mathcal{L}}
\newcommand{\cQ}{\mathcal{Q}}
\newcommand{\cB}{\mathcal{B}}
\newcommand{\cV}{\mathcal{V}}
\newcommand{\cD}{\mathcal{D}}
\newcommand{\cT}{\mathcal{T}}
\newcommand{\bM}{\mathbf{M}}
\newcommand{\bY}{\mathbf{Y}}
\newcommand{\bZ}{\mathbf{Z}}
\newcommand{\bU}{\mathbf{U}}
\newcommand{\fU}{\mathfrak{U}}
\newcommand{\fW}{\mathfrak{W}}
\newcommand{\fZ}{\mathfrak{Z}}
\newcommand{\fS}{\mathfrak{S}}
\newcommand{\fE}{\mathfrak{E}}
\newcommand{\fH}{\mathfrak{H}}
\newcommand{\fX}{\mathfrak{X}}
\newcommand{\fY}{\mathfrak{Y}}
\newcommand{\fg}{\mathfrak{g}}
\newcommand{\fc}{\mathfrak{c}}

\newcommand{\dR}{\mathbf{R}}
\newcommand{\pH}[1]{{}^{p}\mathcal{H}^{#1}}
\newcommand{\IC}{\mathrm{IC}}
\newcommand{\MHM}{\mathrm{MHM}}
\newcommand{\Chow}{\mathrm{Chow}}
\newcommand{\Hilb}{\mathrm{Hilb}}
\newcommand{\Sym}{\mathrm{Sym}}
\newcommand{\Tot}{\mathrm{Tot}}
\newcommand{\Pic}{\mathrm{Pic}}
\newcommand{\Coh}{\mathrm{Coh}}
\newcommand{\Spec}{\mathrm{Spec}}
\newcommand{\Filt}{\mathrm{Filt}}
\newcommand{\Grad}{\mathrm{Grad}}
\newcommand{\Hom}{\mathrm{Hom}}
\newcommand{\Ext}{\mathrm{Ext}}
\newcommand{\End}{\mathrm{End}}
\newcommand{\cHom}{\mathcal{H}om}
\newcommand{\cExt}{\mathcal{E}xt}
\newcommand{\Exp}{\mathrm{Exp}}
\newcommand{\PT}{\mathrm{PT}}
\newcommand{\ev}{\mathrm{ev}}
\newcommand{\gr}{\mathrm{gr}}
\newcommand{\fib}{\mathrm{fib}}
\newcommand{\cofib}{\mathrm{cofib}}
\newcommand{\Cone}{\mathrm{Cone}}
\newcommand{\tr}{\mathrm{tr}}
\newcommand{\Tr}{\mathrm{Tr}}
\newcommand{\id}{\mathrm{id}}
\newcommand{\res}{\mathrm{res}}
\newcommand{\rat}{\mathrm{rat}}
\newcommand{\coker}{\mathrm{coker}}
\newcommand{\Img}{\mathrm{Im}}
\newcommand{\Crit}{\mathrm{Crit}}
\newcommand{\red}{\mathrm{red}}
\newcommand{\sss}{\mathrm{ss}}
\newcommand{\sd}{\mathrm{sd}}
\newcommand{\sh}{\mathrm{sh}}
\newcommand{\vdim}{\mathrm{vdim}}
\newcommand{\Fr}{\mathrm{Fr}}
\newcommand{\sgn}{\mathrm{sgn}}
\newcommand{\Rhom}{\dR\!\Hom}
\newcommand{\Rgam}{\dR\Gamma}
\newcommand{\wt}[1]{\widetilde{#1}}
\newcommand{\ol}[1]{\overline{#1}}

\title[GV invariants and Macdonald formula II]{Gopakumar-Vafa invariants and Macdonald formula II}
\author{Lutian Zhao}
\address{Kavli Institute for the Physics and Mathematics of the Universe, University of Tokyo, 5-1-5 Kashiwanoha, Kashiwa, 277-8583, Japan.}
\email{lutian.zhao@ipmu.jp}
\subjclass[2020]{14N35, 14D23, 14F10}
\keywords{Gopakumar-Vafa invariants, stable pairs, vanishing cycles, perverse sheaves, Macdonald formula}

\begin{document}
\begin{abstract}
We prove the cohomological Gopakumar-Vafa/Pandharipande-Thomas correspondence for the local plane and quadric in every effective curve class and Euler characteristic. The strict supports of the simple perverse constituents of the stable pair vanishing cycle direct images are closures of loci of transverse unions of smooth connected curves. Wall-crossing and duality leave a finite range of Euler characteristics to consider. In this range, a commutation relation for point modifications on bounded stacks of surface pairs excludes all other supports. Applying the family Macdonald formula on the reduced locus gives the correspondence as an identity of semisimplified perverse direct images on the Chow variety.
\end{abstract}
\maketitle

\section{Introduction}\label{sec:intro}

The Macdonald formula~\cite{Mac62} describes the cohomology of the symmetric products of a smooth curve. Its extensions to families of planar curves relate the cohomology of relative Hilbert schemes to that of compactified Jacobians~\cite{MS13,MY14,MSV21}. On a local surface, this connects stable pairs with sheaf-theoretic Gopakumar-Vafa invariants. Simple perverse constituents supported entirely on non-reduced cycles are invisible in the reduced-curve formula. A support theorem for the stable pair direct image is therefore needed to extend the correspondence to all one-cycles.

For the local plane and quadric, we determine the possible strict supports and show that each meets the reduced locus. The reduced-curve calculation then gives the cohomological Gopakumar--Vafa/Pandharipande--Thomas (GV/PT) correspondence of~\cite{Zha26} in all curve classes.

\subsection{Stable pair direct images and their supports}\label{intro:main}
Set
\begin{align}\label{intro:X}
S=\PP^2\ \text{ or }\ S=\mathbb{F}_0:=\PP^1\times\PP^1,\qquad
p\colon X:=\Tot(K_S)\to S.
\end{align}
Every compact curve in $X$ lies set-theoretically in the zero section; see Lemma~\ref{lem:lower}. We therefore identify the Chow variety of one-cycles of class $\beta\in H_2(S,\ZZ)$ with the complete linear system $B_\beta:=|\cO_S(\beta)|$. For an effective class $\beta$, write
\begin{align}\label{intro:numerics}
D_\beta:=\dim B_\beta,\qquad
g_\beta:=1+\frac{\beta\cdot(\beta+K_S)}{2},\qquad
d:=L\cdot\beta,
\end{align}
where $L=\cO_{\PP^2}(1)$ or $L=\cO_{\mathbb{F}_0}(1,1)$. Note that $g_\beta$ may be negative on $\mathbb{F}_0$. Let $j_\beta\colon B_\beta^{\red}\hookrightarrow B_\beta$ denote the open locus of reduced divisors, including reducible and disconnected ones.

A stable pair on $X$ consists of a compactly supported pure one dimensional sheaf $F$ and a section $s\colon\cO_X\to F$ with zero dimensional cokernel. Let $P_\chi(X,\beta)$ be its moduli space for $[F]=\beta$ and $\chi(F)=\chi$. The Hilbert-Chow map
\begin{align}\label{intro:PTmap}
\Pi_{\chi,\beta}\colon P_\chi(X,\beta)\to B_\beta
\end{align}
sends a pair to the fundamental cycle of its sheaf. We use the perverse sheaf of vanishing cycles $\phi^{\PT}_{\chi,\beta}$ with the orientation induced by Toda's shifted cotangent description~\cite{Tod24}. Appendix~\ref{sec:orientation} identifies this orientation with the determinant orientation on $X$. Put
\begin{align}\label{intro:K}
K_{\chi,\beta}:=\dR\Pi_{\chi,\beta*}\phi^{\PT}_{\chi,\beta}\in D^b_c(B_\beta).
\end{align}
The map $\Pi_{\chi,\beta}$ is proper by Lemma~\ref{lem:proper}. For a bounded constructible complex $K$, write
\begin{align}\label{intro:ss}
K^{\sss}:=\bigoplus_{i\in\ZZ}\bigl(\pH{i}K\bigr)^{\sss}[-i],
\end{align}
where $\pH{i}$ denotes perverse cohomology and the superscript on the right denotes Jordan-H\"older semisimplification. For an irreducible closed subvariety $Z$ and an irreducible local system $W$ on a dense smooth open subset of $Z$, the simple perverse sheaf $\IC_Z(W)$ is obtained by intermediate extension of $W[\dim Z]$ and has strict support $Z$.

Let $\Gamma$ be the set of effective classes containing a smooth connected member:
\begin{align}\label{intro:Gamma}
\Gamma_{\PP^2}=\{dH : d\ge 1\},\qquad
\Gamma_{\mathbb{F}_0}=\{(a,b) : a,b>0\}\cup\{(1,0),(0,1)\}.
\end{align}
Here $H$ is the class of a line and $(a,b)$ is the class of $\cO_{\mathbb{F}_0}(a,b)$. We write $\gamma>0$ for a non-zero effective class. A class partition $\lambda=(m_\gamma)_{\gamma>0}$ of $\beta$, written $\lambda\vdash\beta$, consists of non-negative integers satisfying $\sum_\gamma m_\gamma\gamma=\beta$. If every class appearing in $\lambda$ belongs to $\Gamma$, let $\Sigma_\lambda$ be the closure of the locus of unions of distinct smooth connected curves with these component classes, meeting pairwise transversely with no triple points; see Subsection~\ref{sub:strata}.

\begin{thm}\label{intro:thm-support}
\textup{(Theorem~\ref{thm:support})} For $S=\PP^2$ and $S=\mathbb{F}_0$, every simple constituent of every perverse cohomology sheaf $\pH{i}K_{\chi,\beta}$ has strict support $\Sigma_\lambda$ for some partition $\lambda$ of $\beta$ into classes in $\Gamma$. In particular every such support meets $B_\beta^{\red}$, and
\begin{align}\label{intro:NR}
K^{\sss}_{\chi,\beta}\cong {}^{p}j_{\beta!*}\bigl(j_\beta^{*}K_{\chi,\beta}\bigr)^{\sss},
\end{align}
where intermediate extension is applied in each perverse degree. The assertion on strict supports holds in every generic chamber of Toda's stability conditions.
\end{thm}

Equation~\eqref{intro:NR} is condition $(\mathrm{NR})$ of~\cite{Zha26}. Together with the restriction to the supports $\Sigma_\lambda$, it proves the two support assertions left open there.

\subsection{The GV/PT correspondence}\label{intro:corr}\label{intro:history}
Gopakumar and Vafa~\cite{GV98} conjectured that the Gromov-Witten series of a Calabi-Yau 3-fold is determined by integer curve counts $n_{g,\beta}\in\ZZ$. Katz-Klemm-Vafa~\cite{KKV99} proposed a refinement indexed by two representations of $\mathfrak{sl}_2$. Following Hosono-Saito-Takahashi~\cite{HST01}, Katz~\cite{Kat08} and Kiem-Li~\cite{KL12}, Maulik-Toda~\cite{MT18} defined $n_{g,\beta}$ using vanishing cycles. Let $\mathrm{Sh}_\beta(X)$ be the moduli space of stable one dimensional sheaves $E$ with $[E]=\beta$ and $\chi(E)=1$, and let
\begin{align}\label{intro:HC}
\pi\colon \mathrm{Sh}_\beta(X)\to \Chow_\beta(X)
\end{align}
be its Hilbert-Chow map. With $\phi_{\mathcal{S}h}$ its perverse sheaf of vanishing cycles with Calabi-Yau orientation, their definition is
\begin{align}\label{intro:MTdef}
\sum_{i\in\ZZ}\chi\bigl(\pH{i}(\dR\pi_*\phi_{\mathcal{S}h})\bigr)y^i
=\sum_{g\ge 0}n_{g,\beta}\bigl(y^{\frac12}+y^{-\frac12}\bigr)^{2g}.
\end{align}
The perverse filtration and the Lefschetz action on hypercohomology give the two Lefschetz gradings.

Maulik-Toda conjectured the agreement of these invariants with the stable pair invariants of Pandharipande-Thomas~\cite{PT09,PT10}, both globally and over each point of the Chow variety~\cite[Conjectures~3.13 and~3.14]{MT18}. For irreducible one-cycles on local surfaces, they proved the correspondence~\cite[Theorem~1.3]{MT18} using the planar-curve Macdonald formulas of Migliorini-Shende~\cite{MS13} and Maulik-Yun~\cite{MY14}. The family support theorem of Migliorini-Shende-Viviani~\cite{MSV21} supplies the reduced-curve calculation used in~\cite{Zha26} and here. For a reduced locally planar curve, Shende~\cite{She12} expresses the multiplicities of the Severi strata in its versal deformation in terms of Euler characteristics of Hilbert schemes of points. Shen~\cite{Shen25} surveys the perverse filtration on compactified Jacobians of integral locally planar curves and its relation to their Hilbert schemes.

Choi-Katz-Klemm~\cite{CKK14} computed refined stable pair invariants for the local plane and quadric in low degrees and tested a conjectural refined Bogomolny--Prasad--Sommerfield (BPS) product formula. Choi-van Garrel-Katz-Takahashi~\cite{CGKT20} used stable pair wall-crossing to calculate the cohomology of one dimensional sheaf moduli on del Pezzo surfaces in classes of low arithmetic genus. For local curves, Kinjo-Koseki~\cite{KK26} proved $\chi$-independence of GV invariants; they also established $\chi$-independence of BPS cohomology for Higgs bundles.

For ample curve classes on a toric del Pezzo surface, Maulik-Shen~\cite{MS23} proved that the intersection cohomology of the semistable sheaf moduli, with its perverse and Hodge filtrations, is independent of the Euler characteristic. On $\PP^2$, Pi-Shen~\cite{PS23} determined tautological generators for the cohomology ring when the degree and Euler characteristic are coprime. Kononov-Pi-Shen~\cite{KPS23} proposed the equality $P=C$ between the perverse and Chern filtrations in low cohomological degrees. They also verified the agreement of the perverse-sheaf and refined stable pair BPS invariants in degrees three and four.

For $\beta\in\Gamma$, let $U_\beta\subset B_\beta$ be the locus of smooth connected curves, with universal curve $\rho_\beta\colon\cC_\beta|_{U_\beta}\to U_\beta$, and set $V_\beta:=R^1\rho_{\beta*}\QQ$. For a local system $W$ on $U_\beta$, the notation $\IC_{B_\beta}(W)$ denotes the intermediate extension of $W[D_\beta]$.

Let $q$ be a formal variable recording the Euler characteristic.

\begin{defi}\label{intro:def-A}
For $\beta\in\Gamma$, define the formal series
\begin{align}\label{intro:A}
\cA_\beta(q):=\bigoplus_{j=0}^{2g_\beta}\ \bigoplus_{u,v\ge 0}
q^{\,j+u+v+1-g_\beta}\,\IC_{B_\beta}\bigl(\wedge^{j}V_\beta\bigr)[g_\beta-1-j-2v].
\end{align}
For an effective class $\beta\notin\Gamma$, set $\cA_\beta(q):=0$.
\end{defi}

The indices $j,u,v$ record the contributions of $H^1,H^0,H^2$ in the Macdonald formula~\cite{Mac62}. Proposition~\ref{prop:primitive} identifies this series with the cohomological GV data from $\dR\pi_*\IC$ for stable sheaves of Euler characteristic one. For a ruling class $f=(1,0)$ on $\mathbb{F}_0$, one has $g_f=0$ and $V_f=0$, hence
\begin{align}\label{intro:ruling-A}
\cA_f(q)=\bigoplus_{u,v\ge 0}q^{u+v+1}\QQ_{\PP^1}[-2v],\qquad \cA_{mf}(q)=0\ (m>1).
\end{align}
Let $Q^\beta$ record the curve class. To express the contribution of several components, use the addition maps
\begin{align}\label{intro:mu}
\mu_\lambda\colon \prod_{\gamma>0}B_\gamma^{m_\gamma}\to B_\beta,\qquad
G_\lambda:=\prod_{\gamma>0}\mathfrak{S}_{m_\gamma}.
\end{align}
The map $\mu_\lambda$ sends a tuple of divisors to their sum, and $G_\lambda$ permutes the divisors of each class.

\begin{defi}\label{intro:def-Exp}
The symmetric-power series for addition of cycles is
\begin{align}\label{intro:Exp}
\Exp_{\Chow}\Bigl(\bigoplus_{\gamma>0}\cA_\gamma(q)Q^{\gamma}\Bigr)
:=\QQ_{B_0}\oplus\bigoplus_{\beta>0}Q^{\beta}\bigoplus_{\lambda\vdash\beta}
\Bigl[\dR\mu_{\lambda*}\Bigl(\boxtimes_{\gamma>0}\cA_\gamma(q)^{\boxtimes m_\gamma}\Bigr)\Bigr]^{G_\lambda},
\end{align}
where $B_0:=\Spec\CC$. The permutation actions use the Koszul signs of the cohomological shifts, and the formal variables $q,Q$ are even.
\end{defi}

For complexes $A$ on $B_\gamma$ and $B$ on $B_\delta$, convolution means $\dR\mu_*(A\boxtimes B)$, where $\mu$ adds the two divisors. Each $(\chi,\beta)$-coefficient is a finite sum: a fixed class has finitely many partitions, and every $\cA_\gamma(q)$ is bounded below in $q$. Define the stable pair series by
\begin{align}\label{intro:ZPT}
\cZ_{\PT}(q,Q):=\QQ_{B_0}\oplus\bigoplus_{\beta>0,\ \chi\in\ZZ}
q^{\chi}K^{\sss}_{\chi,\beta}[-\chi]\,Q^{\beta}.
\end{align}

\begin{thm}\label{intro:thm-gvpt}
\textup{(Theorem~\ref{thm:gvpt})} For $S=\PP^2$ and $S=\mathbb{F}_0$, there is an isomorphism
\begin{align}\label{intro:gvpt}
\cZ_{\PT}(q,Q)\cong \Exp_{\Chow}\Bigl(\bigoplus_{\beta>0}\cA_\beta(q)Q^{\beta}\Bigr),
\end{align}
i.e. each coefficient is an isomorphism of finite direct sums of shifts of semisimple perverse sheaves on $B_\beta$.
\end{thm}

Over $B_\beta^{\red}$, the identity, with its permutation actions for repeated component classes, is Theorem~\ref{thm:reduced}. By Theorem~\ref{intro:thm-support} and~\cite[Theorem~1.1]{Zha26}, both sides are sums of intersection complexes on the supports $\Sigma_\lambda$. We extend the identity to $B_\beta$ by intermediate extension.

Taking hypercohomology and Euler characteristics gives the numerical correspondence~\cite[Conjecture~3.13]{MT18}; taking stalks first gives its local version~\cite[Conjecture~3.14]{MT18}. The two Lefschetz gradings give the refined product formula in~\eqref{eq:refined-PTGV}.

\subsection{Proof of the support theorem}\label{intro:strategy}
Toda~\cite{Tod10} used wall-crossing to prove rationality of generating series of ordinary Euler characteristics of stable pair moduli. Bridgeland~\cite{Bri11} subsequently proved the numerical Donaldson--Thomas/Pandharipande--Thomas (DT/PT) correspondence and rationality of the associated generating functions using Hall algebras.

The proof is by induction on $d=L\cdot\beta$. Toda's wall-crossing~\cite{Tod24} and hyperbolic localization~\cite{KPS26,Des25} express a change of chamber as convolutions of a rank one factor of smaller degree with semistable sheaf factors. We apply the induction hypothesis to the former and cohomological integrality~\cite{BDINP25} together with Maulik-Shen's support theorem~\cite{MS23} to the latter; ruling classes require the calculation in Subsection~\ref{sub:ruling}. The construction over $B_\beta$ is given in Proposition~\ref{prop:wall}.

Toda's duality exchanges $\chi$ and $-\chi$. Since $P_\chi(X,\beta)$ is empty for $\chi<1-g_\beta$, chamber independence and duality prove the support assertion outside the range
\begin{align}\label{intro:window}
1-g_\beta\le\chi\le g_\beta-1.
\end{align}
For $g_\beta\le0$ the range is empty. Otherwise, we relate consecutive Euler characteristics by point modifications on stacks of surface pairs $\cO_S\to F$, imposing a lower bound on the Harder-Narasimhan (HN) slopes of $F$. Kinjo's dimensional reduction~\cite{Kin22} identifies the shifted cotangent vanishing cycle direct images with direct images of dualizing complexes. The comparison with the PT complexes also involves direct images from unstable HN strata, whose rank one factors have lower degree (Theorem~\ref{thm:bounded}).

Choose a pencil in $|L|$ with base locus $Z$, and let $B_{\beta,Z}$ be the locus of curves avoiding $Z$. In the Verdier quotient by complexes whose perverse constituents have strict supports $\Sigma_\lambda\cap B_{\beta,Z}$, let $V_\chi$ be the image of the direct image from the stack of surface pairs with a lower slope bound, normalized as in~\eqref{eq:K-chi-c}. By Lemma~\ref{lem:slices}, it represents the PT direct image up to normalization, and changes of the slope bound are invertible.

For a smooth divisor $D\in|rL|$ with $r>d$, adding a point and removing a point on $D$ give operators $e_\chi\colon V_\chi\to V_{\chi+1}$ and $f_\chi\colon V_{\chi+1}\to V_\chi$. Section~\ref{sec:hecke} maps the two compositions to the same relative Quot stack. The difference is supported where the modifications cancel. It factors through an endomorphism of the dualizing complex of the pair stack, which connectedness makes a scalar. Counting on smooth curves gives
\begin{align}\label{intro:commutator}
f_\chi e_\chi-e_{\chi-1}f_{\chi-1}=(D\cdot\beta)\,\id_{V_\chi}.
\end{align}
The opens $B_{\beta,Z}$ cover $B_\beta$ as the pencil varies.

The finite range~\eqref{intro:window} forces all $V_\chi$ to vanish: on the lowest non-zero object, the commutator would give $f^ke^k=k!(D\cdot\beta)^k\id$, whereas $e^k=0$ for large $k$ (Lemma~\ref{lem:weyl}). The operators are related to those of Rennemo~\cite{Ren18} and Kivinen~\cite{Kiv19} for a fixed curve and to Toda's $K$-theoretic Hecke relation~\cite{Tod20}; here the relation acts on dualizing complexes over the Chow variety.

\subsection{Conventions}\label{intro:notation}
All varieties, schemes and stacks are defined over $\CC$. We follow~\cite{Zha26} for the Chow varieties and Macdonald series, Toda~\cite{Tod24} for stable pairs and stability parameters, and Kinjo~\cite{Kin22} for shifted cotangent stacks and vanishing cycles. We use
\begin{align}\label{intro:n}
n:=\chi+g_\beta-1.
\end{align}
For a stable pair supported on a divisor on $S$, this is the length of the zero dimensional cokernel.

Boldface letters such as $\bM,\bY$ denote derived stacks, and $t_0$ denotes classical truncation. Fraktur letters denote open substacks and correspondence stacks, specified as classical or derived when introduced. Fiber products of derived stacks are derived. For an algebraic group $G$, $BG$ denotes its classifying stack. We write $\LL_{\fX}$ for the cotangent complex and $T^*[-1]\fX$ for the shifted cotangent stack. In a pair complex $[\cO\to F]$, the terms lie in degrees zero and one; our cone convention is $\fib(f)=\Cone(f)[-1]$ and $\cofib(f)=\Cone(f)$. The zero section of $p\colon X\to S$ is denoted by $i\colon S\hookrightarrow X$.

Constructible sheaves have $\QQ$-coefficients. We use the perverse t-structure of~\cite{BBD82} and mixed Hodge modules~\cite{Sai90}. The notation $M[k]$ denotes a cohomological shift and $M(m)$ a Tate twist; following~\cite{Kha19}, set $M\langle r\rangle:=M[2r](r)$. The constant Hodge complex $\QQ_Y^H$ has rational realization $\QQ_Y$. In generating series, $\LL=[\QQ(-1)]$ is the Tate class and $\LL^{1/2}$ is a formal square root. We write $\omega_Y$ for the dualizing complex and normalize intersection complexes to be perverse. All direct images, exceptional direct images, pull-backs and exceptional pull-backs are derived, denoted by $\dR f_*,\dR f_!,f^*,f^!$.

The complex $K^t_{\chi,\beta}$ is the Chow direct image in Toda's chamber $t$; $K_{\chi,\beta}$ is the stable pair case. The notation $K_\chi(c)$ in Section~\ref{sec:hecke} instead records a lower slope bound $c$ for surface pairs.

\medskip \noindent\textbf{Acknowledgments.} The author thanks Sheldon Katz for his guidence. The author also thanks Yukinobu Toda for discussion. This work was supported by JSPS KAKENHI Grant Number JP25K17226 and by World Premier International Research Center Initiative (WPI), MEXT, Japan.

\medskip
\noindent\textbf{Use of AI tools.}
The author used ChatGPT to assist with drafting and revising the exposition and organizing parts of the manuscript. It was also used to discuss mathematical arguments and suggest points for further verification. The author takes full responsibility for the final text, including the mathematical arguments and references.

\section{Stable pairs and sheaves on the two surfaces}\label{sec:prelim}
On the surfaces~\eqref{intro:X}, every effective linear system is base point free and any two effective classes have non-negative intersection. The intersection inequality bounds the Euler characteristic of a stable pair from below.

\subsection{Curve classes and stable pairs}\label{sub:classes}
Let $L=\cO_{\PP^2}(1)$ on the plane and $L=\cO_{\mathbb{F}_0}(1,1)$ on the quadric. In both cases the canonical bundle is a negative power of $L$, and we write
\begin{align}\label{eq:kappa}
K_S=L^{-\kappa},\qquad \kappa=3\ (S=\PP^2),\quad \kappa=2\ (S=\mathbb{F}_0).
\end{align}
The numbers in~\eqref{intro:numerics} have the following values.
\begin{center}
\renewcommand{\arraystretch}{1.25}
\begin{tabular}{c|c|c|c|c}
$S$ & $\beta$ & $d$ & $D_\beta$ & $g_\beta$\\ \hline
$\PP^2$ & $dH$, $d\ge 1$ & $d$ & $d(d+3)/2$ & $(d-1)(d-2)/2$\\
$\mathbb{F}_0$ & $(a,b)$, $a,b\ge 0$ & $a+b$ & $ab+a+b$ & $(a-1)(b-1)$
\end{tabular}
\end{center}
The zero class is excluded unless stated otherwise. We write $\gamma\le\beta$ if $\beta-\gamma$ is effective, and $\gamma<\beta$ if moreover $\gamma\neq\beta$. The arithmetic genus on $\mathbb{F}_0$ can be negative: a general divisor in class $(m,0)$ is a disjoint union of $m$ fibers, and $g_{(m,0)}=1-m$. The following identities hold on both surfaces:
\begin{align}\label{eq:common-numerics}
1-g_\beta=\frac{\kappa d-\beta^2}{2},\qquad
D_\beta=\frac{\beta^2+\kappa d}{2},\qquad
D_\beta+g_\beta-1=\beta^2.
\end{align}

Let $\cC_\beta\to B_\beta$ be the universal divisor.

\begin{lem}\label{lem:lower}
If $P_\chi(X,\beta)\neq\emptyset$, then $\chi\ge 1-g_\beta$. More precisely, there exist effective classes $\beta_k$ and non-negative integers $\delta_k$, $\ell$ such that $\sum_k\beta_k=\beta$ and
\begin{align}\label{eq:zero-section-estimate}
\chi-(1-g_\beta)=\sum_{i<j}\beta_i\cdot\beta_j+\kappa\sum_{k\ge 1}k\,(L\cdot\beta_k)+\sum_k\delta_k+\ell.
\end{align}
\end{lem}
\begin{proof}
Following~\cite[Lemma~2.3]{Zha26}, let $(F,s)$ be a stable pair, let $\cO_C$ be the image of $s$, and let $\ell$ be the length of the cokernel of $s$. Every compact curve in $X$ is set-theoretically contained in the zero section. Hence if $I$ is the ideal of the zero section and $J$ that of $C$, the filtration $J+I^k$ of $\cO_C$ is finite. Its layers
\begin{align*}
E_k:=(J+I^k)/(J+I^{k+1})
\end{align*}
are quotients of $I^k/I^{k+1}\cong K_S^{-k}$, and comparison with the reflexive hull of the kernel gives exact sequences
\begin{align*}
0\to T_k\to E_k\to \cO_{C_k}\otimes K_S^{-k}\to 0,
\end{align*}
where $T_k$ is zero dimensional of length $\delta_k$ and $C_k$ is an effective divisor of class $\beta_k$, possibly zero. By adjunction, $\chi(E_k)=-\beta_k\cdot(\beta_k+K_S)/2-kK_S\cdot\beta_k+\delta_k$. Summing over $k$, adding $\ell$, and using $\beta^2-\sum_k\beta_k^2=2\sum_{i<j}\beta_i\cdot\beta_j$, we obtain~\eqref{eq:zero-section-estimate}. The right hand side is non-negative: effective classes have non-negative intersection, $-K_S=L^{\kappa}$ is ample, and $\delta_k$ and $\ell$ are lengths.
\end{proof}

\begin{lem}\label{lem:proper}
The moduli space $P_\chi(X,\beta)$ is projective, and $\Pi_{\chi,\beta}$ is proper.
\end{lem}
\begin{proof}
Let $\ol{X}:=\PP_S(\cO_S\oplus K_S)$ be the projective compactification, with zero section $i_0\colon S\hookrightarrow\ol{X}$ and infinity section $S_\infty$, so that $X=\ol{X}\setminus S_\infty$ and
\begin{align*}
\cO_{\ol{X}}(S_\infty)|_{S_\infty}\cong K_S^{-1}.
\end{align*}
Observe that an integral curve not contained in $S_\infty$ has non-negative intersection with this effective divisor; the intersection is zero precisely when the curve is disjoint from it. A curve contained in $S_\infty$ has strictly positive intersection, since $-K_S$ is ample. Thus $S_\infty$ has non-negative intersection with every effective curve. The class $i_{0*}\beta$ has intersection zero with $S_\infty$, so every effective cycle of this class is disjoint from $S_\infty$.

Consider the projective moduli space $P_\chi(\ol{X},i_{0*}\beta)$ of stable pairs on $\ol{X}$~\cite{PT09}. The locus on which the universal sheaf is supported away from $S_\infty$ is open: its complement is the proper image of the intersection of the universal support with $S_\infty$. Every geometric point belongs to this open subset by the intersection calculation above and purity, so the open subset is the whole moduli space as a scheme. Restriction to $X$ and extension by zero are inverse on families over any base, giving the identification with $P_\chi(X,\beta)$. Hence $P_\chi(X,\beta)$ is projective and its morphism to $B_\beta$ is proper.
\end{proof}

\subsection{Slope bounds}\label{sub:slope}
For a pure one dimensional sheaf $F$ on $S$, let $[F]$ be its fundamental class and set
\begin{align*}
\mu(F):=\frac{\chi(F)}{L\cdot[F]}.
\end{align*}
We write $\mu_{\min}(F)$ and $\mu_{\max}(F)$ for the smallest and the largest slopes in the HN filtration of $F$. For a fixed class $\beta$, put
\begin{align}\label{eq:alpha}
\alpha:=\frac{1-g_\beta}{d}
=\begin{cases}(3-d)/2, & S=\PP^2,\\ 1-ab/(a+b), & S=\mathbb{F}_0,\ \beta=(a,b).\end{cases}
\end{align}
\begin{lem}\label{lem:slope}
Every pure quotient of a stable pair sheaf in class $\gamma\le\beta$ has slope at least $\alpha$. Its push-forward to $S$ satisfies $\mu_{\min}\ge\alpha$.
\end{lem}
\begin{proof}
A pure quotient of a stable pair sheaf inherits a section from the original pair, so it is again a stable pair sheaf, in some class $\delta$. By Lemma~\ref{lem:lower} its slope is at least $(1-g_\delta)/(L\cdot\delta)$. On the plane this equals $(3-L\cdot\delta)/2\ge\alpha$, since $L\cdot\delta\le d$. On the quadric, write $\delta=(c,e)$; the function $ce/(c+e)$ is increasing in each non-negative coordinate, so $1-ce/(c+e)\ge 1-ab/(a+b)=\alpha$.

To pass to the surface, recall that a compactly supported sheaf on $X$ is the same as a Higgs sheaf $(E,\theta\colon E\to E\otimes K_S)$ on $S$, with $E$ the push-forward. Since $K_S=L^{-\kappa}$, twisting by $K_S$ lowers every slope by $\kappa$. For the maximal HN term $E_1\subset E$ we have $\mu_{\max}((E/E_1)\otimes K_S)=\mu_{\max}(E/E_1)-\kappa<\mu(E_1)$, so the composite $E_1\to(E/E_1)\otimes K_S$ vanishes. Applying the same argument to the quotient, the whole HN filtration of $E$ is $\theta$-invariant. Its last graded piece is therefore a quotient of $F$ on $X$, and the first assertion gives it slope at least $\alpha$.
\end{proof}

\subsection{Vanishing cycles and orientations}\label{sub:PTsheaf}
Vanishing cycles on a Calabi-Yau moduli space depend on a choice of orientation, namely a square root of the virtual canonical line bundle together with an isomorphism of its square with the virtual canonical bundle~\cite{Joy15, MT18}. The gluing theorem for perverse sheaves on oriented $d$-critical loci is~\cite[Theorem~6.9]{BBDJS15}; its mixed Hodge module form is discussed in~\cite[Section~6.4]{BBDJS15}. For $P_\chi(X,\beta)$ we use the \emph{cotangent orientation} supplied by Toda's shifted cotangent description~\cite{Tod24} of the moduli of rank one objects on $X$. Theorem~\ref{thm:oriented} identifies it with the Calabi-Yau orientation given by the determinant line bundle.

The symbol $\phi^{\PT}_{\chi,\beta}$ denotes the rational vanishing cycle sheaf, and we use the same symbol for its monodromic mixed Hodge lift whenever supports are discussed. Set
\begin{align}\label{eq:v}
v:=\beta^2+\chi,
\end{align}
let $\phi^{\sd}_{\chi,\beta}$ be the Verdier self-dual monodromic vanishing cycle mixed Hodge module for the cotangent orientation, and take
\begin{align}\label{eq:PT-hodge}
\phi^{\PT}_{\chi,\beta}:=\phi^{\sd}_{\chi,\beta}(-v/2)
\end{align}
for the mixed Hodge lift, with the half Tate twist convention of~\cite[Section~5.2.6]{BDINP25}. Over $\Pi_{\chi,\beta}^{-1}(B_\beta^{\red})$, which is smooth by Proposition~\ref{prop:hilb}, the module~\eqref{eq:PT-hodge} is $\QQ^H[D_\beta+n]$, with rational realization $\QQ[\beta^2+\chi]$. Proposition~\ref{prop:DR} translates this convention into the normalization of Kinjo's dimensional reduction.

\subsection{Sheaves in a ruling class}\label{sub:ruling}
Maulik-Shen~\cite{MS23} treat ample classes. The non-ample effective classes on the quadric are the multiples of its two rulings. Set $S=\mathbb{F}_0$ and $f=(1,0)$, so that the members of $|f|$ are the fibers of the first projection $p_1\colon S\to\PP^1$, and we let $p_2\colon S\to\PP^1$ be the second projection. For a class $\gamma$, let $M_\gamma$ be the moduli space of $L$-stable one dimensional sheaves on $S$ of class $\gamma$ and Euler characteristic one, and let $h_\gamma\colon M_\gamma\to B_\gamma$ be the Hilbert-Chow map.

\begin{prop}\label{prop:ruling}
The stack of $L$-semistable pure sheaves of class $mf$ and Euler characteristic $\eta$ is empty unless $m\mid\eta$. If $\eta=m(k+1)$, it is equivalent to the stack of length $m$ sheaves on $\PP^1$, via
\begin{align}\label{eq:ruling-equivalence}
T\mapsto p_1^*T\otimes p_2^*\cO_{\PP^1}(k).
\end{align}
Its stable objects have $m=1$. Consequently
\begin{align*}
M_f=\PP^1,\qquad M_{mf}=\emptyset\ (m>1),\qquad
\dR h_{f*}\IC_{M_f}=\QQ_{\PP^1}[1].
\end{align*}
The same statements hold for the other ruling.
\end{prop}
\begin{proof}
Let $F$ be stable of class $mf$. Sheaves with disjoint supports split as direct sums, so the support of $F$ is a thickening of a single fiber. On an affine neighborhood of the corresponding point of $\PP^1$, the coordinate function acts on $F$ by an endomorphism; stability gives $\End(F)=\CC$, so this action is a scalar, and $F$ is scheme-theoretically supported on the reduced fiber. By purity, $F$ is a vector bundle on that fiber, which is a $\PP^1$, and a stable vector bundle on $\PP^1$ has rank one. Thus $F$ is the push-forward of $\cO_{\PP^1}(k)$ from that fiber for some $k$.

Now let $F$ be semistable and take a Jordan-H\"older filtration. Each factor is a line bundle on a fiber. Equality of slopes forces all factors to have the same Euler characteristic $k+1$, so $\eta=m(k+1)$; conversely, any iterated extension of such factors is semistable. After twisting by $p_2^*\cO(-k)$, all factors become structure sheaves of fibers. Note that $R^1p_{1*}$ vanishes and the adjunction map $p_1^*p_{1*}F\to F$ is an isomorphism on each factor. The associated exact sequences give both statements for iterated extensions, proving~\eqref{eq:ruling-equivalence} on objects. Cohomology and base change give the same equivalence for families, hence for the stacks. Specializing to Euler characteristic one gives the last assertions.
\end{proof}

\begin{rmk}\label{rmk:ruling-deform}
Every sheaf in~\eqref{eq:ruling-equivalence} can be deformed to one supported on $m$ distinct fibers. Indeed, locally a torsion sheaf on a smooth curve is a finite dimensional vector space with one endomorphism; perturbing each Jordan block to a matrix with distinct eigenvalues, independently near each of the finitely many support points, gives the deformation. The eigenvalues can be chosen to avoid any prescribed finite set of fibers.
\end{rmk}

\section{Supports under wall-crossing}\label{sec:supports}
Wall-crossing leads to convolutions along addition of divisors of rank one and semistable sheaf direct images.

\subsection{Divisors with prescribed component classes}\label{sub:strata}
Recall that $\Gamma$ is the set~\eqref{intro:Gamma} of effective classes with a smooth connected member. A partition of $\beta$ into classes in $\Gamma$ is a class partition $\lambda=(m_\gamma)_{\gamma\in\Gamma}$ with $\sum_\gamma m_\gamma\gamma=\beta$; we write $\lambda\vdash_\Gamma\beta$. On $\PP^2$ every effective class lies in $\Gamma$; on $\mathbb{F}_0$ a class $(m,0)$ must be written as $m$ copies of the ruling $f$. For $\lambda\vdash_\Gamma\beta$ we set
\begin{align}\label{eq:stratum}
\Sigma_\lambda:=\mu_\lambda\Bigl(\prod_{\gamma\in\Gamma}B_\gamma^{m_\gamma}\Bigr)\subset B_\beta,\qquad
U_\lambda\subset\Sigma_\lambda,
\end{align}
where $U_\lambda$ is the locus of divisors which are unions of $\sum_\gamma m_\gamma$ distinct smooth connected curves of the classes prescribed by $\lambda$, meeting pairwise transversely with no triple points. In the notation of~\cite[Section~3.2]{Zha26}, $U_\lambda=\Sigma^{\circ}_\lambda$ and $\Sigma_\lambda=Z_\lambda$. Since $\mu_\lambda$ is finite and its source is irreducible, $\Sigma_\lambda$ is closed and irreducible, and $U_\lambda$ is a dense open subset of it. Distinct partitions give distinct strata. For the one-term partition of $\beta\in\Gamma$ we have $\Sigma_\lambda=B_\beta$ and $U_\lambda=U_\beta$.

\subsection{Complexes with these supports}\label{sub:categories}
For a variety $Y$, we write $\MHM(Y)$ for the abelian category of mixed Hodge modules on $Y$~\cite{Sai90}, using its monodromic version for vanishing cycle mixed Hodge modules. We denote by $\widehat{D}\MHM(B_\beta)$ the bounded below left completion of $D^b\MHM(B_\beta)$ for the perverse t-structure. The cohomology of stabilizer groups need not be bounded, so we work in this completion and make support assertions in each perverse degree. By~\cite[Theorem~3.8(2)]{Tub25}, finite type direct image preserves this category.

Let $\cS_\beta$ be the full subcategory of $\widehat{D}\MHM(B_\beta)$ consisting of the complexes $K$ such that every simple constituent of every $\pH{i}(K)$ has strict support $\Sigma_\lambda$ for some $\lambda\vdash_\Gamma\beta$. For the zero class we set $\cS_0:=\widehat{D}\MHM(B_0)$. The assertion of Theorem~\ref{thm:support} is $K_{\chi,\beta}\in\cS_\beta$.

Verdier duality exchanges the left and right perverse completions~\cite[Proposition~3.13, Corollary~3.14]{Tub25}. We apply it to bounded direct images or finite perverse truncations and then pass to the left completion.

Let
\begin{align*}
\rat\colon\widehat{D}\MHM(B_\beta)\to\widehat{D}_c(B_\beta,\QQ)
\end{align*}
be the rational realization, where $\widehat{D}_c$ is the analogous perverse completion of rational constructible complexes. It is conservative, perverse t-exact, and commutes with the six operations used below~\cite[Theorem~3.1, Proposition~3.16]{Tub25}. A simple Hodge module and the simple constituents of its rational realization have the same strict support, so membership in $\cS_\beta$ can be checked after applying $\rat$.

For $\gamma_1+\cdots+\gamma_r=\beta$, we write
\begin{align}\label{eq:addition}
\mu\colon B_{\gamma_1}\times\cdots\times B_{\gamma_r}\to B_\beta
\end{align}
for addition of divisors.

\begin{defi}\label{def:G}
Let $\cG_\beta\subset\widehat{D}\MHM(B_\beta)$ be the smallest full triangulated subcategory, closed under finite direct sums, direct summands, Tate twists, tensor products with bounded below graded mixed Hodge structures with finite dimensional graded pieces, and invariants under finite groups, which contains every finite convolution
\begin{align*}
\dR\mu_*\bigl(\IC_{B_{\gamma_1}}(L_1)\boxtimes\cdots\boxtimes\IC_{B_{\gamma_r}}(L_r)\bigr),\qquad
\gamma_i>0,\quad \sum_i\gamma_i=\beta,
\end{align*}
where each $L_i$ is a semisimple admissible variation of mixed Hodge structure on a dense open subset of $B_{\gamma_i}$ and $\IC_{B_{\gamma_i}}(L_i)$ has strict support $B_{\gamma_i}$. For $\beta=0$ we set $\cG_0:=\widehat{D}\MHM(B_0)$.
\end{defi}

\begin{lem}\label{lem:serre}
The category $\cS_\beta$ is closed under shifts, cones, direct summands, finite direct sums, Tate twists, and tensor products with mixed Hodge modules on $B_\beta$ whose rational realizations are shifts of local systems. It also satisfies:
\begin{enumerate}
\item[(i)] $\cG_\beta\subset\cS_\beta$;
\item[(ii)] if $K\in\cS_\gamma$, $G\in\cG_\delta$, and $\mu\colon B_\gamma\times B_\delta\to B_{\gamma+\delta}$ is the addition of divisors, then $\dR\mu_*(K\boxtimes G)\in\cS_{\gamma+\delta}$.
\end{enumerate}
Both statements remain true after taking invariants under a finite group.
\end{lem}
\begin{proof}
Consider the perverse Hodge modules whose simple constituents have the prescribed strict supports. They form a Serre subcategory, so the closure assertions follow from the perverse long exact sequence and the fact that tensoring with such a Hodge module preserves strict supports.

Each perverse cohomology object has finite length, and the Postnikov tower converges in the left completion. A finite d\'evissage reduces $K$ in (ii) to a term $\IC_{\Sigma_\lambda}(L)[m]$, and reduces the other factor to a generator of $\cG_\delta$ with source support $\prod_iB_{\gamma_i}$. For each $\gamma_i\in\Gamma$ use the one-term partition $\lambda_i=(\gamma_i)$; for a multiple ruling $\gamma_i=mf$, use the partition into $m$ copies of $f$. These exhaust the effective classes on the two surfaces, and in each case $B_{\gamma_i}=\Sigma_{\lambda_i}$.

By~\eqref{eq:stratum}, addition sends $\prod_iB_{\gamma_i}$ onto $\Sigma_{\lambda_1+\cdots+\lambda_r}$ and sends $\Sigma_\lambda\times\prod_iB_{\gamma_i}$ onto $\Sigma_{\lambda+\lambda_1+\cdots+\lambda_r}$, where partitions are added by adding the multiplicities of each class. The addition map is finite and perverse t-exact. Its direct image of an intersection complex with irreducible source support has only simple constituents with strict support equal to the image of that source support, as in~\cite[Proposition~4.3]{Zha26}. This proves (i) on generators and (ii) after d\'evissage. The operations in Definition~\ref{def:G} preserve the assertion: each fixed perverse degree involves only finitely many terms from a bounded below graded Hodge factor, and finite group invariants are exact over $\QQ$.
\end{proof}

\subsection{Direct images from semistable sheaves}\label{sub:rank-zero}
Davison-Meinhardt~\cite[Theorem~A]{DM20} proved cohomological integrality for quivers with potential. For the sheaf stacks below, we use the cohomological integrality theorem of~\cite{BDINP25}.

Let $\cM^{\sss}_X(\gamma,\eta)$ be the stack of semistable compactly supported pure sheaves on $X$ of class $\gamma>0$ and Euler characteristic $\eta$, let $h_{\eta,\gamma}\colon\cM^{\sss}_X(\gamma,\eta)\to B_\gamma$ be the Hilbert-Chow map, and set
\begin{align}\label{eq:H-eta-gamma}
H_{\eta,\gamma}:=\dR h_{\eta,\gamma*}\phi^{\sh}_{\eta,\gamma}.
\end{align}
Here $\phi^{\sh}_{\eta,\gamma}$ is the vanishing cycle module for the canonical cotangent orientation, compatible with direct sums.

\begin{lem}\label{lem:rank-zero}
We have $H_{\eta,\gamma}\in\cG_\gamma\subset\cS_\gamma$.
\end{lem}
\begin{proof}
Write a compactly supported sheaf on $X$ as a Higgs pair $(E,\theta)$ on $S$. Twisting by $K_S=L^{-\kappa}$ lowers one dimensional slopes by $\kappa$, so $\Hom(E_1,(E/E_1)\otimes K_S)=0$ for the maximal HN term $E_1$, and $E_1$ is $\theta$-invariant. Thus semistability on $X$ forces $E$ to be semistable on $S$. Now $\Hom(E,E\otimes K_S)=0$, so $\theta=0$. The semistable stack on $X$ is the shifted cotangent of the semistable stack on $S$.

Fix a reduced Hilbert polynomial and take the stack of semistable surface sheaves over all numerical components, including zero. Its components are of finite type, have Simpson good moduli spaces, and the stack is closed under direct sums and summands. For two semistable sheaves $E$, $F$ of this slope, Serre duality and the negativity of $K_S$ give
\begin{align*}
\Ext^2(E,F)=\Hom(F,E\otimes K_S)^\vee=0,
\end{align*}
so the stack is smooth. At a polystable closed point $\bigoplus_iS_i\otimes V_i$, stability makes the dimensions of $\Hom$ between the stable factors symmetric, while $\chi(S_i,S_j)=-[S_i]\cdot[S_j]$, where $\chi(A,B)=\sum_k(-1)^k\dim\Ext^k_S(A,B)$ is the Euler pairing. Thus $\dim\Ext^1(S_i,S_j)=\dim\Ext^1(S_j,S_i)$, and the same holds between polystable objects. These equalities and the moduli properties above verify conditions (i)--(iii) and (v) of the integrality theorem; see~\cite[Section~10.2.3]{BDINP25}. A torus grading of a sheaf $E$ is a homomorphism $T\to\operatorname{Aut}(E)$ from an algebraic torus; it decomposes $E$ into its character eigensheaves. The corresponding graded stack parametrizes such decompositions. The required line bundles whose first Chern classes restrict to a basis of $H^2(BT,\QQ)$ exist by smoothness and~\cite[Corollary~9.1.4]{BDINP25}, including on these graded stacks. Toda's shifted cotangent presentation~\cite[Section~3.4]{Tod24} has the canonical orientation with direct-sum comparison isomorphisms equivariant under permutations, as in~\cite[Lemma~10.2.9]{BDINP25}.

Dimensional reduction~\cite[(6.2.1.4)]{BDINP25} identifies the vanishing cycle module with the Verdier dual of the constant object in the derived category of mixed Hodge modules on the surface stack. The integrality theorem for smooth stacks~\cite[Theorem~10.2.7]{BDINP25} then expresses its good moduli direct image as a sum of intersection complexes of the good moduli spaces of the graded stacks, indexed by the torus gradings specified in that theorem, tensored with the cohomology of the classifying stacks of the grading tori and followed by finite group invariants. All eigensheaves have the fixed slope. By~\cite[Theorem~1.2.7]{BDINP25}, a graded component contributes the intersection complex of its good moduli space if the good moduli map, after rigidification by its grading torus, is generically quasi-finite, and contributes zero otherwise.

Let $M_{\delta,\xi}$ be the Simpson moduli space of semistable surface sheaves of class $\delta$ and Euler characteristic $\xi$. For an ample class $\delta$ and any $\xi$, consider the grading given by scalar multiplication. The irreducibility results of Maulik-Shen~\cite[Theorem~2.3, Proposition~2.10, Section~2.6]{MS23} show that stable sheaves supported on smooth curves form a dense locus, identified with the relative Picard stack of degree $\xi+g_\delta-1$. After rigidification by scalar automorphisms, it is the relative Picard scheme $\Pic^{\xi+g_\delta-1}(\cC_\delta/U_\delta)$, on which the good moduli map is an isomorphism; hence it is generically quasi-finite, and the corresponding summand is $\IC_{M_{\delta,\xi}}$. By~\cite[Theorem~0.4]{MS23}, the Chow direct image of this intersection complex has only the full support $B_\delta$.

For the non-ample classes on $\mathbb{F}_0$, Proposition~\ref{prop:ruling} identifies the semistable stack in class $mf$ with the stack of length $m$ sheaves on $\PP^1$. Its good moduli space is $\Sym^m\PP^1$. Over $m$ distinct fibers the stabilizer is $\Gm^m$; after rigidification by the scalar $\Gm$, it is $\Gm^{m-1}$. For $m>1$ the rigidified good moduli map is therefore not generically quasi-finite, and its contribution in~\cite[Theorem~1.2.7]{BDINP25} is zero. For $m=1$ the direct image has full support $B_f=\PP^1$. The higher classes contribute through symmetric products of this term. Ample classes and multiples of a ruling exhaust the effective classes on the two surfaces.

The Hilbert-Chow map is additive: $h_{\eta,\gamma}\circ\oplus=\mu\circ\prod_ih_{\eta_i,\gamma_i}$ with $\sum_i\gamma_i=\gamma$ and $\sum_i\eta_i=\eta$. Applying Chow direct image to the integrality decomposition therefore gives finite convolutions of the direct images just computed. In each fixed numerical and perverse range, the other factors are graded Hodge structures and finite group invariants, as allowed in Definition~\ref{def:G}. Thus $H_{\eta,\gamma}\in\cG_\gamma$; Lemma~\ref{lem:serre}(i) gives $H_{\eta,\gamma}\in\cS_\gamma$.
\end{proof}

At Euler characteristic one, these direct images identify the series in Definition~\ref{intro:def-A} with the cohomological GV data of Maulik-Toda~\cite{MT18}.

\begin{prop}\label{prop:primitive}
Let $M_\beta$ be the moduli space of $L$-stable pure one dimensional sheaves on $S$ of class $\beta$ and Euler characteristic one, with Hilbert-Chow map $h_\beta\colon M_\beta\to B_\beta$. For $\beta\in\Gamma$ we have
\begin{align*}
\pH{i}\bigl(\dR h_{\beta*}\IC_{M_\beta}\bigr)=\IC_{B_\beta}\bigl(\wedge^{g_\beta+i}V_\beta\bigr),\qquad -g_\beta\le i\le g_\beta,
\end{align*}
and all other perverse degrees vanish. For $\beta\notin\Gamma$, $M_\beta$ is empty. Consequently \eqref{intro:A} is built from the cohomological GV data of~\cite{MT18}, and \eqref{intro:ruling-A} gives the series for the ruling classes.
\end{prop}
\begin{proof}
At Euler characteristic one, semistability coincides with stability. The moduli space $M_\beta$ is projective and smooth of dimension $\beta^2+1$, since $\Ext^2(F,F)=\Hom(F,F\otimes K_S)^\vee=0$ for a stable sheaf $F$. For an ample class $\beta$, the support theorem of Maulik-Shen says that every perverse cohomology sheaf of the direct image is the intermediate extension of its restriction to $U_\beta$. Over $U_\beta$ the Hilbert-Chow map is the relative Picard variety of degree $g_\beta$ of the family of smooth curves, whose cohomology local systems are the $\wedge^jV_\beta$. The relative dimension $g_\beta$ converts the cohomological index $j$ into the perverse index $i=j-g_\beta$, and the intermediate extension gives the formula. Proposition~\ref{prop:ruling} proves the assertion for the two ruling classes and the emptiness for their higher multiples; these are the only effective classes outside the ample cone.

For a local del Pezzo surface the BPS sheaf is the intersection complex of the moduli space of surface sheaves~\cite[Section~0.3, Theorem~0.6]{MS23}; in the absence of strictly semistable sheaves, its definition specializes to~\cite{MT18}. Since $\chi=1$, the displayed direct image for $M_\beta$ is the complex defining the GV invariants.
\end{proof}

\section{Wall-crossing over the Chow variety}\label{sec:wall}
Toda's HN stratifications~\cite[Section~6.2.1]{Tod24} compare adjacent chambers. Their localization filtrations over $B_\beta$ apply also at the zero wall. Combining these comparisons with duality will relate $\chi$ and $-\chi$. Classical hyperbolic localization for equivariant constructible complexes is due to Braden~\cite{Bra03}; we use its Donaldson-Thomas counterparts~\cite{KPS26,Des25}.

\subsection{Stability chambers}\label{sub:heart}\label{sub:window}
Let $\ol{X}=\PP_S(K_S\oplus\cO_S)$, let $\ol{p}\colon\ol{X}\to S$ be the projection, and let $S_\infty$ be the divisor at infinity. We use Toda's weak stability conditions on the heart
\begin{align}\label{eq:heart}
\cA_X:=\langle\cO_{\ol{X}},\Coh_{\le 1}(X)[-1]\rangle_{\mathrm{ex}}.
\end{align}
Here the bracket means the extension closure, and $\Coh_{\le 1}(X)$ is the category of compactly supported coherent sheaves on $X$ of dimension at most one. For a rank zero object $E[-1]$, the numerical notation $(0,\delta,\eta)$ records $[E]=\delta$ and $\chi(E)=\eta$; for a rank one object whose surface pair is $\cO_S\to F$, we write $(1,[F],\chi(F))$. A flat family in $\cA_X$ is a family given by Toda's diagrams with the surface sheaf flat over the base~\cite[Theorem~4.5.7, Corollary~4.5.9]{Tod24}. Its derived fibers lie in $\cA_X$, and the diagram description commutes with base change. For a rank one class, Toda's weak slope $\mu^\dagger_t$ is the real parameter $t$; for a positive degree rank zero class $(0,\beta,\eta)$ it is $\eta/(L\cdot\beta)$. Let $P^t_\chi(X,\beta)$ be the stack of stable objects in a generic chamber $t$, with Hilbert-Chow map $\Pi^t_{\chi,\beta}$ and oriented vanishing cycle mixed Hodge module $\phi^t_{\chi,\beta}$, and set
\begin{align}\label{eq:chamber-K}
K^t_{\chi,\beta}:=\dR\bigl(\Pi^t_{\chi,\beta}\bigr)_*\phi^t_{\chi,\beta}.
\end{align}
By~\cite[Definition~4.2.5]{Tod24} and the discussion following it, the stack is of finite type and there are finitely many walls for each numerical class. For $t\gg 0$ the chamber is the PT chamber, so $K^t_{\chi,\beta}=K_{\chi,\beta}$ by~\cite[Theorem~4.2.10(i)]{Tod24} and Lemma~\ref{lem:proper}. We write $0^+$ and $0^-$ for the chambers immediately to the right and to the left of zero.

\begin{thm}\label{thm:chamber}
Fix $\beta>0$ and assume that, for every $0\le\gamma<\beta$, every Euler characteristic and every generic chamber, the rank one Chow direct image belongs to $\cS_\gamma$. Then for every $\chi$, the condition $K^t_{\chi,\beta}\in\cS_\beta$ is independent of the generic chamber parameter $t$. In particular
\begin{align*}
K_{\chi,\beta}\in\cS_\beta\ \Longleftrightarrow\ K^{0^+}_{\chi,\beta}\in\cS_\beta.
\end{align*}
Moreover $K^{0^+}_{\chi,\beta}\in\cS_\beta$ if and only if $K^{0^+}_{-\chi,\beta}\in\cS_\beta$.
\end{thm}

\begin{cor}\label{cor:window}
Under the hypothesis of Theorem~\ref{thm:chamber}, the support assertion $K_{\chi,\beta}\in\cS_\beta$ holds outside the interval
\begin{align}\label{eq:window}
1-g_\beta\le\chi\le g_\beta-1.
\end{align}
If $g_\beta\le 0$, it holds for every $\chi$.
\end{cor}

\subsection{Filtered objects and orientations}\label{sub:determinant}
For a derived stack $\fX$, write $\Filt(\fX):=\operatorname{Map}([\AAA^1/\Gm],\fX)$ and $\Grad(\fX):=\operatorname{Map}(B\Gm,\fX)$ for its stacks of filtered and graded objects~\cite{HL20,KPS26}. Evaluation at $1$ gives $\ev$, and restriction to $0$ gives the associated graded map $\gr$. For a numerical type $\tau$, the notation $\Filt_\tau$ specifies the weights and factor classes, with flat graded factors in $\cA_X$. We write $\Grad_\tau$ for the corresponding stack of graded objects.

\begin{lem}\label{lem:rees}
Let $\bM^\dagger_S$ be Toda's derived stack of pure surface pairs. A framing of a rank one object $I$ is a specified isomorphism $Li_\infty^*I\simeq\cO_S$, where $i_\infty\colon S_\infty\hookrightarrow\ol{X}$ is the inclusion. Lemma~\ref{lem:derived-cotangent} identifies the derived stack of these objects in $\cA_X$ whose corresponding surface sheaf is pure of dimension one with $T^*[-1]\bM^\dagger_S$. The open substack of objects stable for $t\gg0$ is the stable pair moduli space, with its vanishing cycle mixed Hodge module identified by Theorem~\ref{thm:oriented}. Let $\fU$ be the derived open substack remaining after the preceding HN strata in the order of Proposition~\ref{prop:wall} have been removed. Write $\beta=\gamma+\delta$, and fix a two-weight numerical type with graded classes $(1,\gamma,\chi_0)$ and $(0,\delta,\eta)$. Then:
\begin{enumerate}
\item[(i)] on the locus of flat filtrations of the prescribed type, the classical truncation of $\Filt(\fU)$ is Toda's stack of short exact sequences in $\cA_X$ with the prescribed factor classes, and the truncation of the evaluation map sends a sequence to its middle term;
\item[(ii)] the full two-weight component of $\Grad(T^*[-1]\bM^\dagger_S)$ is canonically the derived product of the framed rank one factor and the compactly supported rank zero factor, and the component for $\fU$ is the inverse image derived open substack in this product;
\item[(iii)] the orientation line restricted to the graded component is the tensor product of the factor orientation lines and the determinant of the positive weight part of the cotangent complex, computed in Lemma~\ref{lem:orient}. Consequently the localized orientation of Kinjo-Park-Safronov on the graded stack is the product orientation.
\end{enumerate}
\end{lem}
\begin{proof}
Since $\AAA^1$ is affine and $\Gm$ is linearly reductive, classical truncation commutes with these mapping stack constructions:
\begin{align*}
t_0\Filt(\fU)=\Filt(t_0\fU),\qquad t_0\Grad(\fU)=\Grad(t_0\fU).
\end{align*}
On the locus of flat filtrations, the Rees equivalence identifies a two-weight filtration with a short exact sequence in $\cA_X$, and the evaluation at $1$ with its middle term. Toda's exact surface functor identifies these with his families of short exact sequences, compatibly with arbitrary classical base change~\cite[Theorem~4.5.7, Corollary~4.5.9, Section~6.1.4]{Tod24}.

There is a natural equivalence
\begin{align}\label{eq:Grad-cotangent}
\Grad\bigl(T^*[-1]\bM^\dagger_S\bigr)\simeq T^*[-1]\Grad\bigl(\bM^\dagger_S\bigr).
\end{align}
This is the graded critical locus identity of~\cite[Section~6.1.12]{BDINP25}, applied to the zero function on $\bM^\dagger_S$. For the prescribed type, the graded pair is the direct sum of a framed pair and a pure sheaf. Its deformation complex has weight zero on the two diagonal blocks and non-zero weights on the cross blocks. Taking weight zero therefore gives the product of the derived factor stacks; applying~\eqref{eq:Grad-cotangent} gives the corresponding product of shifted cotangents. The decomposition works for any finite filtration with distinct weights. For $\fU$, take the inverse image of its classical open substack. If membership is determined by the prescribed HN type of the split object, this is the product of the open substacks where the factors satisfy the prescribed stability conditions, as asserted in (ii).

The weight-zero terms give the factor orientation lines. The positive weight part of the cotangent complex gives the determinant in (iii), whose inverse enters localization. Lemma~\ref{lem:orient} identifies the resulting orientation with the product orientation, including the isomorphisms to the virtual canonical lines.
\end{proof}

\begin{lem}\label{lem:orient}
Let $I$ be a framed rank one object and let $p(I)=(\cO_S\to F)$ be the surface pair obtained from Toda's map in \cite[Theorem~4.1.3]{Tod24}. Let $L_I$ be the orientation line of Theorem~\ref{thm:oriented}. For a surface sheaf $A$, write $L_g(A):=\det\Rhom_S(A,A)$; the orientation of a compactly supported sheaf $E$ on $X$ in Lemma~\ref{lem:rank-zero} is $L_g(p_*E)$. Then
\begin{align*}
L_I=L_g(F)\otimes\bigl(\det\Rgam(S,F)\bigr)^{-1},\qquad L_g(F)=\det\Rhom_S(F,F).
\end{align*}
For an exact sequence $0\to I_0\to I\to E[-1]\to 0$ in the heart $\cA_X$, set $E_S:=p_*E$. There is a functorial isomorphism
\begin{align}\label{eq:orient}
L_I\cong L_{I_0}\otimes L_g(E_S)\otimes\det\Rhom_X(I_0,E[-1]).
\end{align}
For an exact sequence in the opposite order, $0\to E[-1]\to I\to I_0\to 0$, the corresponding isomorphism is
\begin{align}\label{eq:orient-reversed}
L_I\cong L_g(E_S)\otimes L_{I_0}\otimes\det\Rhom_X(E[-1],I_0).
\end{align}
These isomorphisms are associative for iterated filtrations. They agree with the product orientation and the determinant of the positive weight part of the cotangent complex in the localization of filtered objects, and Toda's derived duality preserves them.
\end{lem}
\begin{proof}
The orientation on the stack of compactly supported sheaves is the shifted cotangent orientation, with its direct-sum and associativity isomorphisms~\cite[Example~8.32, Proposition~8.34, Corollary~8.36]{KPS26}. For surface pairs, the pair triangle identifies
\begin{align*}
\det\LL_{\bM^\dagger_S}=\det\Rhom_S(F,F)\otimes\bigl(\det\Rgam(S,F)\bigr)^{-1}=L_I.
\end{align*}
Its square is identified with the virtual canonical line by the cotangent construction~\cite[Example~3.20]{KPS26}. We compute on the reduced truncation of the universal filtered stack; the graded determinant lines pull back with the exact diagrams.

Write $p(I_0)=(\cO_S\to F_0)$. Toda's exact surface functor turns the given sequence into either $0\to F_0\to F\to E_S\to 0$ or the reverse order. For surface sheaves $A,B$, adjunction along the zero section $i\colon S\hookrightarrow X$ and Serre duality give
\begin{align}\label{eq:surface-cross}
\det\Rhom_X(i_*A,i_*B)=\det\Rhom_S(A,B)\otimes\det\Rhom_S(B,A).
\end{align}
For a general $I_0$, the cross term is computed from Toda's diagram
\begin{align*}
\cO_S\to U_0\to F_0\otimes K_S^{-1},\qquad I_0=\fib\bigl(\ol{p}^*U_0\to\ol{p}^*F_0(S_\infty)\bigr).
\end{align*}
Put $D_{01}:=\Rhom_X(I_0,E[-1])$. Restrict this reconstruction to $X$, where $\cO_{\ol X}(S_\infty)$ is trivial, and apply $\Rhom_X(-,E[-1])$. Adjunction for $p^*$ and $p_*$ gives
\begin{align*}
\det D_{01}=\frac{\det\Rhom_S(F_0,E_S)}{\det\Rhom_S(U_0,E_S)}.
\end{align*}
Calabi-Yau duality gives the same determinant for $D_{10}:=\Rhom_X(E[-1],I_0)$, and expanding $U_0$ through its fiber sequence and applying Serre duality on $S$ yields
\begin{align}\label{eq:framed-cross}
\det D_{01}=\det D_{10}=\frac{\det\Rhom_S(F_0,E_S)\otimes\det\Rhom_S(E_S,F_0)}{\det\Rgam(S,E_S)}.
\end{align}
On the other hand, the multiplicativity of determinants along either surface exact sequence gives
\begin{align*}
L_g(F)=L_g(F_0)\otimes L_g(E_S)\otimes\det\Rhom_S(F_0,E_S)\otimes\det\Rhom_S(E_S,F_0),
\end{align*}
and $\det\Rgam(S,F)=\det\Rgam(S,F_0)\otimes\det\Rgam(S,E_S)$. Substituting these into the definition of $L_I$ proves the two orientation isomorphisms.

Give $I_0$ weight zero and $E[-1]$ weight one. The tangent terms of nonzero weight are $D_{01}[1]$ and $D_{10}[1]$. Writing $\LL^{>0}$ for the positive weight part of the cotangent complex, we have
\begin{align*}
\LL^{>0}=(D_{10}[1])^\vee=D_{01}[2],\qquad \det\LL^{>0}=\det D_{01}.
\end{align*}
Localization tensors $L_I$ by the inverse of this line~\cite[Section~6.5]{KPS26}, and what remains is the tensor product of the factor orientation lines. The reverse order exchanges the roles of $01$ and $10$.

For the square isomorphisms, let $K_I$, $K_{I_0}$, $K_E$ be the virtual canonical lines with orientation lines $L_I$, $L_{I_0}$, $L_g(E_S)$. The endomorphism filtration, with $D_{10}=D_{01}^\vee[-3]$, gives
\begin{align}\label{eq:canonical-square}
K_I=K_{I_0}\otimes K_E\otimes(\det D_{01})^{\otimes 2}.
\end{align}
Taking graded determinants of the same exact diagram shows that the square of~\eqref{eq:orient} is~\eqref{eq:canonical-square}. For three steps, additivity for a double fiber sequence identifies the two orders of taking determinants~\cite[Proposition~2.1, diagram~(2.5)]{KPS26}; induction gives associativity for every finite filtration.

Toda's duality takes the dagger presentation at $v=(\beta,\chi)$ to the sharp presentation at $v^\vee=(\beta,-\chi)$ \cite[Lemmas~4.4.2, 4.4.4]{Tod24}, preserving the cotangent orientation line. Lemma~\ref{lem:linear-duality} identifies the oriented sharp and dagger presentations at $v^\vee$. These comparisons preserve direct sums and weight decompositions and exchange $D_{01}$ and $D_{10}$. They carry~\eqref{eq:orient} to~\eqref{eq:orient-reversed}, together with~\eqref{eq:canonical-square} and the determinant identities for iterated filtrations.
\end{proof}

\begin{lem}\label{lem:linear-duality}
Let $\bM$ be a quasi-smooth derived stack and let $E$ be a perfect complex of amplitude $[-1,0]$. Set
\begin{align*}
\bM^\dagger:=\Spec_{\bM}\Sym(E),\qquad \bM^\sharp:=\Spec_{\bM}\Sym(E^\vee[1]).
\end{align*}
The comparison $t_0T^*[-1]\bM^\dagger\simeq t_0T^*[-1]\bM^\sharp$ in~\cite[Lemma~4.6.13]{Tod24} identifies their $d$-critical sections and their canonical cotangent orientations. Therefore it identifies their oriented vanishing cycle mixed Hodge modules. The comparison commutes with base change, direct sums and the weight decompositions used for the localization of filtered objects.
\end{lem}
\begin{proof}
On a derived affine chart, write $\bU$ as the derived zero locus of a section $s$ of a vector bundle $V$ on a smooth scheme $Y$, and represent $E$ by $[E^{-1}\xrightarrow{\varphi}E^0]$. Toda's construction gives the following common critical chart for the two cotangent stacks:
\begin{align*}
Z=\Tot_Y\bigl((E^0)^\vee\oplus E^{-1}\oplus V^\vee\bigr),\qquad w(y,e,e',v)=\langle\varphi(e'),e\rangle+\langle s(y),v\rangle.
\end{align*}
Both $d$-critical sections are represented by $w$ modulo the square of the critical ideal.

In both presentations the cotangent orientation line, pulled back to this chart, is
\begin{align}\label{eq:common-root}
K_Y\otimes\det V\otimes\det E^0\otimes(\det E^{-1})^{-1}=K_Z.
\end{align}
For the dagger presentation, take the determinant of the cotangent complex of the zero locus on $\Tot((E^0)^\vee)$; for the sharp presentation, use the zero locus on $\Tot(E^{-1})$. Both isomorphisms from the squares of the orientation lines become $K_Z^{\otimes 2}\to K_{\Crit(w),w}$. We use the displayed ordering of the summands of $Z$ in the graded determinant convention; permutations carry the associated Koszul signs. On a quotient chart for a smooth group $G$ with Lie algebra $\fg$, the quotient cotangent complex contributes $(\det\fg^\vee)^{-1}$ to each orientation line. The function $w$ and the identity~\eqref{eq:common-root} are $G$-equivariant, so they descend with this common factor.

Pull-back of $s$, $E$ and their evaluation pairings pulls back $w$ and~\eqref{eq:common-root}. Direct sums use the direct sum determinant identity; taking weight zero retains the diagonal evaluation pairings and their determinant factors.

To change the two-term resolution, pass to a common refinement. It is then enough to consider an isomorphism of complexes and the addition of an acyclic complex $[A\xrightarrow{1}A]$. An isomorphism acts by the same change of variables on $w$ and by its determinant on both copies of~\eqref{eq:common-root}. Adding the acyclic complex replaces $w$ by $w+\langle a,a^*\rangle$; the contraction of $\det A$ with $\det A^\vee$ gives the standard hyperbolic orientation on both sides.

For changes of zero locus presentation, use the common refinement and quadratic normal form comparison for critical charts~\cite[Theorem~3.25]{KPS26}. The corresponding stabilization theorem for vanishing-cycle sheaves is~\cite[Theorem~5.4]{BBDJS15}. An elementary stabilization adds a smooth coordinate $u$ and the equation $u=0$, giving the function $w+uu'$ with $u'$ dual to $u$. The determinant contractions agree in the dagger and sharp presentations. For a general embedding of critical charts, the determinant of the quadratic form on the normal bundle gives the same isomorphism between the squares of the chart canonical bundles in these presentations. The $\mu_2$-torsor of its square roots supplies the twist in the stabilization isomorphism, and the oriented comparisons agree on this torsor.

Associativity of graded determinants gives the cocycle condition for successive embeddings, with signs determined by the chosen ordering. The Lie algebra factor $(\det\fg^\vee)^{-1}$ in each orientation line then gives smooth descent to stacks.
\end{proof}

\begin{lem}\label{lem:duality}
Toda's derived duality identifies $P^t_\chi(X,\beta)\cong P^{-t}_{-\chi}(X,\beta)$ and preserves the fundamental cycle. Consequently
\begin{align*}
K^t_{\chi,\beta}\in\cS_\beta\ \Longleftrightarrow\ K^{-t}_{-\chi,\beta}\in\cS_\beta.
\end{align*}
\end{lem}
\begin{proof}
The isomorphism of moduli stacks is~\cite[Lemma~4.4.6]{Tod24}. Duality preserves each irreducible component of the support and its generic multiplicity, so it commutes with the Hilbert-Chow map to $B_\beta$. The derived duality of~\cite[Lemma~4.4.2]{Tod24} takes the vanishing-cycle module for the $\dagger$ presentation in class $(\beta,\chi)$ to the one for the $\sharp$ presentation in class $(\beta,-\chi)$; on classical stacks this is the reflection of~\cite[Lemma~4.4.4]{Tod24}. Lemma~\ref{lem:linear-duality} identifies the sharp and dagger orientations at $(\beta,-\chi)$. Writing $\phi^\dagger_{\chi,\beta}$ and $\phi^\sharp_{\chi,\beta}$ for these vanishing-cycle modules, we obtain
\begin{align*}
\phi^\dagger_{\chi,\beta}\xrightarrow{\ \text{duality}\ }\phi^\sharp_{-\chi,\beta}\xrightarrow{\ \text{linear comparison}\ }\phi^\dagger_{-\chi,\beta}.
\end{align*}
Taking direct images to $B_\beta$ gives the chamber complexes up to the uniform shift and Tate normalization, which preserve strict supports. The two support assertions are therefore equivalent.
\end{proof}

\subsection{Relative flags and wall localization}\label{sub:flags}
\begin{lem}\label{lem:flags}
Let $T$ be a classical base and let $I_T$ be a flat family of rank one objects in $\cA_X$. Fix an ordered sequence of numerical classes $\gamma_1,\ldots,\gamma_m$ whose sum is the class of $I_T$, with exactly one rank one class and all other classes of the form $(0,\delta,\eta)$. The rank one class may occur at any position. Then the functor of filtrations
\begin{align*}
0=I_0\subset I_1\subset\cdots\subset I_m=I_T,\qquad [I_j/I_{j-1}]=\gamma_j,
\end{align*}
is represented by a closed subfunctor of the product of the relative flag-Quot schemes of the coherent sheaves in Toda's diagram. The flag in $\cO_{S\times T}$ is fixed to be zero before the rank one step and $\cO_{S\times T}$ from that step onward, and the closed conditions require every arrow of the diagram to preserve the flags. In particular the evaluation map to the rank one moduli stack is representable, proper and of finite presentation, and the construction commutes with arbitrary classical base change.
\end{lem}
\begin{proof}
Toda's exact diagram equivalence identifies flat filtrations in $\cA_X$ with filtrations by subdiagrams over any classical base~\cite[Theorem~4.5.7, Corollary~4.5.9]{Tod24}. Form the relative flag-Quot scheme for each coherent sheaf in the diagram.

Concretely, in the diagram
\begin{align*}
0\to\cO\to U\to F\otimes K_S^{-1}\to 0,\qquad \varphi\colon U\to F,
\end{align*}
take flags $U_j\subset U$ and $F_j\subset F$ with the prescribed Hilbert polynomials and flat quotients, and require
\begin{align*}
\cO_j\to U_j,\qquad U_j\to F_j\otimes K_S^{-1},\qquad \varphi(U_j)\subset F_j,
\end{align*}
where $\cO_j=0$ before the rank one step and $\cO_j=\cO$ thereafter. Observe that before the rank one step, $U_j\cap\cO=0$ by the fiberwise torsion freeness of $\cO$. Hence $U_j\to F_j\otimes K_S^{-1}$ is injective on fibers; afterwards the same holds for $U_j/\cO$. Equality of the prescribed Hilbert polynomials then makes these maps isomorphisms, and flatness together with the fiberwise exactness criterion gives exact rows after arbitrary base change.

An arrow preserves a flag step precisely when its composite with the corresponding quotient map vanishes. These composites are defined on the stack and pull back under base change. On a smooth atlas, finite locally free presentations express their vanishing as zero loci of morphisms of coherent modules. The resulting closed subfunctor of the product of projective flag-Quot schemes descends from the atlas, so the evaluation map is representable, proper and of finite presentation.
\end{proof}

\begin{prop}\label{prop:wall}
Fix the class $(1,\beta,\chi)$ and a non-negative wall $\lambda$. Choose generic parameters $t_-<\lambda<t_+$ in the adjacent chambers, simultaneously for this class and for the lower degree rank one classes in its wall filtrations. In the shifted cotangent enhancement of Lemma~\ref{lem:rees}, let $\fW_\lambda$ be the finite type derived open substack over Toda's stack $\cP^\lambda_\chi(X,\beta)$ of wall-semistable objects, with the orientation fixed in Subsection~\ref{sub:PTsheaf}. Denote its Hilbert-Chow map to $B_\beta$ by $\Pi_{\fW}$ and its DT mixed Hodge module by $\phi_{\fW}$.

Then the wall direct image $\dR\Pi_{\fW*}\phi_{\fW}$ has two finite sequences of localization triangles, one for each HN order, whose open terms are the adjacent chamber complexes $K^{t_-}_{\chi,\beta}$ and $K^{t_+}_{\chi,\beta}$. These sequences lie in $\widehat{D}\MHM(B_\beta)$ and give finite perverse spectral sequences. After rational realization, every other term is a finite direct sum of direct summands of the rational realizations of convolutions along addition of divisors
\begin{align}\label{eq:wall-term}
\dR\mu_*\bigl(K^{u}_{\chi_0,\gamma}\boxtimes H_{\eta,\delta}\bigr)\otimes V_\tau,\qquad \chi_0+\eta=\chi,
\end{align}
where $\tau=((1,\gamma,\chi_0),(0,\delta,\eta))$ is the numerical type, $\gamma+\delta=\beta$, $e=L\cdot\gamma$, $u$ is a chamber parameter for the rank one factor, and $V_\tau$ is the one dimensional Tate object with the shift and half Tate twist specified by hyperbolic localization~\cite[Corollary~7.19, Remark~7.20]{KPS26}. On every HN stratum in the complement of $P^{t_\pm}_\chi(X,\beta)$, $e<d$ and $\eta=\lambda(d-e)$. In particular $\eta=0$ at the zero wall.
\end{prop}
\begin{proof}
Toda's surface functor $E\mapsto\dR\ol{p}_*(E(-S_\infty)[1])$ is exact on flat families in $\cA_X$ and commutes with base change \cite[Theorem~4.5.7, Corollary~4.5.9]{Tod24}. Rydh's norm construction gives the fundamental one-cycle, which is additive in exact sequences \cite[Part~IV, Theorem~7.14, Remark~7.10(ii)]{Ryd08}. Thus the cycle of a filtered object is the sum of the cycles of its factors.

Toda's HN property holds also at $\lambda=0$~\cite[Definition~4.2.5 and the following discussion]{Tod24}. For either adjacent parameter, a wall-semistable object outside the adjacent stable locus has an HN filtration with one rank one factor and rank zero factors of slope $\lambda$. Combining successive rank zero quotients by extension gives one semistable rank zero factor of that slope. Its degree is at most $d$ and its slope fixes its Euler characteristic, leaving finitely many types. As in~\cite[Section~6.2.1]{Tod24}, a type with factor classes $(1,\gamma',\chi_0')$, $(0,\delta',\eta')$ precedes one with factor classes $(1,\gamma,\chi_0)$, $(0,\delta,\eta)$ when
\begin{align*}
\gamma=\gamma'+\epsilon,\qquad \delta'=\delta+\epsilon,\qquad
\chi_0=\chi_0'+\lambda(L\cdot\epsilon),\qquad \eta'=\eta+\lambda(L\cdot\epsilon)
\end{align*}
for an effective class $\epsilon$ with $L\cdot\epsilon>0$. We take the transitive closure of this relation. The rank one degree strictly increases from a preceding type to a subsequent one, so this relation is a partial order. Choose a linear extension in which preceding types are treated first.

For a type $\tau$ with $\gamma+\delta=\beta$ and classes $(1,\gamma,\chi_0)$, $(0,\delta,\eta)$, let $p_\tau\colon\fE_\tau\to t_0\fW_\lambda$ be the map sending a short exact sequence in $\cA_X$ with these factor classes to its middle term. This map is proper and representable by Lemma~\ref{lem:flags}. After removing the closed images of all preceding types, call the remaining open substack $U_\tau$. We claim that
\begin{align}\label{eq:HN-inverse}
p_\tau^{-1}(U_\tau)=\fE^{\sss}_\tau,
\end{align}
where the right hand side consists of the flags whose factors are stable in the induced rank one chamber and pure semistable of rank zero slope $\lambda$. Indeed, the wall-semistability lemma of~\cite[Section~6.1.4]{Tod24} first makes both factors wall-semistable. If the rank one factor were unstable, its adjacent chamber HN filtration would split off a non-zero class $\epsilon$ and produce a flag of a preceding type. Conversely, stable factors give the unique adjacent chamber HN flag. Both sides are open substacks of $\fE_\tau$, so equality on geometric points proves~\eqref{eq:HN-inverse}.

Write $A$ for the subobject and $B$ for the quotient in the two-step HN flag. The restricted map is proper and radicial, and its relative tangent space is $\Hom_{\cA_X}(A,B)=0$ by the strict phase inequality between the factors. The vanishing tangent space makes the restricted map unramified; properness and radiciality then make it a closed immersion. If $\fS_\tau$ is its image and $\fU_\tau$ is the derived open substack over $U_\tau$, the Rees description gives
\begin{align}\label{eq:HN-stratum}
t_0\Filt_\tau(\fU_\tau)\xrightarrow{\ \sim\ }\fS_\tau.
\end{align}
Removing this closed image and continuing through the linear order constructs the finite stratification. The split object of any two factors in $P^u_{\chi_0}(X,\gamma)\times\cM^{\sss}_X(\delta,\eta)$ has type $\tau$, and conversely~\eqref{eq:HN-inverse} identifies every graded point in the open substack with these stability loci. Since derived open substacks are determined by their truncations, Lemma~\ref{lem:rees} yields
\begin{align*}
\Grad_\tau(\fU_\tau)=P^u_{\chi_0}(X,\gamma)\times\cM^{\sss}_X(\delta,\eta),
\end{align*}
where the stability symbols now denote derived open substacks.

At each stage of this ordering, let $i$ denote the inclusion of the closed stratum being removed and $j$ its open complement, and use the localization triangle
\begin{align*}
i_*i^!\phi_{\fW}\to\phi_{\fW}\to\dR j_*j^*\phi_{\fW}\xrightarrow{+1}.
\end{align*}
By~\eqref{eq:HN-stratum}, the closed term is $\dR\ev_{\tau*}\ev_\tau^!\phi_{\fU_\tau}$. The finite sequence of localization triangles gives a convergent spectral sequence in each perverse degree. We compute this term by hyperbolic localization~\cite[Corollary~7.19]{KPS26}, applied to the two-weight component of the attractor correspondence
\begin{align}\label{eq:attractor}
\xymatrix{\Grad(\fX) & \Filt(\fX)\ar[l]_-{\gr}\ar[r]^-{\ev} & \fX}
\end{align}
of an oriented $(-1)$-shifted symplectic stack $\fX$. The localization functor is $\gr_*\ev^!$; the same form is obtained independently in~\cite{Des25} for the correspondence of filtered and graded objects. The morphism $(\gr,\ev)\colon\Filt(\fX)\to\Grad(\fX)\times\fX$ is Lagrangian by~\cite[Corollary~5.19]{KPS26}, and $\fU_\tau$ is of finite type, quasi-separated, and has affine stabilizers. Its localized orientation is the product orientation by Lemma~\ref{lem:orient}.

Use monodromic mixed Hodge localization with half Tate twists~\cite[Remark~7.20]{KPS26}, in the form of~\cite[Section~1.2.13]{BDINP25}. Rational realization commutes with direct image and exceptional pull-back~\cite[Proposition~3.16]{Tub25} and takes this morphism to the rational localization isomorphism. Conservativity yields
\begin{align}\label{eq:wall-iso}
\dR\gr_{\tau*}\ev_\tau^!\phi_{\fU_\tau}\cong\phi_{\Grad_\tau(\fU_\tau)}\otimes V_\tau.
\end{align}
Thom-Sebastiani identifies the rational realization on the graded stack with the exterior product of the factor DT modules~\cite[Corollary~4.4]{KPS26}, and the rational realization detects membership in $\cS_\beta$.

Let $\Pi_{\Grad,\tau}$ be the product of the factor Hilbert-Chow maps to $B_\gamma\times B_\delta$. Additivity of fundamental cycles gives
\begin{align*}
\Pi_{\fW}\circ\ev_\tau=\mu\circ\Pi_{\Grad,\tau}\circ\gr_\tau.
\end{align*}
Applying the direct image to~\eqref{eq:wall-iso}, followed by rational realization and the K\"unneth formula, gives
\begin{align*}
\dR(\Pi_{\fW}\ev_\tau)_*\ev_\tau^!\phi_{\fU_\tau}
&\cong\dR\mu_*\dR\Pi_{\Grad,\tau*}\dR\gr_{\tau*}\ev^!_\tau\phi_{\fU_\tau}\\
&\cong\dR\mu_*\bigl(K^u_{\chi_0,\gamma}\boxtimes H_{\eta,\delta}\bigr)\otimes V_\tau.
\end{align*}
Note that for each of these HN strata, the rank zero factor has positive degree $d-e$ and slope $\lambda$. Thus $e<d$ and $\eta=\lambda(d-e)$; at $\lambda=0$ this gives $\eta=0$.
\end{proof}

\begin{proof}[Proof of Theorem~\ref{thm:chamber}]
At a positive wall, Lemma~\ref{lem:rank-zero} places $H_{\eta,\delta}$ in $\cG_\delta$, and the induction hypothesis places $K^u_{\chi_0,\gamma}$ in $\cS_\gamma$. Lemma~\ref{lem:serre} therefore puts every term~\eqref{eq:wall-term} in the two filtrations of Proposition~\ref{prop:wall} in $\cS_\beta$. In each perverse degree, the finite spectral sequence and the Serre property show that the wall direct image belongs to $\cS_\beta$ if and only if the corresponding open chamber term does. Applying this to both HN orders compares the adjacent chambers.

For the zero wall, choose $0<\epsilon<1/d$. Every nonzero rank zero slope $\eta/f$, with $f\le d$, has absolute value at least $1/d$, so zero is the only wall in $(-\epsilon,\epsilon)$. Proposition~\ref{prop:wall} applies with $\lambda=0$, $\eta=0$ and $\chi_0=\chi$; the terms contributed by its HN strata have the form
\begin{align}\label{eq:zero-wall-term}
\dR\mu_*\bigl(K^u_{\chi,\gamma}\boxtimes H_{0,\delta}\bigr)\otimes V_\tau,\qquad e<d.
\end{align}
The same argument compares $K^{-\epsilon}_{\chi,\beta}$ and $K^{+\epsilon}_{\chi,\beta}$. Lemma~\ref{lem:duality} reflects each negative wall to a positive one. Crossing the finitely many walls proves chamber independence. Finally, duality identifies the support assertion for $K^{0^-}_{\chi,\beta}$ with that for $K^{0^+}_{-\chi,\beta}$, and chamber independence identifies the former with the assertion for $K^{0^+}_{\chi,\beta}$.
\end{proof}

\begin{proof}[Proof of Corollary~\ref{cor:window}]
For $\chi<1-g_\beta$ the stable pair moduli space is empty by Lemma~\ref{lem:lower}, so the PT direct image is zero, and chamber independence gives the support assertion in every chamber for these Euler characteristics. The reflection gives it also for $\chi>g_\beta-1$. If $g_\beta\le 0$, the two ranges cover all integers.
\end{proof}

\section{Bounded surface pairs and dimensional reduction}\label{sec:bounded}
Dimensional reduction identifies the vanishing-cycle direct image with a dualizing complex; the HN strata give a finite filtration of the complement.

\subsection{A lower bound on HN slopes}\label{sub:slice}
We fix $S$ and $\beta$ throughout this section and put $\alpha=(1-g_\beta)/d$ as in~\eqref{eq:alpha}. The stability function is Toda's $\mu^\dagger_t$.

We use Toda's $\bM_S$ for the derived stack of pure one dimensional sheaves on $S$ and $\bM^\dagger_S$ for the derived stack of pairs. For a real number $c$, let $\bM(\chi;c)$ be the derived open substack of sheaves $E$ with $[E]=\beta$, $\chi(E)=\chi$ and $\mu_{\min}(E)\ge c$. Let $\bM^\dagger(\chi;c)$ be the derived stack of pairs $\cO_S\to E$ with $E\in\bM(\chi;c)$, and put $M^\dagger(\chi;c)=t_0\bM^\dagger(\chi;c)$.

\begin{prop}\label{prop:slice}
The stack $\bM(\chi;\alpha)$ is of finite type and has only finitely many HN types. Across all numerical classes, the condition on the smallest HN slope is closed under extensions and pure quotients. Every stable pair in class $\beta$ maps, under Toda's surface functor, into this open substack.
\end{prop}
\begin{proof}
Write the classes of the HN factors as $(\gamma_i,\chi_i)$ and put $d_i=L\cdot\gamma_i$. Then
\begin{align*}
d_i\ge 1,\qquad \sum_i\gamma_i=\beta,\qquad \sum_id_i=d,\qquad d_i\alpha\le\chi_i\le\chi-(d-d_i)\alpha,
\end{align*}
and there are only finitely many numerical types satisfying these constraints. The semistable stacks of fixed type are of finite type. For pure one dimensional sheaves, the Gieseker HN order is the order of the slopes $\chi_i/d_i$ used here. By~\cite[Theorem~5]{Nit11}, the relative Gieseker HN filtration exists uniquely over its schematic stratum and is compatible with arbitrary locally noetherian base change, and~\cite[Theorem~8]{Nit11} globalizes these loci to representable locally closed substacks. Over a product of fixed finite type factor stacks, the relative $\Rhom$ complex is perfect, so the stack parametrizing extensions of the universal factors is of finite type. Iterating this construction proves the first assertion.

For a pure sheaf, the lower HN slope bound is equivalent to that bound on every non-zero pure quotient, so it is preserved by pure quotients. For extensions, intersect a quotient with the image of the first term and pass to the induced quotient of the second; additivity of degree and Euler characteristic gives the bound. Lemma~\ref{lem:slope} treats stable pairs.
\end{proof}

HN semicontinuity makes $\bM^\dagger(\chi;\alpha)$ a derived open in Toda's quasi-smooth, locally finitely presented surface pair stack. By Proposition~\ref{prop:slice}, it is a quasi-smooth derived Artin stack of finite type, so shifted cotangent dimensional reduction applies. Toda's dual obstruction cone theorem~\cite[Theorem~4.1.3]{Tod24} identifies the classical truncation of the shifted cotangent stack with the finite type stack of framed rank one objects over this open substack. A cotangent coordinate is the morphism $E\otimes K_S^{-1}\to\fib(\cO_S\to E)[1]$ in the diagram description in Lemma~\ref{lem:derived-cotangent}. Put
\begin{align}\label{eq:Y}
Y:=M^\dagger(\chi;\alpha),\qquad \bY:=\bM^\dagger(\chi;\alpha),\qquad \wt{Y}:=t_0\bigl(T^*[-1]\bY\bigr).
\end{align}
The relevant maps are
\begin{align}\label{eq:rho}
\xymatrix{\wt{Y}\ar[r]^-{\pi} & Y\ar[r]^-{h_\chi} & B_\beta,} \qquad \rho_\chi:=h_\chi\circ\pi,
\end{align}
where $\pi$ forgets this cotangent coordinate and $h_\chi$ records the fundamental divisor of the surface sheaf. The universal pair has tangent complex $\Rhom_S([\cO_S\to F],F)$, so the Riemann-Roch theorem gives
\begin{align}\label{eq:vdim}
v:=\vdim\bY=\beta^2+\chi.
\end{align}

\subsection{Dimensional reduction}\label{sub:DR}

\begin{prop}\label{prop:DR}
For the canonical cotangent orientation, Kinjo's perverse vanishing cycle sheaf $\varphi^p_{T^*[-1]\bY}$ satisfies
\begin{align}\label{eq:Kinjo-DR}
\dR\pi_!\varphi^p_{T^*[-1]\bY}\cong\QQ_Y[v],\qquad \dR\pi_*\varphi^p_{T^*[-1]\bY}\cong\omega_Y[-v].
\end{align}
For the monodromic mixed Hodge lift $\phi_{\bY}$, normalized on a smooth surface pair base of dimension $v$ by $\QQ^H_Y[v]$, we have
\begin{align*}
\dR\pi_*\phi_{\bY}\cong\omega_Y[-v](-v).
\end{align*}
Consequently, setting $\varphi_\chi:=\phi_{\bY}[v](v)$, we have
\begin{align}\label{eq:section-reduction}
\dR\rho_{\chi*}\varphi_\chi\cong\dR h_{\chi*}\omega_Y.
\end{align}
These identities commute with restriction to an open substack defined by a stronger lower slope bound.
\end{prop}
\begin{proof}
The first rational identity is~\cite[Theorem~4.14]{Kin22}, and Verdier duality gives the second; see also~\cite[Remark~5.7]{Kin22}. For mixed Hodge modules, write $\phi^{\sd}_{\bY}$ for the self-dual vanishing cycle module in the stack-level dimensional reduction theorem~\cite[(6.2.1.4)]{BDINP25}. In that convention the isomorphism is
\begin{align*}
\dR\pi_*\phi^{\sd}_{\bY}\cong\omega_Y[-v](-v/2).
\end{align*}
Half Tate twists here are taken in the monodromic category of~\cite[Section~5.2.6]{BDINP25}. Our convention is $\phi_{\bY}=\phi^{\sd}_{\bY}(-v/2)$: on a smooth base the Verdier self-dual module is $\QQ^H_Y[v](v/2)$, so the result is $\QQ^H_Y[v]$, as in Subsection~\ref{sub:PTsheaf}. The displayed formula therefore gives $\dR\pi_*\phi_{\bY}=\omega_Y[-v](-v)$, and then $\dR\pi_*\varphi_\chi=\omega_Y$. These identities restrict to stronger lower slope bounds by naturality: the open inclusion and its shifted cotangent pull-back form a Cartesian square. Composing with $h_\chi$ gives~\eqref{eq:section-reduction} in the completed category.
\end{proof}

\begin{thm}\label{thm:bounded}
Fix $\beta>0$ and put $d=L\cdot\beta$. Assume that every rank one Chow direct image in a class $\gamma$ with $L\cdot\gamma<d$ belongs to $\cS_\gamma$, for every Euler characteristic and every generic chamber. Then for every $\chi$ we have
\begin{align}\label{eq:bounded-reduction}
K_{\chi,\beta}\in\cS_\beta\ \Longleftrightarrow\ \dR\rho_{\chi*}\varphi_\chi\in\cS_\beta.
\end{align}
\end{thm}

\subsection{The large chamber HN strata}\label{sub:large-chamber}
For a filtration in $\cA_X$, process its rank zero quotients from last to first. For each quotient $G[-1]$, let $T_0(G)$ be the maximal zero dimensional subsheaf of $G$. Replace the quotient by $(G/T_0(G))[-1]$ and enlarge the preceding subobject to the inverse image of $T_0(G)[-1]$. Continue backwards through the remaining rank zero quotients.

\begin{lem}\label{lem:last-chamber}
There is a positive generic parameter $t$ such that:
\begin{enumerate}
\item[(i)] the $\mu^\dagger_t$-stable locus in $T^*[-1]\bM^\dagger(\chi;\alpha)$ is $P_\chi(X,\beta)$;
\item[(ii)] in every non-trivial HN type in this bounded stack, all rank zero factors have positive curve degree and the unique rank one factor is a PT stable pair of degree strictly smaller than $d$;
\item[(iii)] only finitely many numerical HN types occur, and every graded factor satisfies the lower slope bound $\alpha$;
\item[(iv)] for every filtration in $\cA_X$ of any numerical HN type in (iii), consider the surface sheaves underlying its rank zero quotients, before and after the operations above. Every one dimensional Gieseker HN factor of these sheaves has slope strictly smaller than $t$.
\end{enumerate}
\end{lem}
\begin{proof}
Let $E\in\bM(\chi;\alpha)$ and let $A\subset E$ be a pure subsheaf of degree $r>0$. Let $\ol{B}:=(E/A)/T_0(E/A)$ be its quotient by the maximal zero dimensional subsheaf; then $\chi(\ol{B})\ge\alpha(d-r)$ by the slope bound; when $r=d$ we put $\ol{B}=0$ and $\chi(\ol{B})=0$. Observe that the zero dimensional kernel of $E/A\to\ol{B}$ contributes a non-negative length to the Euler characteristic of $E/A$. Thus
\begin{align*}
\mu(A)\le\frac{\chi-\alpha(d-r)}{r},\qquad b:=\max_{1\le r\le d}\frac{\chi-\alpha(d-r)}{r}<\infty.
\end{align*}
Choose a positive generic parameter $t>\max\{b,0\}$.

Order a $\mu^\dagger_t$-HN filtration by decreasing slope. Suppose that rank zero factors occurred before the unique rank one factor. Combining all such factors and applying Toda's exact surface functor~\cite[Theorem~4.5.7, Corollary~4.5.9]{Tod24}, we would obtain a subsheaf of $E$. If this subsheaf had degree zero, purity of $E$ would make it zero. If it had positive degree $f$, every positive degree constituent would have slope greater than $t$, while zero dimensional constituents could only increase the total Euler characteristic; its slope would therefore exceed $t>b$, which contradicts the bound on subsheaf slopes. Thus the rank one factor comes first. Every subsequent rank zero factor has slope smaller than $t$, and in particular cannot be zero dimensional, since a zero dimensional sheaf has slope $+\infty$. The quotient after the rank one factor is consequently an extension of pure one dimensional sheaves and gives a pure quotient $Q$ of $E$. The lower slope bound passes to pure quotients, so $\mu_{\min}(Q)\ge\alpha$ and every subsequent HN factor has positive degree and slope at least $\alpha$.

Write $(1,\gamma,\chi_0)$ for the class of the rank one factor. For a positive parameter, Toda's lower bound lemma~\cite[Lemma~4.2.7]{Tod24}, together with the filtration by powers of the ideal of the zero section in Lemma~\ref{lem:lower}, bounds $\chi_0$ below by the least structure sheaf Euler characteristic in a class at most $\gamma$. Additivity then gives
\begin{align*}
\min_{0\le\delta\le\gamma}\Bigl(-\frac{\delta\cdot(\delta+K_S)}{2}\Bigr)\le\chi_0\le\chi-(d-e)\alpha,\qquad 0\le\gamma\le\beta,\quad e=L\cdot\gamma,
\end{align*}
so there are finitely many possible rank one classes. The degrees of the rank zero factors are positive and sum to $d-e$; their Euler characteristics are bounded below by $\alpha$ times their degrees and above by additivity. Hence the full list of numerical HN types is finite.

Let $\cT$ be the finite list defined by these inequalities. It contains the HN types for every positive generic $t>\max\{b,0\}$. For each ordered type, form the stack of filtrations in $\cA_X$ with the prescribed numerical classes of Lemma~\ref{lem:flags} over $t_0(T^*[-1]\bM^\dagger(\chi;\alpha))$. There are finitely many such stacks of filtrations, each proper over the ambient stack. Their graded coherent sheaves therefore form a bounded family.

To bound the modified flags, begin with the last quotient $G$ and write $T_0(G)$ for its maximal zero dimensional subsheaf. For a fixed sufficiently positive twist, $\mathrm{length}\,T_0(G)=h^0(T_0(G)(k))\le h^0(G(k))$, so these lengths have a uniform bound. For each length $\ell$ in this finite interval, the relative Quot scheme of quotients $G\twoheadrightarrow G/T$ with Hilbert polynomial $P_G-\ell$ parametrizes all possible subsheaves $T\subset G$ of that length. Pulling $T$ back to the preceding subobject in $\cA_X$ gives another finite type family. Repeat this construction at the preceding step with its new uniform length bound. There are at most $d$ positive degree steps, since each rank zero factor has degree at least one and their degrees sum to $d-e\le d$. The induction from the last step to the first therefore produces a bounded family containing the surface sheaves of every subobject and quotient arising from these modifications.

Fix a twist $k$ and a uniform bound $N$ for $h^0(G(k))$ on the finite type family. If $T_0(G)$ is the maximal zero dimensional subsheaf and $A$ is the maximal HN subsheaf of the pure quotient $G/T_0(G)$, then
\begin{align*}
f\bigl(\mu(A)+k\bigr)=\chi(A(k))\le h^0(A(k))\le h^0(G/T_0(G)(k))\le h^0(G(k))\le N
\end{align*}
whenever the left hand side is positive, where $f$ is the positive curve degree of $A$. This gives a uniform upper bound for $\mu(A)$, and all other positive dimensional HN slopes are smaller. Denote a common upper bound by $b_{\mathrm{fl}}$.

Now choose $t>\max\{b,b_{\mathrm{fl}},0\}$ beyond the last wall of the target class and of every rank one class in $\cT$. Toda's large chamber theorem~\cite[Theorem~4.2.10(i)]{Tod24} identifies the stable objects and all stable rank one factors with PT pairs. Lemma~\ref{lem:slope} places their surface sheaves in the open substack with bound $\alpha$; the rank zero factors satisfy this bound by the quotient argument above. A non-trivial HN type contains a positive degree rank zero factor, so its rank one factor has degree less than $d$.
\end{proof}

\begin{prop}\label{prop:HN}
For the parameter chosen in Lemma~\ref{lem:last-chamber}, the classical stack $t_0(T^*[-1]\bM^\dagger(\chi;\alpha))$ has a finite schematic HN stratification. Choose a linear extension of the order defined below in which $\sigma$ is treated before $\tau$ whenever $\sigma\succ\tau$. For every type $\tau$, after the preceding types in this order have been removed, the evaluation map from the derived stack $\Filt_\tau(\fU_\tau)$ of filtrations of type $\tau$ is an isomorphism
\begin{align*}
t_0\Filt_\tau(\fU_\tau)\xrightarrow{\ \sim\ }\fS_\tau
\end{align*}
onto the scheme-theoretic HN stratum. Here $\fU_\tau$ is the derived open remaining after the preceding types are removed, and $\fS_\tau$ is its classical HN stratum. The derived stack of associated graded objects of type $\tau$ is the product of the moduli stacks of the stable rank one factor and the semistable rank zero factors. For every classical base $T$, a family over $T$ factors through $\fS_\tau$ if and only if it admits the unique relative HN filtration of type $\tau$. This characterization is preserved by arbitrary base change.
\end{prop}
\begin{proof}
Pull back Toda's stack of flat families in $\cA_X$ to the bounded pure sheaf open substack. Lemma~\ref{lem:last-chamber} gives finitely many types, all with the rank one factor first. For each type let $p_\tau\colon\fE_\tau\to t_0(T^*[-1]\bM^\dagger(\chi;\alpha))$ be its map from the stack of filtrations in $\cA_X$ with the prescribed numerical classes. The map is proper and representable by Lemma~\ref{lem:flags}.

We order the types by their polygons. Put $e=L\cdot\gamma$ and $f_j=L\cdot\delta_j$. For factor classes
\begin{align*}
(1,\gamma,\chi_0),\quad(0,\delta_1,\eta_1),\ \ldots,\ (0,\delta_r,\eta_r),\qquad t>\eta_1/f_1>\cdots>\eta_r/f_r,
\end{align*}
let $P_\tau$ be the polygon joining $(e,\chi_0)$ to the successive vertices
\begin{align*}
\Bigl(e+\sum_{i\le j}f_i,\ \chi_0+\sum_{i\le j}\eta_i\Bigr),\qquad 1\le j\le r,
\end{align*}
ending at $(d,\chi)$. Declare $\sigma\succ\tau$ if its rank one degree is smaller, or if those degrees agree and $P_\sigma$ is weakly above $P_\tau$ and strictly above somewhere. This defines a partial order.

Consider any geometric flag of type $\tau$. Starting at its last quotient $G[-1]$, absorb the maximal zero dimensional subsheaf $T\subset G$ into the preceding subobject in $\cA_X$ by taking the inverse image of $T[-1]$, and continue backwards. Under Toda's exact equivalence, each surface subsheaf is replaced by the inverse image of the zero dimensional torsion in its quotient. The Higgs field carries $T$ into $T\otimes K_S$, since its image is zero dimensional and $T$ is the maximal zero dimensional subsheaf. Thus the inverse image remains a subobject on the threefold. Absorbing a nonzero torsion subsheaf raises the corresponding polygon vertex without changing its curve degree. After these operations the quotients are pure.

Let $A_0$ be the resulting rank one subobject and write $I/A_0=Q[-1]$, where $Q$ is pure. The Gieseker HN filtration of its surface sheaf is preserved by the Higgs field, since the map from a maximal slope term to the quotient tensored with $K_S$ vanishes, and its slopes are below $t$ by Lemma~\ref{lem:last-chamber}~(iv).

If $A_0$ is stable, precede the lifted HN filtration of $Q$ by $A_0$; this is the HN filtration of $I$. The HN polygon of the whole quotient dominates every filtration polygon, with equality only for its HN flag. Removing nonzero torsion or replacing a non-semistable displayed quotient by its HN filtration therefore gives strict dominance over $P_\tau$.

If $A_0$ is unstable, the subsheaf bound $t>b$ and purity exclude rank zero factors before its unique rank one HN factor $B_0$. Its remaining factors have positive degree and slopes below $t$, so $\deg B_0<e$. The quotient $I/B_0$ is an extension of $A_0/B_0$ by $Q[-1]$, and pure rank zero sheaves with maximal HN slope below $t$ are closed under extensions. Thus $B_0$ is also the first HN factor of $I$, and its smaller degree gives $\mathrm{type}_{\mathrm{HN}}(I)\succ\tau$. In either case we have proved
\begin{align}\label{eq:type-test}
\text{if the displayed flag is not the HN flag, then } \mathrm{type}_{\mathrm{HN}}(I)\succ\tau.
\end{align}
Combining equal slope factors concatenates collinear segments and preserves strict dominance. Conversely, a stable rank one factor and pure semistable rank zero factors in strictly decreasing phases give the HN flag by uniqueness.

After the preceding types have been removed, denote the remaining open substack by $U_\tau$ and its derived inverse image by $\fU_\tau$. Equation~\eqref{eq:type-test} gives
\begin{align}\label{eq:shatz-inverse}
p_\tau^{-1}(U_\tau)=\fE^{\sss}_\tau,
\end{align}
where the right hand side imposes the stability and purity conditions on the factors. Indeed, a flag outside $\fE^{\sss}_\tau$ has a preceding HN type by~\eqref{eq:type-test}, and conversely an HN object of type $\tau$ cannot admit a flag of a preceding type $\sigma$, since the test would give $\tau=\sigma$ or $\tau\succ\sigma$. As both sides are open substacks, equality on geometric points proves~\eqref{eq:shatz-inverse}.

The restricted flag map is proper and radicial by HN uniqueness. Linearizing the equations that require the diagram arrows to preserve the filtration and successively forgetting its steps filters the relative tangent space by the groups $\Hom_{\cA_X}(A_i,A_j)$ for $i<j$, where the $A_i$ are the ordered graded factors. These groups vanish by the strict HN phase inequalities. The map is therefore unramified, hence a closed immersion by~\cite[Lemma~4]{Nit11}, checked after a smooth atlas. Its image $\fS_\tau$ carries the universal relative HN flag. Over any base, a family with such a flag lifts uniquely to $\fE^{\sss}_\tau$, and pull-back of the universal flag recovers the given one. These constructions prove the stated characterization and its compatibility with base change. Continuing through the finite order gives the schematic stratification. The Rees equivalence identifies the flag stack with $t_0\Filt_\tau(\fU_\tau)$.

Lemma~\ref{lem:rees} identifies the derived stack of graded objects of type $\tau$ with the product of the derived factor stacks, and~\eqref{eq:shatz-inverse} restricts it to the open substacks where the factors satisfy the required stability conditions. The rank one factor satisfies the lower slope bound by Lemma~\ref{lem:slope}; the rank zero factors do so because the lower slope bound passes to pure quotients.
\end{proof}

\begin{prop}\label{prop:filtration}
Inside $t_0(T^*[-1]\bM^\dagger(\chi;\alpha))$, the complement of the stable pair open substack $P_\chi(X,\beta)$ has a finite filtration by closed unions of HN strata. After the direct image to $B_\beta$ and rational realization, every successive term, in each bounded range of perverse cohomological degrees, is a finite direct sum of direct summands of the rational realization of
\begin{align}\label{eq:large-chamber-term}
\dR\mu_{\tau*}\bigl(K_{\chi_0,\gamma}\boxtimes H_{\eta_1,\delta_1}\boxtimes\cdots\boxtimes H_{\eta_r,\delta_r}\bigr)\otimes V_\tau.
\end{align}
Here $e=L\cdot\gamma$, $\delta_j>0$, $f_j=L\cdot\delta_j$, $\mu_\tau\colon B_\gamma\times\prod_jB_{\delta_j}\to B_\beta$ is the addition of divisors, and $V_\tau$ is a bounded below graded mixed Hodge structure with finite dimensional pieces, recording the localization shifts and stabilizer cohomology. For every HN stratum in this complement, $r>0$ and
\begin{align}\label{eq:large-chamber-numerics}
\beta=\gamma+\sum_{j=1}^r\delta_j,\qquad d=e+\sum_{j=1}^rf_j,\qquad \chi=\chi_0+\sum_{j=1}^r\eta_j.
\end{align}
In particular $e<d$. The case $e=0$ means the object $\cO_{\ol{X}}$ with its fixed framing, with $\chi_0=0$ and coefficient $\QQ_{B_0}$.
\end{prop}
\begin{proof}
Lemma~\ref{lem:last-chamber} and Proposition~\ref{prop:HN} give the finite schematic filtration and the products of graded factors. Extension closure keeps these filtrations in the prescribed open substack.

Apply the localization construction of Proposition~\ref{prop:wall} to each derived stack $\Filt_\tau(\fU_\tau)$ of filtrations, using the associative orientation of Lemma~\ref{lem:orient}. The mixed Hodge localization isomorphism~\eqref{eq:wall-iso}, Thom-Sebastiani~\cite[Corollary~4.4]{KPS26}, and additivity of cycles give~\eqref{eq:large-chamber-term} by the K\"unneth formula. Factors of equal slope are treated as one semistable factor.

For the numerics, in the heart one has $[\cO_X\to F]=[\cO_X\to F_0]+\sum_j[Q_j[-1]]$. Comparing degrees and Euler characteristics, with $[Q_j[-1]]=-[Q_j]$, gives~\eqref{eq:large-chamber-numerics}. Every unstable object has a semistable rank zero HN factor with nonzero curve class, so $e<d$.
\end{proof}

\begin{proof}[Proof of Theorem~\ref{thm:bounded}]
Let $j\colon P_\chi(X,\beta)\hookrightarrow\wt Y$ be the stable open substack of the shifted cotangent truncation defined in~\eqref{eq:Y}. If $i$ is the inclusion of its complement, localization gives the triangle
\begin{align*}
i_*i^!\varphi_\chi\to\varphi_\chi\to\dR j_*j^*\varphi_\chi\xrightarrow{+1}.
\end{align*}
Theorem~\ref{thm:oriented} and Proposition~\ref{prop:DR} give $j^*\varphi_\chi\cong\phi^{\PT}_{\chi,\beta}[v](v)$, where $v=\beta^2+\chi$. Shifts and Tate twists preserve $\cS_\beta$.

Proposition~\ref{prop:filtration} expresses the successive terms of the complement as convolutions~\eqref{eq:large-chamber-term}. Their PT factors have degree $e<d$ and satisfy the induction hypothesis; their semistable sheaf factors belong to $\cG$ by Lemma~\ref{lem:rank-zero}. Lemma~\ref{lem:serre} therefore places every convolution in $\cS_\beta$. Since the filtration is finite and $\cS_\beta$ is closed under cones, the direct image of the complement lies in $\cS_\beta$. Now take the direct image of the localization triangle to obtain~\eqref{eq:bounded-reduction}.
\end{proof}

\section{Hecke correspondences}\label{sec:hecke}
Point correspondences on Hilbert schemes of a smooth surface give the Heisenberg actions of Grojnowski and Nakajima~\cite{Gro96,Nak97}. For a fixed locally planar curve, related operators were constructed by Rennemo and Kivinen~\cite{Ren18,Kiv19}. Here point modifications on bounded pair stacks give operators over the Chow variety.

\subsection{Point-modification operators}\label{sub:operators}
Fix $S$ and a class $\beta$ with $g_\beta>0$, and put $d=L\cdot\beta$. We use the bounded pair stacks of Subsection~\ref{sub:slice}. Let $h_\chi(c)\colon M^\dagger(\chi;c)\to B_\beta$ be their Hilbert-Chow maps, so $h_\chi(\alpha)=h_\chi$ in~\eqref{eq:rho}. Set
\begin{align}\label{eq:K-chi-c}
K_\chi(c):=\dR h_\chi(c)_*\,\omega_{M^\dagger(\chi;c)}\langle-\chi\rangle,\qquad \alpha=\frac{1-g_\beta}{d}.
\end{align}
With the twist $\langle-\chi\rangle$, both point-modification operators have degree zero. Proposition~\ref{prop:DR} gives
\begin{align*}
K_\chi(\alpha)\cong\bigl(\dR\rho_{\chi*}\phi_{\bY}\bigr)[\beta^2-\chi](\beta^2).
\end{align*}

For $c\le c'$, restriction along the open inclusion $M^\dagger(\chi;c')\hookrightarrow M^\dagger(\chi;c)$ gives a morphism
\begin{align}\label{eq:res}
\res_\chi(c,c')\colon K_\chi(c)\to K_\chi(c').
\end{align}
Proposition~\ref{prop:slice} applies to any fixed lower bound, so these stacks are of finite type. The notation will also be used over open subsets of $B_\beta$.

Choose a smooth irreducible divisor $D\subset S$ in $|rL|$ with $r>d$. A \emph{point modification} is an exact sequence of pure sheaves
\begin{align}\label{eq:point-modification}
0\to F_-\to F_+\to\cO_x\otimes\cL\to 0
\end{align}
with compatible sections, where $\cL$ is a line bundle on the parameter scheme. Note that the sequence is non-split, since $F_+$ is pure. The line $\cL$ records the scalar automorphisms of the quotient. Let
\begin{align}\label{eq:hecke-diagram}
\xymatrix{ & \fH_\chi\ar[dl]_-{p_-}\ar[d]^-{q}\ar[dr]^-{p_+} & \\ \bM^\dagger(\chi;c) & S & \bM^\dagger(\chi+1;c')}
\end{align}
be the lower pair, point and upper pair maps from the derived Hecke stack. It is defined as a derived fiber product in Lemma~\ref{lem:diagrams}. On the stack of all pure pairs, Toda's two projectivization descriptions~\cite[Lemmas~5.6, 5.7]{Tod20} make the pair projections representable, proper and quasi-smooth, with virtual relative dimensions $1$ and $0$, respectively. Restricting the marked point to $D$ changes the latter dimension to $-1$. The slope bounds satisfy
\begin{align}\label{eq:hecke-slopes}
\mu_{\min}(F_-)\ge c\ \Longrightarrow\ \mu_{\min}(F_+)\ge c,\qquad
\mu_{\min}(F_+)\ge c\ \Longrightarrow\ \mu_{\min}(F_-)\ge c-1.
\end{align}
Indeed, for the first implication, the image of $F_-$ in any pure quotient of $F_+$ is a pure quotient of $F_-$ with colength at most one, so its slope is at least that of the corresponding quotient of $F_-$. For the second, extend a pure quotient of $F_-$ across the modification and remove its zero dimensional torsion; a degree $e$ quotient obtained this way has Euler characteristic at least $ce-1$, and hence slope at least $c-1$.

For creation, we require the upper pair to lie in $M^\dagger(\chi+1;\alpha)$; its lower pair then lies in $M^\dagger(\chi;\alpha-1)$. For removal, the lower pair lies in $M^\dagger(\chi;c)$ with $c=\alpha$ or $\alpha-1$, the removed point lies on $D$, and the upper pair then lies in $M^\dagger(\chi+1;c)$. Write $\fH^E_\chi$ and $\fH^F_\chi(c)$ for these two derived correspondences. In each case the bound on the source is automatic from the bound on the target by~\eqref{eq:hecke-slopes}. The target projection is therefore a base change of the corresponding proper projection on the stack of all pure pairs, while the source projection is quasi-smooth. The purity map for the source and proper trace for the target give the degree zero maps
\begin{align}\label{eq:EF}
\begin{aligned}
E_\chi&\colon K_\chi(\alpha-1)\to K_{\chi+1}(\alpha) &&\text{(add a point)},\\
F_\chi(c)&\colon K_{\chi+1}(c)\to K_\chi(c) &&\text{(remove a point on $D$)}.
\end{aligned}
\end{align}
Before normalizing by $\langle-\chi\rangle$, the target twists are $\langle-1\rangle$ for $E_\chi$ and $\langle 1\rangle$ for $F_\chi(c)$.

For a quasi-smooth map $f\colon Z\to Y$ of virtual dimension $v$, we use the purity transformation
\begin{align*}
\eta_f\colon f^*\omega_Y\langle v\rangle\to f^!\omega_Y=\omega_Z.
\end{align*}
Throughout this section we use the functoriality and derived base change of~\cite[Theorems~3.12, 3.13]{Kha19}, followed by the Hodge realization of~\cite{Tub25}. The Gysin morphisms use the complex orientation, and the trace of a finite \'etale map has its usual degree.

Removal commutes with restriction of the slope bound:
\begin{align}\label{eq:F-restriction}
\res_\chi(\alpha-1,\alpha)\circ F_\chi(\alpha-1)=F_\chi(\alpha)\circ\res_{\chi+1}(\alpha-1,\alpha).
\end{align}
Both sides use the correspondence of upper modifications of lower pairs with bound $\alpha$; the upper pairs retain this bound by~\eqref{eq:hecke-slopes}.

\subsection{Derived Hecke diagrams and purity}\label{sub:diagrams}
Write $\bM_S(0,1)$ for the derived stack of length one sheaves on $S$, and set
\begin{align*}
\cZ:=S\times B\Gm.
\end{align*}
The natural map $\cZ\to\bM_S(0,1)$ sends, on a derived affine scheme $T$, a point $x\colon T\to S$ and a line bundle $\cL$ to $\cO_{\Gamma_x}\otimes\cL$ on $S\times T$, where $\Gamma_x$ is the graph of $x$.

\begin{lem}\label{lem:diagrams}
Let $\fE_\chi$ be the derived stack of fiber sequences $F_-\to F_+\to Q$ of perfect complexes on $S$, with $Q$ of length one, together with a morphism $\cO_S\to F_-$. Require both $F_-$ and $F_+$ to be flat pure sheaves of their prescribed classes on classical truncations. Then Toda's derived Hecke stack is
\begin{align}\label{eq:hecke-basechange}
\fH_\chi=\fE_\chi\times_{\bM_S(0,1)}\cZ.
\end{align}
The data include the identification of $Q$ with the marked point sheaf, the induced upper section and its null-homotopy in $Q$. For points on $D$, replace $\cZ$ by $D\times B\Gm$ in the derived fiber product.
\end{lem}
\begin{proof}
Toda defines the stack of extensions as the stack of fiber sequences of perfect complexes whose classical restrictions are flat exact sequences \cite[Section~3.1, (3.4)--(3.6)]{Tod20}, and his framed extension stack is its derived base change by the stack of lower sections \cite[Section~4.2, (4.10)--(4.11)]{Tod20}. Compose a lower section with the fiber sequence to obtain the upper section and its canonical null-homotopy in the cofiber. Conversely, an upper section with this null-homotopy is a map to the fiber. These are the inverse maps in the equivalence of mapping spaces for a fiber sequence over a derived base.

The top Cartesian square in~\cite[(5.5)]{Tod20} is~\eqref{eq:hecke-basechange}. In that diagram the rigidification of $\bM_S(0,1)$ is $\Spec_S\Sym(K_S[1])$, and the square below it pulls this rigidification back to $S$. Its pull-back in the point sheaf stack is $S\times B\Gm$, with the graph family just described.

Fixing the lower pair and $x$, the non-zero extension maps are maps $\cO_x\to F_-[1]$ modulo the scalar automorphisms of $\cL$, and relative Serre duality identifies the dual of their mapping complex with $F_-^\vee|_x\otimes K_S[1]$. Fixing the upper pair and $x$, the quotient map together with the null-homotopy of its section is a map $I_+[1]\to\cO_x$, where $I_+=[\cO_S\to F_+]$; its non-zero locus modulo scalars is the projectivization of $I_+[1]|_x$. These descriptions give the maps from the projectivizations of~\cite[Lemmas~5.6, 5.7]{Tod20} to the stack of marked exact diagrams. Both perfect complexes have fiber amplitude $[-1,0]$ and ranks $0$ and $-1$; their projectivizations consequently have virtual relative dimensions $-1$ and $-2$ over the pair stack times $S$, which gives the dimensions stated above.
\end{proof}

In a stable $\infty$-category, a square is \emph{bicartesian} if it is both a homotopy pull-back and a homotopy push-out. Its opposite edges have canonically equivalent cofibers, and a span or cospan has a contractible space of bicartesian completions.

\begin{lem}\label{lem:convolution}
Suppose that
\begin{align*}
\xymatrix{\fX & \fZ_1\ar[l]_-{s_1}\ar[r]^-{t_1} & \fY,} \qquad \xymatrix{\fY & \fZ_2\ar[l]_-{s_2}\ar[r]^-{t_2} & \fW}
\end{align*}
are derived correspondences whose classical truncations have compatible maps to a base $B$, with $s_i$ quasi-smooth of virtual dimension $v_i$ and $t_i$ representable and proper. On the classical truncations, use the purity maps for dualizing complexes and the proper traces. Then their composite is the correspondence given by the purity of
\begin{align*}
s:=s_1\circ\mathrm{pr}_1\colon\fZ_1\times_{\fY}\fZ_2\to\fX
\end{align*}
and the trace of $t:=t_2\circ\mathrm{pr}_2$. The source map has virtual relative dimension $v_1+v_2$. An equivalence of the derived fiber products over $\fX$ and $\fW$ identifies these correspondences, including their purity morphisms.
\end{lem}
\begin{proof}
Put $\fZ:=\fZ_1\times_{\fY}\fZ_2$. The first projection is the derived base change of $s_2$ and the second the base change of the proper map $t_1$, so $s$ is quasi-smooth of virtual dimension $v_1+v_2$ and $t$ is proper. The purity transformation is functorial for compositions and compatible with derived base change \cite[Remark~3.8, Theorems~3.12, 3.13]{Kha19}; applied to $s=s_1\circ\mathrm{pr}_1$, with $\mathrm{pr}_1$ the base change of $s_2$, this expresses $\eta_s$ as the composite of $\mathrm{pr}_1^*\eta_{s_1}$ and the pull-back of $\eta_{s_2}$. Proper base change for $t_1$ moves the first proper direct image past $s_2^*$, and the two proper counits compose to the counit for $t$. An equivalence over $\fX$ and $\fW$ identifies the relative cotangent complexes and their purity transformations by naturality.
\end{proof}

The two compositions in the commutator are represented by the derived fiber products
\begin{align}\label{eq:hecke-composites}
\begin{split}
\bZ_{FE}&:=\fH^E_\chi\times_{\bM^\dagger(\chi+1;\alpha)}\fH^F_\chi(\alpha),\\
\bZ_{EF}&:=\fH^F_{\chi-1}(\alpha-1)\times_{\bM^\dagger(\chi-1;\alpha-1)}\fH^E_{\chi-1}.
\end{split}
\end{align}
The first fiber product uses the two upper-pair maps, and the second uses the two lower-pair maps. Both are correspondences from $\bM^\dagger(\chi;\alpha-1)$ to $\bM^\dagger(\chi;\alpha)$. We write $Z_{FE}:=t_0\bZ_{FE}$ and $Z_{EF}:=t_0\bZ_{EF}$ for their classical truncations.

\subsection{A relative Quot construction}\label{sub:quot}
Choose a pencil in $|L|$ generated by two general members, and let $Z$ be its finite base locus; it consists of one point on $\PP^2$ and two points on $\mathbb{F}_0$. We work over the open subset
\begin{align}\label{eq:B-Z}
B_{\beta,Z}:=\{C\in B_\beta : C\cap Z=\emptyset\}.
\end{align}
The pencil defines $\pi\colon S\setminus Z\to\PP^1$, and we denote its member over $t$ by $L_t$. Observe that every component of every $L_t$ meets $Z$: intersect it with another pencil member and use the ampleness of $L$. A curve avoiding $Z$ therefore has no component in common with any $L_t$. The equation of $L_t$ is a non-zero-divisor on every pure sheaf with such a support. The equation of $D\in|rL|$ has the same property, since $r>d$ prevents $D$ from being a component of the support.

Put $\ell=(r+1)d$ and $c=\alpha-r-1$. Let $(F,s)$ be the final pair in $M^\dagger(\chi;\alpha)$, and mark points $x\in C$ and $y\in C\cap D$, where $C$ is its fundamental divisor. The stack of these marked pairs is proper and representable over $M^\dagger(\chi;\alpha)$; the second marking is finite flat of degree $rd$. Define the relative effective Cartier divisor
\begin{align}\label{eq:moving-divisor}
H_x:=D+L_{\pi(x)}.
\end{align}
Since $y\in D$ and $x\in L_{\pi(x)}$, the divisor $H_x\in|(r+1)L|$ contains both points and its defining section kills both length one quotients. We use its relative line bundle and canonical section for all twists below.

To locate the markings scheme-theoretically on $C$, take a relative length one locally free resolution of the flat family of pure sheaves over a classical parameter scheme. The determinant cuts out the zeroth Fitting divisor and annihilates the sheaf. A length one modification preserves this determinant divisor, since the quotient on the graph of the marked point has its canonical trivial determinant from its relative Koszul resolution. Thus both point quotients in either composition are annihilated by the same relative equation of $C$.

\begin{prop}\label{prop:quot}
There is a representable projective morphism $q_2\colon T_\chi\to M^\dagger(\chi;\alpha)$ parametrizing the markings above and subsheaves
\begin{align}\label{eq:common-subsheaf}
F(-H_x)\subset G\subset F(H_x),\qquad \mathrm{length}\bigl(F(H_x)/G\bigr)=\ell,\qquad \sigma_{H_x}s\in H^0(G).
\end{align}
Here $\sigma_{H_x}\colon F\hookrightarrow F(H_x)$ is the canonical inclusion. The induced pair $(G,\sigma_{H_x}s)$ defines a map $q_1\colon T_\chi\to M^\dagger(\chi;c)$ over $B_{\beta,Z}$. The classical truncations $Z_{FE}$ and $Z_{EF}$ admit representable proper maps $a_{FE}\colon Z_{FE}\to T_\chi$ and $a_{EF}\colon Z_{EF}\to T_\chi$, compatible with their initial and final pairs.
\end{prop}
\begin{proof}
Fiberwise purity and the non-zero-divisor property make $F(H_x)/F(-H_x)$ finite flat of length $2\ell$ over the stack of pairs with these marked points. Take its relative Quot scheme of quotients of length $\ell$ and impose the closed condition that the image of $\sigma_{H_x}s$ in the universal quotient vanishes. This is $T_\chi$. Its universal kernel $G\subset F(H_x)$ is flat, has pure fibers, and has class $\beta$ and Euler characteristic $\chi$. Since $G/F(-H_x)$ is zero dimensional, the first implication in~\eqref{eq:hecke-slopes}, applied with a zero dimensional quotient of arbitrary length, gives $\mu_{\min}(G)\ge\alpha-r-1=c$. The relative Quot scheme is projective, and the marked parameters are proper over the final pair stack, which proves the assertion about $q_2$.

Write a point of $Z_{FE}$ as
\begin{align*}
A\hookrightarrow U\hookleftarrow F,\qquad U/A=\cO_x\otimes\cL_x,\quad U/F=\cO_y\otimes\cL_y;
\end{align*}
here $F$ is the final pair, $U$ is obtained by adding the point $y$ on $D$, and $A$ by removing $x$. Multiplication by $\sigma_{H_x}$ annihilates both quotients, so the map $U\to U(H_x)$ factors uniquely through $F(H_x)$. The composite $A\to U\to F(H_x)$ is injective and its image is the subsheaf $G$. Multiplication also lifts $F(-H_x)\to U(-H_x)$ to $A$, so the canonical subsheaf $F(-H_x)$ is contained in $G$. Compatibility of the pair sections gives the last condition in~\eqref{eq:common-subsheaf}.

For $Z_{EF}$, write a point as
\begin{align*}
A\hookleftarrow J\hookrightarrow F,\qquad A/J=\cO_y\otimes\cL_y,\quad F/J=\cO_x\otimes\cL_x;
\end{align*}
now $J$ is obtained from $F$ by removing $x$, and $A$ by adding $y$. Multiplication by $\sigma_{H_x}$ lifts $A\to A(H_x)$ uniquely to $J(H_x)$, and composing with $J(H_x)\to F(H_x)$ again gives an injection whose image we call $G$. The lift $F(-H_x)\to J$ proves the other containment, and the sections are preserved as before.

Over a classical base these maps are fiberwise injective by the non-zero-divisor property. Their sources and targets are flat, so the cokernels are flat, of length $\ell$ by the Euler characteristics. The unique factorizations through the kernels pull back under base change, defining the required morphisms of stacks.

Both $Z_{FE}$ and $Z_{EF}$ are proper over the final pair stack. In the first order, one chooses an upper modification of $F$ at a point of $D$, then a lower modification at an unrestricted point; all such upper pairs have bound $\alpha$ and all such lower pairs have bound $\alpha-1$. In the second order, one chooses a lower modification of $F$, then an upper modification at a point of $D$; the bounds are again automatic by~\eqref{eq:hecke-slopes}. Since $T_\chi$ is separated over $M^\dagger(\chi;\alpha)$, each map to $T_\chi$ factors through its closed graph and the base change of the proper map to $M^\dagger(\chi;\alpha)$. It is representable and proper.
\end{proof}

Let $i\colon I_\chi\hookrightarrow T_\chi$ be the locus
\begin{align}\label{eq:exceptional-locus}
G=\sigma_{H_x}F\subset F(H_x),\qquad x=y.
\end{align}
The first equality is an equality of embedded subsheaves, with the indicated induced section; it defines a closed section of the relative Quot scheme over the stack of pairs with these marked points. The second equality is closed as well. If $\cC_{\beta,D}:=\cC_\beta\cap(D\times B_{\beta,Z})$ denotes the locus of a point of the universal curve lying on $D$, then
\begin{align}\label{eq:finite-exceptional}
I_\chi=M^\dagger(\chi;\alpha)\times_{B_{\beta,Z}}\cC_{\beta,D},\qquad q\colon I_\chi\to M^\dagger(\chi;\alpha)
\end{align}
is finite flat of degree $rd$: the only datum left is the choice of the common point on $C\cap D$. For the open inclusion $j\colon M^\dagger(\chi;\alpha)\hookrightarrow M^\dagger(\chi;c)$, one has $q_1\circ i=j\circ q$ and $q_2\circ i=q$.

\begin{lem}\label{lem:exchange}
Let $\bZ^\circ_{FE}\subset\bZ_{FE}$ and $\bZ^\circ_{EF}\subset\bZ_{EF}$ be the derived open substacks whose classical truncations are $a_{FE}^{-1}(T_\chi\setminus I_\chi)$ and $a_{EF}^{-1}(T_\chi\setminus I_\chi)$. There is an equivalence $\bZ^\circ_{FE}\simeq\bZ^\circ_{EF}$ over the initial and final pair stacks, identifying the quasi-smooth purity morphisms. Its classical truncation commutes with the maps to $T_\chi\setminus I_\chi$.
\end{lem}
\begin{proof}
Start with $A\hookrightarrow U\hookleftarrow F$. If the images of $A$ and $F$ in $U$ differ, then their sum is $U$: two distinct colength one submodules generate, also when $x=y$. Their intersection $J$ then has colength one in each of $A$ and $F$, with the points $x$ and $y$, and we obtain $A\hookleftarrow J\hookrightarrow F$. If instead the images coincide, the induced rational identification between $A$ and $F$ extends to an identification of their images; this forces $G=\sigma_{H_x}F$ and $x=y$, the equations defining $I_\chi$.

Conversely, start with $A\hookleftarrow J\hookrightarrow F$ and form the push-out $U=(A\oplus F)/J$. Its kernel and quotient sequences show that $U$ is an extension of $A$ by the length one quotient at $x$, and of $F$ by the length one quotient at $y$. Any zero dimensional torsion $T_0\subset U$ injects into both point quotients, since $A$ and $F$ are pure, so $T_0\neq 0$ forces $x=y$ and $\mathrm{length}\,T_0=1$. In that case both maps $A\to U/T_0$ and $F\to U/T_0$ are injective with zero dimensional cokernel of Euler characteristic zero, hence isomorphisms; the identification they induce is the one defined away from $x$ by the common subobject $J$, and it gives $G=\sigma_{H_x}F$. So outside $I_\chi$ the push-out is pure, and we recover the first diagram.

Surjectivity of $A\oplus F\to U$ is detected on fibers, and its kernel is flat because source and target are flat. The push-out in the opposite direction is flat by either of its extension sequences, and its fibers are pure by the preceding argument. The new lower pair $J$ satisfies the bound $\alpha-1$, being a lower modification of $F$, and the new upper pair $U$ satisfies the bound $\alpha$, being an upper modification of $F$. Over a derived affine base $T$, work in the stable category of arrows of $\mathrm{Perf}(S\times T)$ and consider the square of target sheaves
\begin{align}\label{eq:bicartesian}
\xymatrix{J\ar[r]\ar[d] & A\ar[d]\\ F\ar[r] & U.}
\end{align}
The source square is constant at $\cO_{S\times T}$, with identity arrows. For the first composition, complete the cospan $A\to U\leftarrow F$ by homotopy pull-back. For the second, complete the span $A\leftarrow J\to F$ by homotopy push-out. These completions are inverse to each other because~\eqref{eq:bicartesian} is bicartesian, and the quotient objects satisfy the canonical equivalences
\begin{align*}
\cofib(A\to U)\simeq\cofib(J\to F),\qquad \cofib(F\to U)\simeq\cofib(J\to A).
\end{align*}
The specified identifications of the cofibers with $\cO_{\Gamma_x}\otimes\cL_x$ and $\cO_{\Gamma_y}\otimes\cL_y$ give an equivalence over $\cZ\times(D\times B\Gm)$. Equation~\eqref{eq:hecke-basechange} identifies these diagrams with the derived Hecke stacks. The constant source square completes to $\cO_{S\times T}$; hence the sections, their compatibility homotopies and the null-homotopies in both quotients pass through the equivalence.

The fourth vertex is a perfect complex whose derived geometric fibers, by the classical calculation, are pure sheaves in degree zero with the required slope bound. Additivity in the bicartesian square, whose two cofibers have length one, gives its class $\beta$ and prescribed Euler characteristic. The relative Tor-amplitude criterion makes its restriction to $t_0(T)$ a flat family of these sheaves. It therefore belongs to the derived open substack defining Toda's pure sheaf stack. The completions have contractible spaces of choices, so they give inverse equivalences of derived stacks over the initial and final pair stacks.

On a classical base, the map $A\to F(H_x)$ in either construction is the unique lift to $F(H_x)$ of $A\to U\xrightarrow{\sigma_{H_x}}U(H_x)$. The maps to $T_\chi$ consequently agree on the embedded subsheaf $G$, its lower subsheaf, section and markings.

Each convolution has source map of virtual relative dimension zero. Since the equivalence lies over the source, it identifies the relative cotangent complexes and their purity transformations. By Lemma~\ref{lem:convolution}, these are the convolution morphisms. Their proper push-forwards to $T_\chi\setminus I_\chi$ are equal.
\end{proof}

\subsection{Connectedness of the pair stacks}\label{sub:connected}
Since $g_\beta>0$, the class $\beta$ is ample on either surface. A surface pair is stable if its section has zero dimensional cokernel.

\begin{prop}\label{prop:connected}
Let $\beta$ be ample and $M^\dagger(\chi;c)$ non-empty. Then $M^\dagger(\chi;c)$ is connected, also after restricting to divisors avoiding a prescribed finite set $Z\subset S$. If $\chi\ge 1-g_\beta$, it contains stable surface pairs on smooth curves. In particular
\begin{align*}
\End\bigl(\omega_{M^\dagger(\chi;c)}\bigr)=\QQ
\end{align*}
in the derived category of mixed Hodge modules, also on the indicated open substack.
\end{prop}
\begin{proof}
The Rees construction applied to the HN filtration connects a pure sheaf to the direct sum of its semistable factors while preserving the lower slope bound. For an ample factor class, the semistable stack is irreducible with a dense locus of line bundles on smooth curves~\cite[Theorem~2.3, Proposition~2.10, Section~2.6]{MS23}. On $\mathbb{F}_0$ the remaining factor classes are multiple rulings, and Remark~\ref{rmk:ruling-deform} deforms them to line bundles on distinct fibers. We can therefore move all factors to line bundles on distinct smooth components, with transverse intersections and no triple points; they may also be chosen to avoid $Z$. When several factors have the same class, we use that many distinct smooth members of that linear system, which exist because every effective linear system on either surface is base point free with smooth general member, and a general pair of members meets transversely.

On $\mathbb{F}_0$, connectedness of the union follows from $a,b>0$: $H^1(S,\cO_S(-a,-b))=0$ by the K\"unneth formula, so every divisor in the class has $H^0(\cO_C)=\CC$. On $\PP^2$ it follows from the corresponding exact sequence for $\cO_S(-d)$. So the intersection graph of the components is connected.

We now connect this direct sum to a line bundle on the nodal union by deforming the extensions at the intersection points. For two distinct components, the sheaf $\cHom$ between their line bundles is zero and the sheaf $\cExt^1$ is one dimensional at each intersection point, so the global $\Ext^1$ is the direct sum of those local spaces. At each intersection point, choose local equations $u=0$ and $v=0$ for the two components. The extension with parameter $t$ is represented by
\begin{align*}
\coker\begin{pmatrix}u & t\\ 0 & v\end{pmatrix}.
\end{align*}
At $t=0$ this is the direct sum; at $t\neq 0$ it is a line bundle on $uv=0$. Adding components one at a time along the connected intersection graph gives a line bundle on the whole nodal union, and every such extension retains the lower HN bound.

The obstruction to extending the line bundle to a smoothing lies in $H^2(C,\cO_C)=0$. The relative Picard stack is smooth there; after restricting to an open neighborhood, the slope bound and avoidance of $Z$ persist. Its fibers over the irreducible open subset of smooth curves are connected. Thus all the original sheaves are connected to this locus.

For pairs, scale the section to zero. If $n=\chi+g_\beta-1\ge 0$, an effective divisor $W$ of degree $n$ on a smooth curve $C$ gives the stable surface pair $(\cO_C(W),1)$; its slope is $\chi/d\ge c$ whenever the bounded stack is non-empty.

Finally, write $P$ for $M^\dagger(\chi;c)$ or its indicated open substack. Verdier duality identifies $\End(\omega_P)$ with $\End(\QQ_P)$, and the latter is $H^0(P,\QQ)=\QQ$ by connectedness. Checking on smooth atlases before completion gives the equality in mixed Hodge modules as well.
\end{proof}

\subsection{Localization and the commutator}\label{sub:scalar}
Lemma~\ref{lem:exchange} identifies the two correspondence morphisms away from $I_\chi$. Their difference is therefore supported on $I_\chi$.

\begin{lem}\label{lem:supported}
Suppose that $j\colon Y\hookrightarrow X$ is open, with compatible maps $g_X\colon X\to B$ and $g_Y\colon Y\to B$, and that $q_1\colon T\to X$, $q_2\colon T\to Y$ are maps over $B$ with $q_2$ proper. Let $i\colon I\hookrightarrow T$ be closed, with $q_1\circ i=j\circ q$, $q_2\circ i=q$ and $q\colon I\to Y$ proper. Then a correspondence
\begin{align*}
\zeta\colon q_1^*\omega_X\to q_2^!\omega_Y
\end{align*}
which vanishes on $T\setminus I$ induces on direct images a morphism of the form
\begin{align*}
\dR g_{Y*}u\circ\res(X,Y),\qquad u\in\End(\omega_Y),
\end{align*}
where $\res(X,Y)$ is the open restriction.
\end{lem}
\begin{proof}
Apply $\Hom(q_1^*\omega_X,-)$ to the closed-open localization triangle for $q_2^!\omega_Y$. Since the restriction of $\zeta$ to the open complement is zero, exactness supplies a lift $\zeta_I\colon i^*q_1^*\omega_X=q^*\omega_Y\to i^!q_2^!\omega_Y=q^!\omega_Y$. Define $u$ as the composite
\begin{align*}
\omega_Y\to\dR q_*q^*\omega_Y\xrightarrow{\ \dR q_*\zeta_I\ }\dR q_*q^!\omega_Y\to\omega_Y,
\end{align*}
where the first arrow is the unit and the last is the proper counit. By the two adjunctions, $\zeta$ is the push-forward of $\zeta_I$ along $i$. Applying the direct image correspondence and using $q_1\circ i=j\circ q$ identifies its unit with the open restriction followed by the unit for $q$, and the counits compose by $q_2\circ i=q$.
\end{proof}

\begin{thm}\label{thm:relation}
Put $c=\alpha-r-1$. For $\chi\ge 1-g_\beta$, on $B_{\beta,Z}$ the operators in~\eqref{eq:EF} satisfy
\begin{align}\label{eq:mixed-relation}
\bigl(F_\chi(\alpha)E_\chi-E_{\chi-1}F_{\chi-1}(\alpha-1)\bigr)\circ\res_\chi(c,\alpha-1)=rd\cdot\res_\chi(c,\alpha).
\end{align}
The equality holds in $\widehat{D}\MHM(B_{\beta,Z})$.
\end{thm}
\begin{proof}
We calculate with dualizing complexes and restore the common twist $\langle-\chi\rangle$ at the end.

Let $a_{FE}\colon Z_{FE}\to T_\chi$ and $a_{EF}\colon Z_{EF}\to T_\chi$ be the proper maps of Proposition~\ref{prop:quot}. The source maps from $\bZ_{FE}$ and $\bZ_{EF}$ to $\bM^\dagger(\chi;c)$ are quasi-smooth of virtual dimension zero. Pushing their purity morphisms forward to $T_\chi$ gives two elements $\zeta_{FE},\zeta_{EF}\in\Hom(q_1^*\omega_{M^\dagger(\chi;c)},q_2^!\omega_{M^\dagger(\chi;\alpha)})$. Explicitly, for either $a_\bullet$ the element is the composite
\begin{align}\label{eq:pushed-purity}
\begin{split}
q_1^*\omega_{M^\dagger(\chi;c)}
&\to\dR a_{\bullet*}a_\bullet^*q_1^*\omega_{M^\dagger(\chi;c)}
\xrightarrow{\ \dR a_{\bullet*}\eta_{q_1a_\bullet}\ }\dR a_{\bullet*}\omega_{Z_\bullet}\\
&=\dR a_{\bullet*}a_\bullet^!q_2^!\omega_{M^\dagger(\chi;\alpha)}\to q_2^!\omega_{M^\dagger(\chi;\alpha)},
\end{split}
\end{align}
where the outer arrows are the pull-back unit and proper counit, and $q_1a_\bullet$ is the classical truncation of the quasi-smooth source map. Lemma~\ref{lem:exchange} and proper base change show that $\zeta_{FE}-\zeta_{EF}$ vanishes on $T_\chi\setminus I_\chi$.

By Lemma~\ref{lem:supported}, the morphism induced by this difference after direct image to $B_{\beta,Z}$ has the form $\dR h_\chi(\alpha)_*u\circ\res_\chi(c,\alpha)$ for some $u\in\End(\omega_{M^\dagger(\chi;\alpha)})$, and Proposition~\ref{prop:connected} gives $u=\lambda\,\id$ for a single rational number $\lambda$.

Lemma~\ref{lem:convolution} identifies the morphisms induced by $\zeta_{FE}$ and $\zeta_{EF}$ after direct image to $B_{\beta,Z}$ with the two compositions on the left hand side of~\eqref{eq:mixed-relation}, precomposed with $\res_\chi(c,\alpha-1)$: their initial pairs automatically lie in the open substack with lower slope bound $\alpha-1$, so the unit for the wider source factors through this open restriction. We compute $\lambda$ on smooth curves.

Take $C$ smooth and transverse to $D$, and let $W$ be an effective divisor of degree $n=\chi+g_\beta-1$ consisting of distinct points disjoint from $D$. The final pair is $(F,s)=(\cO_C(W),1)$. In the first order, an upper modification is $F(y)$ for some $y\in C\cap D$, and the lower initial pair is $F(y-x)$, where $x$ is one of the $n+1$ zeros of its upper section; there are therefore $(n+1)rd$ points over the final pair. In the second order, the lower pair is $F(-x)$ for one of the $n$ points $x\in W$, and the initial pair is $F(-x+y)$ for some $y\in C\cap D$; there are $nrd$ points. Both families are finite \'etale over this open subset. Note that for $n=0$ the second is empty.

To identify the purity classes, let $V\subset M^\dagger(\chi;\alpha)$ denote this open set of final pairs. Over $V$, the derived surface pair locus is the smooth relative symmetric power of dimension $N=\dim B_\beta+n=\beta^2+\chi$. Its deformation complex has amplitude $[0,1]$, tangent dimension $N$, and Euler characteristic $N$, so its obstruction group vanishes. The target projections of $\bZ_{FE}$ and $\bZ_{EF}$ are quasi-smooth of virtual dimension zero: in either order, the two projections have dimension zero after imposing the divisor condition on the appropriate point. Classically, their restrictions over $V$ are the finite \'etale families just counted. Their relative tangent spaces therefore vanish, and virtual dimension zero forces their relative obstruction spaces to vanish as well. Thus the restrictions are derived \'etale over $V$ and smooth of dimension $N$. Their source purity maps between the constant dualizing complexes $\QQ[2N](N)$ are the ordinary fundamental classes, so their proper traces have the degrees above.

Every $G$ in~\eqref{eq:common-subsheaf} over $V$ is a rank one torsion free sheaf on the smooth support curve of the final pair, with generically non-zero section. Its initial pair therefore lies in the smooth open locus of stable surface pairs on smooth curves in $M^\dagger(\chi;c)$, of dimension $N$. On $q_2^{-1}(V)$, the complexes $q_1^*\omega_{M^\dagger(\chi;c)}$ and $q_2^*\omega_V$ are consequently identified with $\QQ[2N](N)$. Precomposing $\zeta_{FE},\zeta_{EF}$ with this canonical identification and applying the unit and proper trace for $q_2$ gives the two degrees. Applied to the localized lift in Lemma~\ref{lem:supported}, this operation gives $u|_V$ by its defining unit-counit formula. Therefore $\lambda=(n+1)rd-nrd=rd$, which proves~\eqref{eq:mixed-relation}.
\end{proof}
\section{Proof of the support theorem}\label{sec:proof-support}
\subsection{Slope restriction and the commutator}

\begin{lem}\label{lem:slices}
Assume the support theorem in every smaller curve degree and every generic chamber. For any $c\le\alpha$, the restriction $\res_\chi(c,\alpha)$ becomes an isomorphism modulo $\cS_\beta$, and either side is isomorphic there, up to normalization, to the PT direct image. This holds also on $B_{\beta,Z}$, with the category defined by the strict supports $\Sigma_\lambda\cap B_{\beta,Z}$.
\end{lem}
\begin{proof}
Apply the argument of Theorem~\ref{thm:bounded} to the two lower bounds $c$ and $\alpha$. Proposition~\ref{prop:slice} bounds the HN types for each fixed lower bound, so Lemma~\ref{lem:last-chamber} provides a last chamber common to both finite type stacks. Their stable open substacks are both $P_\chi(X,\beta)$: every PT sheaf satisfies the stronger bound $\alpha$. Each HN stratum in the complement has a PT factor of smaller curve degree and semistable rank zero factors of positive curve degree. The induction hypothesis, Lemma~\ref{lem:rank-zero} and Lemma~\ref{lem:serre} therefore place these terms in $\cS_\beta$.

The inclusion of the two open substacks gives a map between their localization triangles, inducing the identity on the common PT open substack. By Proposition~\ref{prop:DR}, dimensional reduction followed by $\langle-\chi\rangle$ identifies the full terms with $K_\chi(c)$ and $K_\chi(\alpha)$. Their restriction map is consequently an isomorphism modulo $\cS_\beta$, and both terms identify there with the normalized PT direct image. Restricting the same triangles to $B_{\beta,Z}$ gives the second assertion.
\end{proof}

\begin{lem}\label{lem:weyl}
In a $\QQ$-linear additive category, let $V_i$ be objects, only finitely many of which are non-zero, with maps $e_i\colon V_i\to V_{i+1}$ and $f_i\colon V_{i+1}\to V_i$. If $f_ie_i-e_{i-1}f_{i-1}=m\,\id_{V_i}$ for a non-zero rational number $m$, then every $V_i$ is zero.
\end{lem}
\begin{proof}
Form the finite direct sum $V=\bigoplus_iV_i$. The given maps define endomorphisms $e$, $f$ of $V$ which raise and lower the index, and the relation reads $fe-ef=m\,\id_V$. Suppose that $V\neq 0$ and choose the least $i$ with $V_i\neq 0$. Then $f$ vanishes on $V_i$, since $V_{i-1}=0$. Induction on $k$ using the commutator gives
\begin{align*}
fe^k=e^kf+kme^{k-1},\qquad f^ke^k|_{V_i}=k!\,m^k\,\id_{V_i}.
\end{align*}
For sufficiently large $k$, the map $e^k$ vanishes because only finitely many $V_j$ are non-zero. Since $k!\,m^k$ is invertible in $\QQ$, the displayed identity forces $\id_{V_i}=0$, a contradiction.
\end{proof}

\subsection{Induction on the curve degree}

\begin{thm}\label{thm:support}
For every effective $\beta>0$, every $\chi\in\ZZ$, and every generic chamber of Toda's stability conditions, the rank one Chow direct image belongs to $\cS_\beta$. In particular every simple constituent of every $\pH{i}K_{\chi,\beta}$ has strict support $\Sigma_\lambda$ for some $\lambda\vdash_\Gamma\beta$. Here the PT vanishing cycle sheaf has the determinant orientation, identified with the cotangent orientation by Theorem~\ref{thm:oriented}.
\end{thm}
\begin{proof}
We induct on $d=L\cdot\beta$, proving the assertion simultaneously for every Euler characteristic and every generic chamber.

If $d=1$, the curve is a line or a ruling fiber. In the filtration of Lemma~\ref{lem:lower}, the zeroth layer is non-zero at a generic point of $C$ by Nakayama's lemma, so $L\cdot\beta_0\ge 1$. Since $\sum_kL\cdot\beta_k=1$, we have $\beta_k=0$ for $k\ge 1$. The ideal of the line or ruling fiber therefore annihilates the stable pair sheaf generically. Its image in the sheaf is zero dimensional, hence zero by purity, so the sheaf is scheme-theoretically supported on the smooth curve. It is a line bundle with a section of non-negative degree, so each non-empty $P_\chi(X,\beta)$ is a projective bundle over $B_\beta=\PP^2$ or $\PP^1$. Proposition~\ref{prop:hilb}, using the orientation comparison of Theorem~\ref{thm:oriented}, identifies its PT vanishing cycle sheaf with the shifted constant sheaf. Its direct image is a sum of shifted constant sheaves, with strict support $B_\beta=\Sigma_{(\beta)}$. Chamber independence gives the assertion in every chamber.

Fix $d>1$ and assume the theorem in smaller degrees. Corollary~\ref{cor:window} treats $g_\beta\le 0$. For $g_\beta>0$, the class $\beta$ is ample, so Section~\ref{sec:hecke} applies.

Choose a pencil with base locus $Z$ and work on $B_{\beta,Z}$. Let $\cQ$ be the Verdier quotient of $\widehat{D}\MHM(B_{\beta,Z})$ by the full triangulated subcategory of complexes whose simple perverse constituents have strict support $\Sigma_\lambda\cap B_{\beta,Z}$ for some $\lambda$. Write $V_\chi$ for the image of $K_\chi(\alpha)$ in $\cQ$. Lemma~\ref{lem:slices} identifies it with the image of $K_{\chi,\beta}$, with the specified normalization. We will show that $V_\chi=0$ for every $\chi$.

All restrictions $\res_\chi(c,c')$ with $c\le c'\le\alpha$ are invertible in $\cQ$. Creation and removal therefore give maps
\begin{align*}
e_\chi:=E_\chi\circ\res_\chi(\alpha-1,\alpha)^{-1}\colon V_\chi\to V_{\chi+1},\qquad
f_\chi:=F_\chi(\alpha)\colon V_{\chi+1}\to V_\chi.
\end{align*}
In $\cQ$, precompose the relation of Theorem~\ref{thm:relation} with $\res_\chi(c,\alpha)^{-1}$ and use
\begin{align*}
\res_\chi(\alpha-1,\alpha)\circ\res_\chi(c,\alpha-1)=\res_\chi(c,\alpha).
\end{align*}
Together with~\eqref{eq:F-restriction}, this identifies the two compositions with $f_\chi e_\chi$ and $e_{\chi-1}f_{\chi-1}$, giving
\begin{align}\label{eq:commutator}
f_\chi e_\chi-e_{\chi-1}f_{\chi-1}=rd\cdot\id_{V_\chi}\qquad(\chi\ge 1-g_\beta).
\end{align}

Lemma~\ref{lem:lower} makes the PT space empty for $\chi<1-g_\beta$, while chamber comparison and reflection give the support assertion for $\chi>g_\beta-1$. Thus only finitely many of the quotient objects can be non-zero:
\begin{align*}
V_\chi=0\qquad\text{unless}\qquad 1-g_\beta\le\chi\le g_\beta-1.
\end{align*}
Below this interval the objects are zero, so the commutator relation holds there as well. Lemma~\ref{lem:weyl}, applied with $m=rd\neq 0$, gives $V_\chi=0$ for every $\chi$.

The opens cover $B_\beta$. Indeed, given a curve, choose a first pencil member containing none of its components and a second member avoiding its finite intersection with the curve. The resulting pencil has base locus disjoint from the curve. If a simple constituent had strict support $W$ distinct from every $\Sigma_\lambda$, choose an open subset $B_{\beta,Z}$ meeting $W$. Its restriction would have strict support $W\cap B_{\beta,Z}$. The vanishing of $V_\chi$ makes this equal to $\Sigma_\lambda\cap B_{\beta,Z}$ for some $\lambda$; taking closures in $B_\beta$ gives $W=\Sigma_\lambda$, a contradiction. Thus $K_{\chi,\beta}\in\cS_\beta$. Chamber comparison completes the induction for every generic chamber.
\end{proof}

\begin{cor}\label{cor:nonreduced}\label{cor:middle-ext}
For every effective $\beta>0$ and $\chi\in\ZZ$, every simple constituent of the perverse cohomology of $K_{\chi,\beta}$ has support meeting $B_\beta^{\red}$. The same assertion holds in every generic chamber. In particular, with $j_\beta\colon B_\beta^{\red}\hookrightarrow B_\beta$, one has
\begin{align*}
K^{\sss}_{\chi,\beta}\cong{}^pj_{\beta!*}\bigl(j_\beta^*K_{\chi,\beta}\bigr)^{\sss}.
\end{align*}
\end{cor}
\begin{proof}
By Theorem~\ref{thm:support}, each strict support is a $\Sigma_\lambda$, which meets the reduced locus. Each simple constituent is therefore the intermediate extension of its restriction to $B_\beta^{\red}$. This also proves Theorem~\ref{intro:thm-support}.
\end{proof}

\section{The correspondence over reduced curves}\label{sec:reduced}
Smoothness of the relative Hilbert schemes and surjectivity onto the node-smoothing directions allow us to apply~\cite[Section~3]{Zha26} and~\cite[Theorem~5.10]{MSV21}. Appendix~\ref{app:nodal-actions} proves the equivariant nodal formula; the comparison with the symmetric powers is given below. As before, $n=\chi+g_\beta-1$.

\subsection{The relative Hilbert schemes}\label{sub:hilb}
Let $\cC^{[n]}_{\beta,\red}$ be the relative Hilbert scheme of length $n$ subschemes of the universal divisor over $B^{\red}_\beta$.

\begin{prop}\label{prop:hilb}
For $n\ge 0$, the relative Hilbert scheme $\cC^{[n]}_{\beta,\red}$ is smooth of dimension $D_\beta+n$ and projective over $B^{\red}_\beta$. There is an isomorphism
\begin{align*}
P_{1-g_\beta+n}(X,\beta)|_{B^{\red}_\beta}\cong\cC^{[n]}_{\beta,\red},
\end{align*}
under which
\begin{align}\label{eq:Hilb}
\phi^{\PT}_{1-g_\beta+n,\beta}\big|_{B^{\red}_\beta}\cong\QQ_{\cC^{[n]}_{\beta,\red}}[D_\beta+n].
\end{align}
For $n<0$, the stable pair locus over $B_\beta^{\red}$ is empty.
\end{prop}
\begin{proof}
Consider first the relative Hilbert scheme over the complete linear system. It is the zero locus of the universal restriction section of a rank $n$ bundle on $B_\beta\times\Hilb^n(S)$, which is a smooth variety of dimension $D_\beta+2n$~\cite{Fog68}. Every component of $\Hilb^n(C)$ has dimension $n$ for a planar curve, including a reducible or non-reduced one~\cite[Theorem~1.1]{Lua23}. To apply Luan's statement for affine plane curves, choose a line avoiding the finite subscheme on the plane, or a fiber of each ruling avoiding it on the quadric. In either case the subscheme lies in an affine plane. The fiber dimension bound and the number of defining equations therefore give pure dimension $D_\beta+n$; the section is regular, and miracle flatness gives relative dimension $n$. Projectivity follows from the relative Hilbert construction.

Now let $C\in B_\beta^{\red}$. Write $\nu\colon\wt{C}\to C$ for its normalization and $\fc_C=\cHom_C(\nu_*\cO_{\wt{C}},\cO_C)$ for its conductor ideal. Finite duality and adjunction give
\begin{align*}
\fc_C\otimes\cO_C(C)\cong\nu_*\bigl(K_{\wt{C}}\otimes\nu^*K_S^{-1}\bigr),\qquad
H^1\bigl(C,\fc_C\otimes\cO_C(C)\bigr)\cong H^0\bigl(\wt{C},\nu^*K_S\bigr)^\vee=0,
\end{align*}
where the last equality uses the negativity of $K_S$ on every component of the normalization. Locally at a singular point $x$, write $C=(f(u,v)=0)$. The conductor ideal maps to the first order deformation space $\cO_{C,x}/(\partial_uf,\partial_vf)$. The displayed vanishing says that the Kodaira-Spencer image of the complete linear system is transverse to the direct sum of these images over the singular points, as required by~\cite[Theorem~4.16(3)]{MSV21}. The theorem gives smoothness along $C^{[n]}$ for every relative Hilbert scheme.

Let a local surface stable pair have reduced fundamental divisor $C$. Its surface sheaf $G$ is torsion free of rank one on $C$: the generic multiplicity one along each component implies $I_CG=0$ generically, where $I_C\subset\cO_S$ is the ideal of $C$, and purity makes it zero everywhere. The Higgs field lies in
\begin{align*}
\Hom_C(G,G\otimes K_S)\hookrightarrow H^0\bigl(\wt{C},\nu^*K_S\bigr)=0,
\end{align*}
so it vanishes. Over an Artin local base, the successive leading terms lie in this zero vector space tensored with the successive powers of the maximal ideal. They vanish by induction; completion and finite presentation then give vanishing for arbitrary families. Relative Gorenstein duality identifies $\cO_C\to G$ with the colength $n$ ideal $\cHom_C(G,\cO_C)\subset\cO_C$. This identification commutes with base change and gives the isomorphism of moduli functors.

On this stable pair locus, the surface pair tangent complex has amplitude $[0,1]$ and Euler characteristic $\beta^2+\chi=D_\beta+n$. Its tangent space has this dimension by the smoothness just proved, so the obstruction space is zero, and the derived surface pair locus is classical and smooth. Theorem~\ref{thm:oriented} identifies $\phi^{\PT}_{1-g_\beta+n,\beta}$ with the vanishing cycle sheaf of the shifted cotangent stack with its canonical orientation. For a smooth base, this is the shifted constant sheaf in~\eqref{eq:Hilb}. For negative colength the locus is empty.
\end{proof}

\begin{lem}\label{lem:transv}
Let $C\in U_\lambda$ be a transverse union of smooth connected curves. The Kodaira-Spencer map of the linear system $B_\beta$ at $C$ surjects onto the direct sum of the first order smoothing spaces of the nodes of $C$.
\end{lem}
\begin{proof}
The singularities of $C$ are its nodes, and the conductor $\fc_C$ is their ideal, so the smoothing space of the nodes is $H^0(C,\cO_C(C)/\fc_C\otimes\cO_C(C))$. The vanishing $H^1(S,\cO_S)=0$ makes $H^0(S,\cO_S(C))\to H^0(C,\cO_C(C))$ surjective. The next map, from $H^0(C,\cO_C(C))$ to the node-smoothing space, is surjective by the conductor vanishing $H^1(C,\fc_C\otimes\cO_C(C))=H^0(\wt{C},\nu^*K_S)^\vee=0$ proved in Proposition~\ref{prop:hilb}. Their composition is the required map. For the plane this is~\cite[Lemma~5.1]{Zha26}.
\end{proof}

\subsection{The identity on the complete Chow variety}\label{sub:identity}
Let $\cR_{\chi,\beta}$ be the coefficient of $q^\chi Q^\beta$ in the series~\eqref{intro:Exp}. For a partition $\lambda\vdash_\Gamma\beta$, list its classes with multiplicity as $(\gamma_1,\ldots,\gamma_\ell)$ and put $N(\lambda)=\sum_{i<j}\gamma_i\cdot\gamma_j$. We retain the notation $T_{\lambda,s}$ and $U_{\lambda,s}$ of~\cite[(12), (28)]{Zha26}, whose definitions we recall. For $s\ge 0$, on the cover $\wt{U}_\lambda\to U_\lambda$ ordering the components, let $\wt{\rho}_{\lambda,s}$ be the projection to $\wt{U}_\lambda$ from the relative length $s$ Hilbert scheme of their disjoint union. For the permutation group $G_\lambda$ in~\eqref{intro:mu}, set $\varepsilon_\lambda=\bigotimes_\gamma(\sgn_{\mathfrak S_{m_\gamma}})^{\gamma^2}$. Define $\cH^k_{\lambda,s}$ as the descent of $R^k\wt{\rho}_{\lambda,s*}\QQ$ under the graded K\"unneth permutation action tensored with $\varepsilon_\lambda$. Put
\begin{align}\label{eq:T}
T_{\lambda,s}:=\bigoplus_{k=0}^{2s}\IC_{\Sigma_\lambda}(\cH^k_{\lambda,s})[s-k].
\end{align}
For $\chi=1-g_\beta+N(\lambda)+s$, the other term is defined by
\begin{align}\label{eq:U}
U_{\lambda,s}[-\chi]:=[q^\chi]\Bigl[\dR\mu_{\lambda*}\Bigl(\boxtimes_{\gamma\in\Gamma}\cA_\gamma(q)^{\boxtimes m_\gamma}\Bigr)\Bigr]^{G_\lambda}.
\end{align}
Here $[q^\chi]$ means the coefficient extraction. We set both terms equal to zero for $s<0$.

\begin{lem}\label{lem:nodal-descent}
Let $\beta\in\Gamma$, $n\ge 0$, and $\chi=1-g_\beta+n$. For $\lambda\vdash_\Gamma\beta$, put $s=n-N(\lambda)$. If $s\ge 0$, the local system on $U_\lambda$ defining the summand of $\pH{i}(j_\beta^*K_{\chi,\beta})$ with strict support $\Sigma_\lambda\cap B^{\red}_\beta$ is $\cH^{s+i}_{\lambda,s}$. This summand is zero if $s<0$.
\end{lem}
The proof in Appendix~\ref{app:nodal-actions} computes the branch-exchange character in the nodal support formula of~\cite[Theorem~5.10]{MSV21}.

\begin{thm}\label{thm:reduced}
Let $\beta\in\Gamma$ on either surface. For every $\chi$ and every $\lambda\vdash_\Gamma\beta$, the summand of $K^{\sss}_{\chi,\beta}[-\chi]$ with strict support $\Sigma_\lambda$ is isomorphic to
\begin{align*}
T_{\lambda,n-N(\lambda)}[-\chi]\cong U_{\lambda,n-N(\lambda)}[-\chi].
\end{align*}
Consequently
\begin{align*}
j_\beta^*K^{\sss}_{\chi,\beta}[-\chi]\cong j_\beta^*\cR_{\chi,\beta}.
\end{align*}
\end{thm}
\begin{proof}
By Proposition~\ref{prop:hilb} and~\cite[Proposition~2.11]{Zha26}, $j_\beta^*K_{\chi,\beta}$ is the direct image of the intersection complex of the smooth relative Hilbert scheme. Lemma~\ref{lem:transv} gives the node-smoothing surjectivity required in~\cite[Proposition~3.3]{Zha26}. Lemma~\ref{lem:nodal-descent} identifies the summand with strict support $\Sigma_\lambda\cap B^{\red}_\beta$ with the cohomology of the length $s=n-N(\lambda)$ Hilbert scheme of the normalization, with the descent and shifts in $T_{\lambda,s}$. Theorem~\ref{thm:support} then gives its intermediate extension to $\Sigma_\lambda$.

The length $r_i$ term of $\cA_{\gamma_i}(q)$ has $q$-degree $1-g_{\gamma_i}+r_i$. Adjunction gives
\begin{align*}
g_\beta=\sum_ig_{\gamma_i}+N(\lambda)-\ell+1,\qquad \sum_i(1-g_{\gamma_i}+r_i)=1-g_\beta+N(\lambda)+\sum_ir_i.
\end{align*}
Hence this choice contributes to $q^\chi=q^{1-g_\beta+n}$ precisely when $\sum_ir_i=s$. Relative to the unnormalized Macdonald coefficients of~\cite[(9)]{Zha26}, its total cohomological shift is $\sum_i(g_{\gamma_i}-1-r_i)=-\chi$. Sum over these length allocations and take the permutation invariants in~\eqref{intro:Exp}. This gives
\begin{align}\label{eq:reduced-expansion}
\cR_{\chi,\beta}\cong\bigoplus_{\lambda\vdash_\Gamma\beta}U_{\lambda,n-N(\lambda)}[-\chi].
\end{align}
Lemma~\ref{lem:serre} shows that $U_{\lambda,s}$ is a direct sum of intersection complexes with strict support $\Sigma_\lambda$.

Work over the cover $\wt{U}_\lambda$ ordering the smooth components. Write $\rho^{[r]}_\gamma\colon\cC^{[r]}_\gamma|_{U_\gamma}\to U_\gamma$ for the relative length $r$ Hilbert scheme of the smooth curve family. In the following products, each factor is pulled back to $\wt U_\lambda$, and the products are taken over that base. The relative Hilbert scheme of the normalization is
\begin{align*}
\coprod_{r_1+\cdots+r_\ell=s}\prod_i\cC^{[r_i]}_{\gamma_i}.
\end{align*}
On this cover, the K\"unneth formula identifies the graded local systems. Permuting the components permutes the allocation $(r_i)$ and acts on its cohomology by the graded K\"unneth rule. The descent of the nodal summand carries the character $\varepsilon_\lambda$. To compare this character with the symmetric powers, let $J$ be the nodes between distinct components and let $\varepsilon_J$ be the one dimensional representation recording branch reversals. Monodromy fixing each component only permutes the nodes and acts trivially on $\varepsilon_J$. This character therefore factors through $G_\lambda$; its determinant description is given in Appendix~\ref{app:nodal-actions}.

For a transposition of two components of class $\gamma$, orient edges to any other component in the same direction. These edges are permuted without reversal. The $\gamma^2$ nodes joining the exchanged components have their branches reversed, so the transposition acts on $\varepsilon_J$ by $(-1)^{\gamma^2}$. Since transpositions generate $G_\lambda$, the resulting character is $\varepsilon_\lambda$. In~\eqref{intro:Exp}, the normalized length $r$ coefficient on the smooth class $\gamma$ locus is $\dR\rho^{[r]}_{\gamma*}\QQ[\gamma^2]$, since $D_\gamma+g_\gamma-1=\gamma^2$. Under the usual suspension isomorphisms, permuting equal factors acts by the graded K\"unneth permutation and by the sign of the shifted one dimensional factors $\QQ[\gamma^2]$. This agrees with the character $\varepsilon_J$ and the action on length allocations above.

Their intermediate extensions therefore satisfy $T_{\lambda,s}\cong U_{\lambda,s}$, and~\eqref{eq:reduced-expansion} proves the theorem.
\end{proof}

\begin{lem}\label{lem:rulings}
For $S=\mathbb{F}_0$ and $\beta=mf$, the isomorphism of Theorem~\ref{thm:reduced} holds.
\end{lem}
\begin{proof}
A reduced member of $|mf|$ is a disjoint union of $m$ fibers, and $B^{\red}_{mf}$ is the complement of the diagonals in $\Sym^m\PP^1$. On the finite \'etale cover ordering the fibers, the stable pair space of Proposition~\ref{prop:hilb} is the disjoint union, over $\sum_in_i=n$, of the products $\prod_i\PP^{n_i}$ of symmetric powers of the fibers, with the constant sheaf shifted by $D_{mf}+n$. Here $D_{mf}=m$, $g_{mf}=1-m$ and $\chi=m+n$, so the common normalization cancels the shift: $D_{mf}+n-\chi=0$. The normalized direct image on the cover is therefore the exterior product of the single fiber complexes $\bigoplus_{v=0}^{n_i}\QQ_{\PP^1}[-2v]$, which are the coefficients of~\eqref{intro:ruling-A}. Descent along the cover, with the component permutations acting through the Koszul signs of the even shifts, gives exactly the $m$-th symmetric power term of the series~\eqref{intro:Exp}; the terms involving $\cA_{m'f}$ with $m'>1$ vanish by~\eqref{intro:ruling-A}. If $n<0$ both sides are zero.
\end{proof}

\begin{thm}\label{thm:gvpt}
For $S=\PP^2$ and $S=\mathbb{F}_0$, and for every $\chi\in\ZZ$ and $\beta>0$, there is an isomorphism
\begin{align*}
K^{\sss}_{\chi,\beta}[-\chi]\cong\cR_{\chi,\beta}
\end{align*}
of finite direct sums of shifts of semisimple perverse sheaves on $B_\beta$. In particular Theorem~\ref{intro:thm-gvpt} holds.
\end{thm}
\begin{proof}
Fix $(\chi,\beta)$. By Corollary~\ref{cor:middle-ext}, $K^{\sss}_{\chi,\beta}[-\chi]$ is the termwise perverse intermediate extension of its restriction to $B^{\red}_\beta$, and by~\cite[Theorem~1.1(i)]{Zha26}, or Lemma~\ref{lem:serre}, so is $\cR_{\chi,\beta}$. Theorem~\ref{thm:reduced} and Lemma~\ref{lem:rulings} identify the restrictions, hence
\begin{align*}
K^{\sss}_{\chi,\beta}[-\chi]\cong{}^pj_{\beta!*}j_\beta^*K^{\sss}_{\chi,\beta}[-\chi]\cong{}^pj_{\beta!*}j_\beta^*\cR_{\chi,\beta}\cong\cR_{\chi,\beta}.
\end{align*}
This follows the intermediate-extension argument of~\cite[Corollary~4.4]{Zha26}.
\end{proof}

\section{The refined product formula}\label{sec:lefschetz}
For the resolved conifold, Morrison-Mozgovoy-Nagao-Szendr\H{o}i~\cite{MMNS12} obtained motivic product formulas and a motivic DT/PT correspondence. Choi-Katz-Klemm~\cite{CKK14} proposed the refined stable pair product formula for more general local surfaces.

Pi-Shen-Si-Zhang~\cite{PSSZ24} proved an asymptotic product formula for the perverse-graded dimensions of the intersection cohomology of sheaf moduli in ample classes on del Pezzo surfaces. For smooth fine moduli spaces, they also obtained $P=C$ in the stable cohomological range, using the Fourier-transform results of Maulik-Shen-Yin~\cite{MSY25}. On $\PP^2$, this equality holds through cohomological degree $2d-4$ for sheaves of degree $d$ and coprime Euler characteristic.

\begin{cor}\label{cor:purity}
In the normalization of~\eqref{eq:Hilb}, $\pH{i}K_{\chi,\beta}$ is pure of weight $\beta^2+\chi+i$ and semisimple. For a relatively ample class $\ell$ on the stable pair space,
\begin{align*}
\ell^i\colon\pH{-i}K_{\chi,\beta}\xrightarrow{\ \sim\ }\pH{i}K_{\chi,\beta}(i),\qquad i\ge 0.
\end{align*}
\end{cor}
\begin{proof}
Over the reduced locus, smoothness and the projective direct image give purity and the relative hard Lefschetz theorem~\cite{Sai88}. A weight graded piece of any other weight would restrict to zero on the reduced locus, so its simple constituents would be supported on $B_\beta\setminus B_\beta^{\red}$, which the support theorem excludes. Thus purity holds globally, and semisimplicity follows~\cite{Sai90}. Similarly, the kernel and cokernel of each displayed Lefschetz map restrict to zero and are subquotients of the same perverse sheaves; the support theorem makes both zero.
\end{proof}

For $\beta\in\Gamma$, put
\begin{align*}
P_{\beta,i}:=\IC_{B_\beta}\bigl(\wedge^{g_\beta+i}V_\beta\bigr),\qquad -g_\beta\le i\le g_\beta.
\end{align*}
The polarization of the relative Jacobian gives the first Lefschetz action, which moves the index $i$, and an ample class on $B_\beta$ gives the second, which acts on hypercohomology. We define the integers $N^\beta_{j_L,j_R}$ by
\begin{align}\label{eq:N-jLjR}
\sum_{i,k}\dim H^k(B_\beta,P_{\beta,i})\,x^iy^k=\sum_{j_L,j_R}N^\beta_{j_L,j_R}\chi_{j_L}(x)\chi_{j_R}(y),\qquad \chi_j(z):=\sum_{m=-j}^jz^{2m}.
\end{align}
Here $j_L,j_R\in\frac12\ZZ_{\ge 0}$, and the summation index $m$ in $\chi_j$ increases by one. We set $N^\beta_{j_L,j_R}=0$ for $\beta\notin\Gamma$. These are the GV invariants in the refined sense of~\cite{KKV99, HST01}.

Corollary~\ref{cor:purity} and the Lefschetz splitting give $K_{\chi,\beta}\cong K^{\sss}_{\chi,\beta}$. Proper push-forward and the K\"unneth formula applied to Theorem~\ref{thm:gvpt} identify the hypercohomology of~\eqref{intro:Exp} with the corresponding graded symmetric algebra. Define
\begin{align*}
Z^{\mathrm{ref}}_{\PT}(q,Q;\LL^{1/2}):=1+\sum_{\beta>0,\,\chi,\,k}\dim H^k(B_\beta,K_{\chi,\beta})\,(-\LL^{1/2})^kq^\chi Q^\beta,
\end{align*}
Then, as in~\cite[Proposition~4.5]{Zha26},
\begin{align}\label{eq:refined-PTGV}
\begin{split}
Z^{\mathrm{ref}}_{\PT}=\prod_{\beta>0}\prod_{j_L,j_R}&\prod_{m_L=-j_L}^{j_L}\prod_{m_R=-j_R}^{j_R}\prod_{m\ge 1}\prod_{v=0}^{m-1}\\
&\Bigl(1-\LL^{-m/2+1/2+v-m_R}(-q)^{m-2m_L}Q^\beta\Bigr)^{(-1)^{2(j_L+j_R)}N^\beta_{j_L,j_R}}.
\end{split}
\end{align}
To see this, write $j=g_\beta+i$ and $u+v=m-1$ in~\eqref{intro:A}, so that the corresponding summand of $\cA_\beta(q)$ becomes $q^{m+i}P_{\beta,i}[-i-1-2v]$. Form its graded symmetric algebra, remove the normalization $[-\chi]$ by substituting $q\mapsto q/r$ in the cohomological series, where $r$ records the cohomological degree, and then put $r=-\LL^{1/2}$. The Lefschetz symmetries $(i,k)\mapsto(-i,-k)$ give~\eqref{eq:refined-PTGV}. Setting $\LL^{1/2}=1$ recovers the Euler characteristic version, hence the numerical PT/GV formula of~\cite{PT10}.
\appendix
\section{Orientation data on the local surface}\label{sec:orientation}
\subsection{The derived moduli problem}
For either surface $S$, let $\ol{X}=\PP_S(\cO_S\oplus K_S)$, with projection $\ol{p}$ and zero and infinity sections $S_0$, $S_\infty$, whose inclusions are $i_0,i_\infty\colon S\hookrightarrow\ol{X}$. Put $\cL:=\cO_{\ol{X}}(S_\infty)$, so that
\begin{align*}
\dR\ol{p}_*\cL=\cO_S\oplus K_S^{-1},\qquad \dR\ol{p}_*\cL^{-1}=0,\qquad K_{\ol{X}}=\cL^{-2}.
\end{align*}
Let $\bM^\dagger_X$ be the derived moduli stack of perfect complexes $I$ on $\ol{X}$ together with a specified equivalence $Li^*_\infty I\simeq\cO_S$, restricted to the derived open substack whose classical points are rank one objects of the heart $\cA_X$ in~\eqref{eq:heart} and have surface sheaf pure of dimension one. Let $\bM^\dagger_S$ be the derived stack of arrows $s\colon\cO_S\to F$ with $F$ pure of dimension one.

\begin{lem}\label{lem:derived-cotangent}
There is an equivalence of derived stacks
\begin{align*}
\eta\colon\bM^\dagger_X\xrightarrow{\ \sim\ }T^*[-1]\bM^\dagger_S
\end{align*}
whose classical truncation is Toda's isomorphism~\cite[Theorem~4.1.3(ii)]{Tod24}. It commutes with the Hilbert-Chow map, with arbitrary derived base change, and with scaling in the normal direction to $S_0$.
\end{lem}
\begin{proof}
Work in the stable $\infty$-category of perfect complexes, with specified homotopies in the diagrams. The relative resolution of the diagonal of the projective bundle gives the two-term reconstruction of~\cite[Section~4.5.4, (4.5.8)--(4.5.10)]{Tod24}. As a resolution of Fourier-Mukai functors, it commutes with derived base change. It identifies $I$ with
\begin{align*}
I=\fib\bigl(\ol{p}^*U\xrightarrow{(a,b)}\ol{p}^*F\otimes\cL\bigr),\qquad a\colon U\to F,\quad b\colon U\to F\otimes K_S^{-1},
\end{align*}
where the notation $(a,b)$ refers to the displayed decomposition of $\dR\ol{p}_*\cL$. In the other direction, $F$ and $U$ are recovered as $\dR\ol{p}_*(I(-S_\infty))[1]$ and $\dR\ol{p}_*(I(-S_0-S_\infty))[1]$, and the resolution of the diagonal provides the unit and counit making the two constructions inverse to each other. Restriction to $S_\infty$ is $\fib(b)$, so the framing amounts to a specified fiber sequence
\begin{align*}
\cO_S\xrightarrow{u}U\xrightarrow{b}F\otimes K_S^{-1},\qquad a\colon U\to F.
\end{align*}
To recover the surface pair, put $s:=au$ and $J:=\fib(s)$. Taking cofibers in
\begin{align*}
\xymatrix{\cO_S\ar[r]^-{u}\ar@{=}[d] & U\ar[r]^-{b}\ar[d]^-{a} & F\otimes K_S^{-1}\ar[d]^-{\vartheta}\\ \cO_S\ar[r]^-{s} & F\ar[r] & J[1]}
\end{align*}
gives the arrow $\vartheta$. Conversely, from $(F,s,\vartheta)$ form the homotopy pull-back
\begin{align*}
U:=F\times_{J[1]}(F\otimes K_S^{-1}),
\end{align*}
where $F\to J[1]$ is the cofiber map of $s$. Its projection to $F$ is $a$, and its other projection has fiber $\cO_S$. The pull-back/push-out identity in a stable $\infty$-category shows that these constructions are inverse on objects and mapping spaces.

Finally, the tangent complex of $\bM^\dagger_S$ at $(F,s)$ is $D:=\Rhom_S(J,F)$, and relative Serre duality gives
\begin{align*}
D^\vee[-1]\simeq\Rhom_S(F\otimes K_S^{-1},J[1]),
\end{align*}
which identifies the possible arrows $\vartheta$ with the shifted cotangent vectors. On classical truncations the universal fiber sequences give Toda's map. Since the construction preserves $F$, it preserves the fundamental cycle. Normal scaling acts by scaling $\vartheta$.
\end{proof}

\subsection{Comparison of the symplectic forms}
\begin{lem}\label{lem:cyclic}
Under the equivalence of Lemma~\ref{lem:derived-cotangent}, the Calabi-Yau and cotangent closed $(-1)$-shifted symplectic forms agree formally along the zero section $\bM^\dagger_S\subset T^*[-1]\bM^\dagger_S$; in particular their $d$-critical sections agree there. The comparison is compatible with the universal surface pair family and with derived base change.
\end{lem}
\begin{proof}
The anticanonical divisor $2S_\infty\in|-K_{\ol{X}}|$ and the trace form on perfect complexes give a shifted symplectic form on the framed moduli stack. Apply~\cite[Theorem~3.3]{Spa16} with framing divisor $S_\infty$, additional divisor zero, and target the $2$-shifted symplectic stack of perfect complexes~\cite[Section~2.3]{PTVV13}. The restriction map for the additional divisor is the identity. Bounded derived Artin open substacks of the perfect complex mapping stacks satisfy the local finite presentation hypothesis; the resulting $(-1)$-shifted symplectic forms glue to a form on $\bM^\dagger_X$. On the PT open substack, the trace-free endomorphism complex is supported in $X$. Excision identifies its framed trace with compactly supported Serre trace for the canonical volume form on $X$, and therefore identifies the universal trace construction of the closed form with the Calabi-Yau construction. This form has weight one under normal scaling.

Fix a surface pair $(F,s)$, put $J=[\cO_S\xrightarrow{s}F]$, and let $I_0=[\cO_{\ol{X}}\to i_{0*}F]$. Write
\begin{align*}
B:=\Rhom_S(F,F),\qquad V:=\Rgam(S,F),\qquad D:=\Cone\bigl(B\xrightarrow{\,b\mapsto b\circ s\,}V\bigr).
\end{align*}
The tangent complex with fixed framing at infinity is
\begin{align*}
\Rhom_{\ol{X}}(I_0,I_0(-S_\infty))[1]\simeq\Rhom_{\ol{X}}(I_0,I_0)_0[1].
\end{align*}
Here the subscript $0$ denotes the fiber of the trace map, that is, the trace-free endomorphism complex. For the second identification use the restriction triangle and $\Rgam(\ol{X},\cO)=\Rgam(S,\cO)$; the trace splits off the restriction to $\End(\cO_S)$.

Endomorphisms of a locally free resolution of $I_0$ vanishing on $S_\infty$ give a differential graded Lie algebra representing $\Rhom(I_0,I_0(-S_\infty))$. Its bracket is the graded commutator of composition, using $\cO(-2S_\infty)\to\cO(-S_\infty)$; its Maurer-Cartan elements and gauge equivalences describe deformations preserving the framing. The pairing is
\begin{align*}
\Rhom(I_0,I_0(-S_\infty))^{\otimes 2}\to\Rhom(I_0,I_0\otimes K_{\ol{X}})\xrightarrow{\ \tr\ }\CC[-3].
\end{align*}
Serre duality makes the pairing non-degenerate. It is cyclic by the graded trace identity: both sides use the inclusion from the third power of the framing ideal. This is the trace used in the shifted symplectic construction above.

Interpreting the zero section resolution $0\to\ol{p}^*F(-S_0)\to\ol{p}^*F\to i_{0*}F\to 0$ through a perfect resolution of $F$ gives, with the natural scaling weights,
\begin{align*}
\Rhom_{\ol{X}}(i_{0*}F,i_{0*}F)&\simeq B\oplus B^\vee[-3],\\
\Rhom_{\ol{X}}(\cO,i_{0*}F)&\simeq V,\\
\Rhom_{\ol{X}}(i_{0*}F,\cO)&\simeq V^\vee[-3].
\end{align*}
The first member of each dual pair has weight zero and the second weight one. Put $\fg:=\Rhom_{\ol{X}}(I_0,I_0)_0$. Keeping the differential induced by $s$, we get
\begin{align}\label{eq:weighted-cyclic}
\fg\simeq D[-1]\oplus D^\vee[-2],
\end{align}
with weights zero and one.

Let $T$ be a derived test scheme in the surface pair stack, let $p\colon S\times T\to T$, and put $Q:=F\otimes K_S^{-1}$. All Hom complexes below are relative perfect complexes, with the derived direct image by $p$ understood. Set
\begin{align*}
\cB:=\dR p_*\dR\cHom(F,F),\qquad \cV:=\dR p_*F,\\
\cB':=\dR p_*\dR\cHom(Q,F),\qquad \cV':=\dR p_*\dR\cHom(Q,\cO).
\end{align*}
The two complexes are
\begin{align}\label{eq:two-cones}
\cD:=\Cone\bigl(\cB\xrightarrow{\,b\mapsto b\circ s\,}\cV\bigr),\qquad
\cC:=\Cone\bigl(\cV'\xrightarrow{\,e\mapsto s\circ e\,}\cB'\bigr).
\end{align}
Here $\cC=\dR p_*\dR\cHom(Q,J[1])$. We use the cohomological cone convention $d(b,a)=(db+f(a),-da)$ for $\Cone(f)$. In a differential graded presentation, an infinitesimal variation of the shifted cotangent vector is a pair $(\varphi,e)$ with
\begin{align*}
\varphi\colon Q\to F,\qquad e\colon Q\to\cO[1],\qquad d\varphi+se=0,\quad de=0.
\end{align*}
The inverse construction in Lemma~\ref{lem:derived-cotangent} is represented by
\begin{align}\label{eq:vertical-chain-map}
U=\cO\oplus Q,\qquad d_U=\begin{pmatrix}d_{\cO} & -e\\ 0 & d_Q\end{pmatrix},\qquad a=(s,\varphi),\qquad b=(0,1).
\end{align}
The equation $d\varphi+se=0$ says that $a$ commutes with the differentials. An infinitesimal deformation $(v,\xi)$ of the surface pair, represented in $\cD$, changes $d_F$ to $d_F+\xi$, $s$ to $s+v$, and the $Q$-block by $\xi\otimes K_S^{-1}$, while preserving $b=(0,1)$; its equations are $d\xi=0$ and $dv+\xi s=0$. Allowing square-zero parameters of arbitrary cohomological degree, these formulas define
\begin{align*}
\delta_0\colon\cD\to\fg_0[1],\qquad \delta_1\colon\cC\to\fg_1[1],
\end{align*}
where $\fg_i$ is the weight $i$ part of the relative framed endomorphism complex. Under reconstruction, $\delta_0$ and $\delta_1$ are the two components of the inverse of $d\eta$ along the zero section.

In the order weight one followed by weight zero, the surface pairing is the composite
\begin{align}\label{eq:relative-trace-pairing}
\ev_S\colon\cC\otimes\cD\to\dR p_*\dR\cHom(F,F\otimes K_S)[1]\xrightarrow{\ \tr\ }\dR p_*K_S[1]\xrightarrow{\ \Tr_{S/T}\ }\cO_T[-1].
\end{align}
The first arrow is Yoneda composition after tensoring the first argument by $K_S$. We compare it with the framed trace through the diagram
\begin{align}\label{eq:trace-square}
\xymatrix@C=4em{\cC\otimes\cD\ar[r]^-{\delta_1\otimes\delta_0}\ar[d]_-{\ev_S} & \fg_1[1]\otimes\fg_0[1]\ar[d]^-{\Tr_{\ol{X}/T}}\\ \cO_T[-1]\ar@{=}[r] & \cO_T[-1]}
\end{align}
in which both traces use graded evaluation and the displayed cone convention.

Put $f(\xi)=\xi s$ and $f'(e)=se$. The unshifted Serre evaluations are
\begin{align*}
\langle\varphi,\xi\rangle_{\cB}=\Tr_{S/T}\tr\bigl((\varphi\otimes K_S)\xi\bigr),\qquad
\langle e,v\rangle_{\cV}=\Tr_{S/T}\bigl((e\otimes K_S)v\bigr).
\end{align*}
Both are valued in $\cO_T[-2]$. Graded cyclicity gives the adjunction identity
\begin{align*}
\langle f'(e),\xi\rangle_{\cB}=\langle e,f(\xi)\rangle_{\cV}.
\end{align*}
If $c=(\varphi,e)\in\cC^p$ and $d=(v,\xi)\in\cD^q$ are homogeneous, then $|\varphi|=p$, $|e|=p+1$, $|v|=q$, and $|\xi|=q+1$. With both terms read in the common suspension $\cO_T[-1]$, the cone pairing is
\begin{align}\label{eq:homogeneous-cone-pairing}
P(c,d)=(-1)^p\langle\varphi,\xi\rangle_{\cB}+\langle e,v\rangle_{\cV}.
\end{align}
To check the chain map condition, write $\partial$ for the cochain differential. The two off-diagonal terms in $P(\partial c,d)+(-1)^pP(c,\partial d)$ are
\begin{align*}
(-1)^{p+1}\langle se,\xi\rangle_{\cB}+(-1)^p\langle e,\xi s\rangle_{\cV}=0,
\end{align*}
and the remaining terms are the Leibniz identities for the two unshifted evaluations, with the sign from the final suspension. Applying the graded Yoneda composition to $\cC=\Rhom(Q,J[1])$ and $\cD=\Rhom(J,F)$ gives exactly~\eqref{eq:homogeneous-cone-pairing}; the sign $(-1)^p$ comes from moving the shifted $\cB[1]$ factor past the first argument. Thus $P$ is~\eqref{eq:relative-trace-pairing}.

For the framed trace, use the reconstruction $K=\fib(\ol{p}^*U\to\ol{p}^*F\otimes\cL)$ at $U=\cO\oplus Q$. In $\dR\ol{p}_*\dR\cHom(K,K(-S_\infty))$ the diagonal blocks are acyclic, since $\dR\ol{p}_*\cL^{-1}=0$. The two off-diagonal blocks have direct images
\begin{align*}
\dR\cHom(U,F)[-1],\qquad \dR\cHom(F,U\otimes K_S),
\end{align*}
the second using $\dR\ol{p}_*\cL^{-2}=K_S[-1]$. For $U=\cO\oplus Q$ these are $\cV[-1]\oplus\cB'[-1]\oplus\cV'\oplus\cB$ after push-forward to $T$. Keep the differential induced by $s$ when eliminating the acyclic blocks: the result is $\cD[-1]\oplus\cC[-1]$ with the cone differentials in~\eqref{eq:two-cones}. Substitution of the matrix~\eqref{eq:vertical-chain-map} and the deformation of the surface pair shows that the two mixed supertrace terms are the $\cB'\otimes\cB[1]$ and $\cV'[1]\otimes\cV$ evaluations. The tensor-suspension sign on the first is $(-1)^p$; the second has sign $+1$. The other mixed block products have zero supertrace. We are left with~\eqref{eq:homogeneous-cone-pairing} times the relative trace in the normal direction.

Trivialize $K_S$ and choose the fiber coordinate $z$, with $S_0=(z=0)$ and $S_\infty=(z^{-1}=0)$. The zero section resolution is contracted on $z\neq 0$ by division by $z$. With the cone and supertrace signs just computed, the two mixed products have relative \v{C}ech representatives
\begin{align*}
(-1)^p\frac{\tr((\varphi\otimes K_S)\xi)}{z},\qquad \frac{(e\otimes K_S)v}{z}.
\end{align*}
The relative dualizing form supplies $dz$. The counit of relative Serre duality for the projective line bundle is normalized by $\Tr_{\ol{X}/S}[dz/z]=1$. Thus the diagram~\eqref{eq:trace-square} commutes.

The trace of a perfect endomorphism is evaluation after the symmetry map, and relative trace is the counit of $\dR f_*\dashv f^!$. For $\ol{X}\times T\xrightarrow{\ol{p}}S\times T\xrightarrow{p}T$, transitivity identifies $\Tr_{\ol{X}/T}$ with $\Tr_{S/T}\circ\dR p_*\Tr_{\ol{X}/S}$.

Under a change $z'=uz$ of normal trivialization, we have $d_{\ol{X}/S}z'/z'=d_{\ol{X}/S}z/z$, and the dual blocks transform inversely. Evaluation, the cone maps and the Serre counits are natural under quasi-isomorphisms of perfect resolutions. The triangle identities identify successive changes of resolution, so~\eqref{eq:trace-square} descends as an identity of maps of relative perfect complexes over the surface pair stack. For a derived base change $T'\to T$, proper perfect duality identifies its pull-back with the trace square for the pulled-back family. Its adjoint is the relative Serre equivalence $\cC\simeq\cD^\vee[-1]$ defining $\eta$.

To compare closed forms, use the completed de Rham complex of~\cite[Section~1.2]{Cal19}, graded by de Rham degree and normal scaling weight. On an equivariant formal differential graded presentation of the shifted cotangent stack, the base coordinates have scaling weight zero and the linear fiber coordinates weight one; choose a graded free resolution with these weights. Write $\mathsf{E}$ for the Euler derivation in the fiber coordinates, $\iota_{\mathsf{E}}$ for the contraction, and $\mathsf{d}=d_{\mathrm{int}}+d_{\mathrm{dR}}$ for the total differential, where $d_{\mathrm{int}}$ and $d_{\mathrm{dR}}$ are the internal and de Rham differentials. With $\mathcal{L}_{\mathsf{E}}$ denoting the Lie derivative and the usual totalization signs, the Cartan identity reads
\begin{align}\label{eq:Euler-Cartan}
[\mathsf{d},\iota_{\mathsf{E}}]=\mathcal{L}_{\mathsf{E}},\qquad [d_{\mathrm{int}},\iota_{\mathsf{E}}]=0,
\end{align}
since the internal differential preserves the scaling weight and the usual graded Cartan calculation applies to the free presentation. Euler contraction is intrinsic, so applying it to the trace comparison respects smooth descent and derived base change.

Because $\Gm$ is linearly reductive, the transported Calabi-Yau closed two-form has an equivariant representative $\Omega=\omega_0+\omega_1+\cdots$ of scaling weight one, where $\omega_i$ has de Rham degree $2+i$; thus $\mathsf{d}\Omega=0$ and $\mathcal{L}_{\mathsf{E}}\Omega=\Omega$. Set
\begin{align*}
\alpha:=\iota_{\mathsf{E}}\omega_0,\qquad \xi:=\sum_{i\ge 1}\iota_{\mathsf{E}}\omega_i.
\end{align*}
Since $d_{\mathrm{int}}\omega_0=0$ we have $d_{\mathrm{int}}\alpha=0$, and~\eqref{eq:Euler-Cartan} gives
\begin{align}\label{eq:closed-form-homotopy}
\Omega-d_{\mathrm{dR}}\alpha=\mathsf{d}\xi.
\end{align}
Observe that every term of $\xi$ has de Rham degree at least two, so~\eqref{eq:closed-form-homotopy} is a homotopy in the complex of closed two-forms.

The functions on the formal completion of the shifted cotangent stack along its zero section have no negative scaling weights, so a weight one two-form contains at most one fiber differential. Its component with one fiber differential has weight zero coefficients and is determined over the base by its restriction to the zero section, while the contraction by $\mathsf{E}$ kills the component with two base differentials. Hence $\alpha$ is the one-form in the base directions, linear in the fiber coordinate, determined by the mixed pairing along the zero section. The universal computation~\eqref{eq:relative-trace-pairing} identifies it under $\eta$ with the tautological one-form $\lambda_{\mathrm{can}}$; homotopies between the mixed perfect complex maps induce homotopies between these linear one-forms. By~\cite[Section~2.1]{Cal19}, the cotangent closed form is $d_{\mathrm{dR}}\lambda_{\mathrm{can}}$. Equation~\eqref{eq:closed-form-homotopy} identifies it with $\Omega$.

\end{proof}

\subsection{Comparison of orientations}
\begin{thm}\label{thm:oriented}
In every numerical class, Toda's classical isomorphism identifies the natural Calabi-Yau $d$-critical structure with the cotangent $d$-critical structure. On the PT open substack, for a stable pair $I=[\cO_X\to F]$ on $X$, the orientation line bundle
\begin{align*}
L_I=\det\Rhom_S(p_*F,p_*F)\otimes\bigl(\det\Rgam(X,F)\bigr)^{-1}
\end{align*}
is isomorphic to the canonical orientation line bundle of the shifted cotangent stack, compatibly with their square isomorphisms to the virtual canonical bundle. Hence the PT vanishing cycle mixed Hodge module used in this paper agrees with the Calabi-Yau vanishing cycle mixed Hodge module with orientation $L_I$. The identification commutes with automorphisms of $S$ and the Hilbert-Chow map.
\end{thm}
\begin{proof}
Transport the natural $d$-critical section through $t_0(\eta)$ and compare it with the cotangent section on $N:=t_0(T^*[-1]\bM^\dagger_S)$. Both have weight one for the same contracting action over $t_0\bM^\dagger_S$, and by Lemma~\ref{lem:cyclic} they agree formally at every point of the zero section.

On a smooth presentation, a $d$-critical section is represented in $\cO/J^2$, where $J$ is the chart ideal. Faithful flatness of completion gives agreement of germs and therefore agreement on a neighborhood of each zero section point. For any point $v\in N$, the orbit map $t\mapsto tv$ extends to $\AAA^1$. The inverse image of this neighborhood contains zero and hence some non-zero $t$, so weight one equivariance gives agreement at $v$. Smooth descent gives equality on $N$.

Along the zero section, the canonical orientation line bundle of the shifted cotangent stack is $(\det D)^{-1}$, and the surface pair triangle gives $(\det D)^{-1}=\det B\otimes(\det V)^{-1}$. The trace computation in Lemma~\ref{lem:cyclic} identifies the pairing of the two summands with graded Serre evaluation, so this isomorphism of line bundles respects their square isomorphisms to the virtual canonical bundle.

For the determinant calculation on the whole cone, use the universal diagram of Lemma~\ref{lem:derived-cotangent} and its fiber sequence $\cO_S\to U\to F\otimes K_S^{-1}$. The vanishing $\dR\ol{p}_*\cL^{-1}=0$ and the pairwise Hom complexes of the reconstruction give
\begin{align}\label{eq:whole-cone-determinant}
\begin{split}
\det\Rhom_{\ol{X}}(I,I)_0
&\cong\det\Rgam(S,\cO_S)^{-1}\otimes\det\Rhom_S(U,U)\otimes\det\Rhom_S(F,F)\\
&\qquad\otimes\bigl(\det\Rhom_S(U,F)\bigr)^{-1}\otimes\bigl(\det\Rhom_S(U,F\otimes K_S^{-1})\bigr)^{-1}\\
&\cong(\det B)^{\otimes 2}\otimes(\det V)^{\otimes(-2)}.
\end{split}
\end{align}
The scalar factor $\det\Rgam(S,\cO_S)$ has been removed by the trace. For the last line, apply the multiplicativity of determinants to $\cO_S\to U\to F\otimes K_S^{-1}$ and use
\begin{align*}
\Rhom_S(F\otimes K_S^{-1},\cO_S)\simeq V^\vee[-2],\qquad \Rhom_S(F\otimes K_S^{-1},F)\simeq B^\vee[-2].
\end{align*}
Determinant multiplicativity, with the graded symmetry isomorphisms, identifies the orientation line bundle with the pull-back of $\det B\otimes(\det V)^{-1}=\det\LL_{\bM^\dagger_S}$.

On the PT open substack, write $F$ for the sheaf on $X$ and $\theta\colon p_*F\to p_*F\otimes K_S$ for the action of the normal coordinate. The family resolution
\begin{align*}
0\to p^*p_*F\otimes p^*K_S^{-1}\xrightarrow{\,z-\theta\,}p^*p_*F\to F\to 0
\end{align*}
and the pair triangle give
\begin{align*}
\det\Rhom_X(I,I)_0\cong\Bigl(\det\Rhom_S(p_*F,p_*F)\otimes\bigl(\det\Rgam(X,F)\bigr)^{-1}\Bigr)^{\otimes 2},
\end{align*}
which is the restriction of~\eqref{eq:whole-cone-determinant} and supplies the determinant orientation used in Theorem~\ref{thm:support}. The isomorphisms from the squares of these line bundles to the virtual canonical bundle agree along the zero section. Their ratio is an invertible function of scaling weight zero on the cone $N$. Now $N$ is the total space of a linear cone over $t_0\bM^\dagger_S$, with strictly positive $\Gm$-weights on the fiber coordinates. Every weight zero function on $N$ is pulled back from the zero section, so the ratio is one everywhere.

The isomorphism of oriented $d$-critical stacks induces the stated isomorphism of vanishing cycle mixed Hodge modules. It commutes with automorphisms of $S$ because they preserve reconstruction, evaluation and trace; it commutes with the Hilbert-Chow map because reconstruction retains the surface sheaf.
\end{proof}

\section{Permutation actions in the nodal formula}\label{app:nodal-actions}\label{sub:nodal}
\subsection{Hilbert schemes of nodal curves} Work over a finite field $k$ with $\QQ_\ell$-cohomology, where $\ell\ne\mathrm{char}\,k$.
Let $C_0$ be a projective nodal curve over $k$, put $C=C_0\otimes_k\ol{k}$, and write $\wt C$ for its normalization. The dual graph $\mathsf G$ has one vertex for each component of $\wt C$ and one edge for each node of $C$; denote its edge set by $J$. Geometric Frobenius $\Fr$ acts on the graph and on the two branches of each node. All graph homology and cohomology groups below have coefficients in $\QQ_\ell$. Put $W=H^0(\wt C,\QQ_\ell)$, the permutation representation on the components.

For $I\subset J$, let $E_I$ be the vector space generated by the two orientations of every edge in $I$, subject to the relation that reversing an orientation changes its sign. Let $P_I=\QQ_\ell[I]$ be the vector space with basis the unoriented edges of $I$. Both spaces have dimension $|I|$ and carry the action of the subgroup of $\operatorname{Gal}(\ol{k}/k)$ preserving $I$. Define
\begin{align*}
\varepsilon_I:=\det E_I\otimes(\det P_I)^{-1}.
\end{align*}
This one dimensional representation records branch reversals; the determinant of $P_I$ cancels the sign from permuting the edges. Its character takes values in $\{\pm1\}$, so $\varepsilon_I^\vee\cong\varepsilon_I$. We write $(-)^*$ for the dual representation. All sums over subsets or partitions below mean orbit sums with their natural action: take the direct sum of the terms in an orbit, with Frobenius carrying each term to the next. Equivalently, induce from the stabilizer of one term. This convention also applies to virtual terms.

For a graph $\mathsf{H}$ obtained by deleting edges of $\mathsf{G}$, set
\begin{align*}
A_{\mathsf{H}}:=H^1(\mathsf{H})\oplus H_1(\mathsf{H})\LL,\qquad B_{\mathsf{H}}:=H^0(\mathsf{H})\oplus H_0(\mathsf{H})\LL.
\end{align*}
Here $\LL=[\QQ_\ell(-1)]$, and $\Lambda_{-q}(A)=\sum_a(-q)^a[\wedge^aA]$. Define
\begin{align*}
\cK_{\mathsf{H}}(q):=\sum_I(q\LL)^{|I|}\varepsilon_I\Lambda_{-q}(A_{\mathsf{H}\setminus I}),
\end{align*}
where the sum runs over edge subsets $I$ whose removal preserves the connected components of $\mathsf H$. Let $\cP(\mathsf{G})$ be the set of partitions of the vertex set into blocks whose induced subgraphs are connected. For $\tau\in\cP(\mathsf{G})$, let $J_\tau$ be the set of edges joining different blocks.

\begin{lem}\label{lem:equivariant-nodal}
In the Grothendieck ring of finite dimensional continuous $\QQ_\ell$-representations of $\operatorname{Gal}(\ol{k}/k)$, coefficientwise in $q$, the alternating cohomology series of the Hilbert schemes
\begin{align*}
Z_H(C,q):=\sum_{m\ge 0}q^m\sum_a(-1)^a[H^a(C^{[m]},\QQ_\ell)]
\end{align*}
satisfies
\begin{align}\label{eq:nodal-subsets}
Z_H(C,q)=\Lambda_{-q}(H^1(\wt{C}))\sum_{I\subset J}(q\LL)^{|I|}\varepsilon_I\frac{\Lambda_{-q}(A_{\mathsf{G}\setminus I})}{\Lambda_{-q}(B_{\mathsf{G}\setminus I})}.
\end{align}
Equivalently,
\begin{align}\label{eq:nodal-partitions}
Z_H(C,q)=\sum_{\tau\in\cP(\mathsf{G})}(q\LL)^{|J_\tau|}\varepsilon_{J_\tau}
\frac{\Lambda_{-q}(H^1(\wt{C}))\,\cK_{\mathsf{G}\setminus J_\tau}(q)}{\Lambda_{-q}(B_{\mathsf{G}\setminus J_\tau})}.
\end{align}
Frobenius may permute the components and exchange the branches in these identities.
\end{lem}
\begin{proof}
The factorization argument of~\cite[Section~3.2]{MSV21} reduces the calculation to the punctual Hilbert schemes at the nodes. Write $Q=|k|$ when taking traces. At a split node the punctual length $m$ Hilbert scheme is, for $m\ge 2$, a chain of $m-1$ projective lines with $m-2$ intersections; the length one Hilbert scheme is a point. Its series of point counts is
\begin{align*}
1+\sum_{m\ge 1}(1+(m-1)Q)t^m=\frac{1-t+Qt^2}{(1-t)^2}.
\end{align*}
Branch exchange reverses the chain. At a non-split node, the trace is $1$ for odd $m$ and $1+Q$ for positive even $m$, so its series is $(1+t+Qt^2)/(1-t^2)$. The corresponding series for the preimages in the normalization are $(1-t)^{-2}$ and $(1-t^2)^{-1}$ respectively. The ratio of the local Hilbert series is consequently $1-\sigma t+Qt^2$, where $\sigma=1$ for a split node and $\sigma=-1$ for a non-split node.

For an orbit $O\subset J$ of $r$ geometric nodes, let $\sigma$ be the sign obtained after traversing the orbit. The same argument over its residue field gives the ratio $1-\sigma q^r+Q^rq^{2r}$. On this orbit,
\begin{align*}
\det(1-q\Fr\mid E_O)=1-\sigma q^r,\qquad \Tr(\Fr\mid\varepsilon_O)=\sigma.
\end{align*}
Thus this ratio is also the trace of $\Lambda_{-q}(E^*_O\oplus E_O\LL)+(q\LL)^r\varepsilon_O$, since $(1-\sigma q^r)(1-\sigma Q^rq^r)+\sigma Q^rq^r=1-\sigma q^r+Q^rq^{2r}$. In an orbit sum, subsets not fixed by Frobenius have trace zero. Multiplying the local ratios and using the smooth Macdonald formula for $\wt{C}$ therefore gives
\begin{align}\label{eq:nodal-before-cancellation}
Z_H(C,q)=\frac{\Lambda_{-q}(H^1(\wt{C}))}{\Lambda_{-q}(W\oplus W\LL)}\sum_{I\subset J}(q\LL)^{|I|}\varepsilon_I\Lambda_{-q}(E^*_{J\setminus I}\oplus E_{J\setminus I}\LL).
\end{align}
Applying the calculation to every power of Frobenius proves the identity in the Grothendieck group of representations.

The graph homology and cohomology exact sequences of~\cite[(3.1), (3.2)]{MSV21} give, for each stabilizer of $I$,
\begin{align*}
[E^*_{J\setminus I}\oplus E_{J\setminus I}\LL]=[A_{\mathsf{G}\setminus I}]+[W\oplus W\LL]-[B_{\mathsf{G}\setminus I}].
\end{align*}
The vertex basis identifies $W$ with $W^*$. Applying $\Lambda_{-q}$ and cancelling the common factor in~\eqref{eq:nodal-before-cancellation} gives~\eqref{eq:nodal-subsets}.

Finally, group $I$ by the connected components of $\mathsf{G}\setminus I$. If this partition is $\tau$, then $I=J_\tau\sqcup I'$, and deleting $I'$ from $\mathsf G\setminus J_\tau$ preserves its connected components. Thus $B_{\mathsf{G}\setminus I}=B_{\mathsf{G}\setminus J_\tau}$ and $\varepsilon_I=\varepsilon_{J_\tau}\otimes\varepsilon_{I'}$. These identifications commute with stabilizers and with induction along the orbit of $\tau$. Regrouping gives~\eqref{eq:nodal-partitions}.
\end{proof}

\subsection{Stalks of intermediate extensions}
\begin{proof}[Proof of Lemma~\ref{lem:nodal-descent}]
Proposition~\ref{prop:hilb} and Lemma~\ref{lem:transv} give the smoothness and transverse node-smoothing hypotheses of~\cite[Proposition~3.3]{Zha26}. After ordering the components, the local formula~\cite[Theorem~5.10]{MSV21} identifies the strict-support local systems with the cohomology of the length $s=n-N(\lambda)$ Hilbert scheme of the normalization. For $s<0$ the summand is zero. Assume $s\ge 0$.

Let $C\in U_\lambda$, let $\mathsf G$ be its dual graph, and let $J$ be the edge set. Choose a transverse slice whose coordinates are the smoothing parameters of the nodes, and a nearby smooth fiber $C_t$. Let $L^{(j)}$ be the fiber of the iterated nearby-cycle local system, with underlying Betti vector space $\wedge^jH^1(C_t,\QQ)$. For each node $e$, let $N_e^{(j)}$ be its logarithm of unipotent monodromy, and put $N_I^{(j)}=\prod_{e\in I}N_e^{(j)}$. In the $\ell$-adic realization, these are morphisms $N_e^{(j)}\colon L^{(j)}\to L^{(j)}(-1)$ and $N_I^{(j)}\colon L^{(j)}\to L^{(j)}(-|I|)$. The Cattani-Kaplan-Schmid (CKS) complex computes the stalk of the intermediate extension across the coordinate divisors~\cite[Section~2.2, Proposition~2.1]{MSV21}. Its degree $p$ term is
\begin{align*}
\bigoplus_{I\subset J,\ |I|=p}\Img N^{(j)}_I\otimes\det P_I.
\end{align*}
Its differential uses the maps $N_e^{(j)}$ and exterior insertion. Thus a permutation sends $x\otimes(e_1\wedge\cdots\wedge e_p)$ to $gx\otimes(ge_1\wedge\cdots\wedge ge_p)$. On the associated weight graded pieces, the maps of~\cite[Lemmas~3.5, 3.6]{MSV21} identify a non-zero image with $\det E_I^*$ times the exterior powers for the partial normalization and the Tate factor $\LL^{|I|}$; the image is zero if the deletion disconnects a component. The exterior index contributes $(-q)^{|I|}$ and the complex degree contributes $(-1)^{|I|}$. Their product is $q^{|I|}$, while $\det E_I^*\otimes\det P_I\cong\varepsilon_I$. Consequently the alternating class of the CKS complex, with the exterior degree recorded by $q$, is $\Lambda_{-q}(H^1(\wt{C}))\,\cK_{\mathsf{G}}(q)$ with the actions specified in Lemma~\ref{lem:equivariant-nodal}.

For $\tau\in\cP(\mathsf G)$, let $C_\tau$ be the partial normalization at the nodes in $J_\tau$ and apply this calculation to $C_\tau$. The coefficients arising from $H^0(C_\tau)$ and its dual have finite component permutation monodromy and extend over the node-smoothing slice. Tensoring by them therefore commutes with intermediate extension. The alternating stalk class of the intermediate extension of the local systems in the smooth Macdonald formula is
\begin{align*}
\frac{\Lambda_{-q}(H^1(\wt{C}))\,\cK_{\mathsf{G}\setminus J_\tau}(q)}{\Lambda_{-q}(B_{\mathsf{G}\setminus J_\tau})}.
\end{align*}
Formula~\eqref{eq:nodal-partitions} now supplies the full stalk identity, with the factor $(q\LL)^{|J_\tau|}\varepsilon_{J_\tau}$ on each partition and with induction along every orbit of partitions.

Spread the nodal family, its finite covers, and these complexes to a finitely generated base and then to good finite residue fields, coefficientwise in $q$. Lemma~\ref{lem:equivariant-nodal} holds at every closed point of the base of unordered components and for every power of Frobenius. The trace comparison and Chebotarev argument of~\cite[Theorem~5.10, Corollary~6.4]{MSV21} therefore compare the semisimple local systems on this base. The local branches of the strict supports through $C$ correspond to the partitions in $\cP(\mathsf G)$. Starting with the larger supports, their CKS terms are the terms just computed; subtracting these successively leaves, on the stratum where the $N(\lambda)$ nodes join distinct smooth components,
\begin{align}\label{eq:normalization-local-system}
(q\LL)^{N(\lambda)}\varepsilon_J\sum_{s\ge 0}q^s\sum_k(-1)^k[R^k\wt{\rho}_{\lambda,s*}\QQ_\ell].
\end{align}
The action permutes the length allocations.

Retain the Tate factor $\QQ_\ell(-N(\lambda))$ contributed by these nodes in~\eqref{eq:normalization-local-system}. Proposition~\ref{prop:hilb} and projectivity make the perverse direct images over the reduced locus pure and semisimple, so they have their canonical strict support decomposition. In perverse degree $i$, the local system with strict support $\Sigma_\lambda$ has weight $D_\beta+n+i-\dim\Sigma_\lambda$. On the partial normalization side, the degree $k$ Hilbert cohomology has weight $k+2N(\lambda)$ after the indicated Tate twist. Since $\dim\Sigma_\lambda=D_\beta-N(\lambda)$ and $k=s+i=n-N(\lambda)+i$, these weights agree and distinguish the degrees after the support is fixed. Consequently the weighted class identity, with the larger strict supports removed successively, identifies the semisimple local systems in each degree on $U_\lambda$. In particular, the strict support local system in perverse degree $i$ is the descent of
\begin{align*}
R^{s+i}\wt{\rho}_{\lambda,s*}\QQ_\ell\otimes\varepsilon_J(-N(\lambda)).
\end{align*}
On the geometric base the Tate twist is a constant one dimensional local system. Forgetting this twist, Betti-\'etale comparison identifies the rational local systems after extending scalars to $\QQ_\ell$. Let $V$ and $W$ be the monodromy representations of the local system defining the summand with strict support $\Sigma_\lambda$ and of the descent of $R^{s+i}\wt{\rho}_{\lambda,s*}\QQ\otimes\varepsilon_J$, on a connected component of $U_\lambda$. Observe that the intertwiner space is cut out by linear equations over $\QQ$, with only finitely many independent equations. Its formation therefore commutes with scalar extension. The determinant is non-zero at the given $\QQ_\ell$-linear isomorphism and is therefore a non-zero polynomial on this rational vector space. Since $\QQ$ is infinite, it is non-zero at some rational point, giving an isomorphism over $\QQ$. The transposition calculation in the proof of Theorem~\ref{thm:reduced} identifies $\varepsilon_J$ with $\varepsilon_\lambda$. After forgetting the Tate factor, the resulting local system is $\cH^{s+i}_{\lambda,s}$.
\end{proof}

\end{document}